\documentclass[12pt]{amsart}
\usepackage{amsfonts, amssymb, amscd, latexsym, graphicx, psfrag, color, float}
\usepackage[all]{xy}

\usepackage{tikz}

\usepackage[hyphens]{url}

\usepackage{todonotes}
\usepackage{hyperref}

\usepackage[hyphenbreaks]{breakurl}

\usepackage{tkz-euclide}
\usetkzobj{all} 
\usepackage{caption}

\newtheorem{dummy}{dummy}[section]
\newtheorem{lemma}[dummy]{Lemma}
\newtheorem{theorem}[dummy]{Theorem}
\newtheorem{innercustomthm}{Theorem}
\newenvironment{customthm}[1]
  {\renewcommand\theinnercustomthm{#1}\innercustomthm}
  {\endinnercustomthm}

\newtheorem{corollary}[dummy]{Corollary}
\newtheorem{proposition}[dummy]{Proposition}
\theoremstyle{definition}
\newtheorem{definition}[dummy]{Definition}
\newtheorem*{definition*}{Definition}
\newtheorem{example}[dummy]{Example}

\newtheorem{remark}[dummy]{Remark}

\newtheorem{assumption}[dummy]{Assumption}
\newcommand{\bA}{\mathbb{A}}

\newcommand{\bG}{\mathbb{G}}

\newcommand{\bN}{\mathbb{N}}
\newcommand{\bP}{\mathbb{P}}
\newcommand{\bQ}{\mathbb{Q}}
\newcommand{\bR}{\mathbb{R}}
\newcommand{\bZ}{\mathbb{Z}}

\newcommand{\cA}{\mathcal{A}}

\newcommand{\cC}{\mathcal{C}}
\newcommand{\cD}{\mathcal{D}}

\newcommand{\cF}{\mathcal{F}}

\newcommand{\cI}{\mathcal{I}}
\newcommand{\cJ}{\mathcal{J}}
\newcommand{\cK}{\mathcal{K}}

\newcommand{\cM}{\mathcal{M}}
\newcommand{\cO}{\mathcal{O}}
\newcommand{\cP}{\mathcal{P}}

\newcommand{\cR}{\mathcal{R}}

\newcommand{\cT}{\mathcal{T}}
\newcommand{\cU}{\mathcal{U}}

\newcommand{\cX}{\mathcal{X}}
\newcommand{\cY}{\mathcal{Y}}
\newcommand{\cZ}{\mathcal{Z}}

\newcommand{\Spec}{\mathrm{Spec}\,}
\newcommand{\Proj}{\mathrm{Proj}}

\newcommand{\Hom}{\mathrm{Hom}}

\newcommand{\Perf}{\mathrm{Perf}}

\newcommand{\Sch}{\mathrm{Sch}}

\newcommand{\Div}{\mathrm{Div}}
\newcommand{\id}{\mathrm{id}}

\newcommand{\val}{\mathrm{val}}
\newcommand{\gp}{\mathrm{gp}}
\newcommand{\Qcoh}{\mathrm{Qcoh}}

\newcommand\radice[2][\relax]{\hspace{-1.5pt}\sqrt[\uproot{2}#1]{#2}}
\newcommand{\Par}{\mathrm{Par}}
\newcommand{\Coh}{\mathrm{Coh}}
\newcommand\Db[1]{\Perf(#1)}

\newcommand\irs[1]{\radice[\infty]{#1}}

\newcommand{\LRS}{\mathrm{LRS}}

\begin{document}

\author[Scherotzke]{Sarah Scherotzke}
\address{Sarah Scherotzke, 
Mathematical Institute of the University of M\"unster, 
Einsteinstrasse 62, 
48149 M\"unster,
Germany}
\email{\href{mailto:scherotz@math.uni-muenster.de}{scherotz@math.uni-muenster.de}}

\author[Sibilla]{Nicol\`o Sibilla}
\address{Nicol\`o Sibilla, Max Planck Institute for Mathematics, Bonn, Germany \\ 
and School of Mathematics, Statistics and Actuarial Sciences\\ 
University of Kent\\ 
Canterbury, Kent CT2 7NF\\
UK}
%\address{School of Mathematics, Statistics and Actuarial Sciences\\ 
%Cornwallis Building\\ 
%University of Kent\\ 
%Canterbury, Kent CT2 7NF\\
%UK}
\email{\href{mailto:N.Sibilla@kent.ac.uk}{N.Sibilla@kent.ac.uk}}

\author[Talpo]{Mattia Talpo}
\address{Mattia Talpo, Department of Mathematics\\
Simon Fraser University\\
8888 University Drive\\
Burnaby BC\\
V5A 1S6 Canada \\ and Pacific Institute for the Mathematical Sciences \\ 4176-2207 Main Mall \\ Vancouver BC V6T 1Z4 \\ Canada
}

\email{\href{mailto:mtalpo@sfu.ca}{mtalpo@sfu.ca}}

\title[Logarithmic derived McKay]{On a logarithmic version of the \\ derived McKay correspondence}



\begin{abstract}
We globalize the derived version of the McKay correspondence of Bridgeland-King-Reid, proven by Kawamata in the case of abelian quotient singularities, to certain log algebraic stacks with locally free log structure. The two sides of the correspondence are given respectively by the infinite root stack and by a certain version of the valuativization (the projective limit of every possible log blow-up). Our results imply, in particular, that in good cases the category of {coherent} parabolic sheaves with rational weights is invariant under log blow-up, up to Morita equivalence.
\end{abstract}

\maketitle

\tableofcontents

\section{Introduction}
\label{sec:outline}

In this paper we investigate a logarithmic version of the derived McKay correspondence.  We show that the categories of coherent sheaves on two geometric objects that are naturally associated with a log scheme,  its \emph{(reduced)  valuativization}    
and its \emph{infinite root stack}, are derived Morita equivalent.  {Furthermore, we prove that} when the logarithmic structure is induced by a flat family, this equivalence restricts to the fibers. The key technical input is an asymptotic form of the derived McKay correspondence for abelian quotient singularities.   
 %As an application, we prove that the {bounded} derived category of  {coherent} parabolic sheaves {with rational weights} 
% over a smooth variety equipped with a normal crossings divisor is invariant under log blow-ups. 
{As a main application we prove the surprising result that the category of parabolic sheaves satisfies a categorified form of excision. More precisely we show that, if $X$ is a  smooth variety equipped with a normal crossings divisor $D,$ the {bounded} derived category of  {coherent} parabolic sheaves {with rational weights}  over $(X,D)$ is invariant under log blow-ups.} 
In the remainder of this introduction we give a more detailed summary of our results, and discuss motivations and future 
 directions. 
 
{Throughout the paper, we work over an algebraically closed field $k$ of characteristic zero.}

\subsection{Logarithmic McKay correspondence} 
The origins of logarithmic geometry go 
back to  
work in arithmetic geometry by Deligne, Faltings, Fontaine-Illusie, Kato and others, in the 1980's. Although initially designed for arithmetic applications, in the last twenty years 
logarithmic geometry has become a key organizing principle in areas as diverse as moduli theory, deformation theory, mirror symmetry, and tropical and symplectic geometry. 
Log schemes are hybrids, they incorporate simultaneously algebro-geometric and combinatorial information. For this reason many aspects of their 
geometry are subtle. This has motivated 
the introduction of several auxiliary geometric objects, which are naturally associated with log schemes and capture some aspects of their geometry, but are more amenable to ordinary geometric techniques. For instance, Kato and Nakayama explained in \cite{KN} how to attach to a complex log scheme $X$ a topological space, nowadays called the \emph{Kato-Nakayama space}, which can be viewed as a sort of underlying topological space of $X$. 
In a different direction, Olsson 
has studied certain kinds of classifying stacks associated with log schemes, and used them in particular to define the cotangent complex in the logarithmic setting, see \cite{Ols, Ol}.

In this paper we will consider two other geometric objects which are associated in a natural way with a log scheme $X$: the \emph{infinite root stack} 
$\irs{X}$, and the \emph{valuativization}   $X^{\val}$.  
The infinite root stack  was introduced by Talpo and Vistoli in \cite{TV}. As shown in \cite{CSST} and in \cite{TaV, formalism}, $\irs{X}$ can be viewed as an algebro-geometric incarnation of the Kato-Nakayama space. Formally, it is defined as a limit of root 
constructions along the support of the log structure of $X$. 
The valuativization  
was defined by Kato in \cite{katov}, and was used by Kato, Pahnke \cite{pahnke}, Olsson \cite{olsson} and others to study compactifications of moduli problems. 
The space $X^{\val}$ can be defined as the limit of all log blow-ups of $X$.

Our main  result shows that the categories of coherent sheaves on $\irs{X}$ and on (a ``reduced'' variant of) 
$X^{\val}$ are derived Morita equivalent. 
This holds under some assumptions on $X$, which can be phrased as requiring that $X$ is the total space of a 
\emph{simple log semistable  
degeneration}. 
We explain next the construction of $\irs{X}$ and 
$X^{\val}$  in some more detail. We will then introduce simple log semistable  morphisms and state our main theorem.

Our results also apply to certain log algebraic stacks, but for simplicity we will stick to log schemes in this introduction.

\subsubsection*{\textbf{The infinite root stack}}
Let $X$ 
be a fine and saturated log scheme. The 
$n$-th root stack of $X$,  which is denoted by $\radice[n]{X}$, is an algebraic stack obtained  by extracting all the $n$-th roots of the divisors in the non-trivial locus of the log structure of $X$. Intuitively this construction adds a stabilizer $\mu_n^l$ to every point where the rank of the log structure   is exactly $l$.   

\begin{example}\label{example:root.stacks}
Let $X=\Spec A$ be an affine scheme with log structure induced by a smooth divisor $D$ with global equation $f=0$ for some $f\in A$.  Then $\radice[n]{X}$ is isomorphic to the quotient stack $$[\Spec ( A[x]/(x^n-f))\, /\, \mu_n],$$ where the $\mu_n$-action is via the obvious grading. If $X$ is a Riemann surface, then $D$ is given by a finite number of points $p_1,\hdots, p_k$, and $\radice[n]{X}$ can be seen as the orbifold obtained by replacing a small disk $\Delta$ centered in $p_i$ with the orbifold chart $[\Delta/\mu_n]$, where $\mu_n$ is acting on $\Delta$ by multiplication.

If $D$ is not smooth, we use a different and more general construction, developed by Borne and Vistoli (following ideas of Olsson) in \cite{borne-vistoli}, that ``separates the branches'' of $D$, and also applies to arbitrary log schemes.
%\textcolor{red}{
%If $D$ is not smooth, we will need to use a different construction that ``separates the branches'' of $D$. For example, if $D$ is union of two smooth components $D_1$ and $D_2$ with equations $f_i=0$ and intersecting transversely, then $\radice[n]{X}$ is the quotient stack $$[\Spec ( A[x_1,x_2]/(x_1^n-f_1, x_2^n-f_2))\, /\, \mu_n\times \mu_n],$$ where each copy on $\mu_n$ is acting on the corresponding variable. In the even more general case where components of $D$ have self-intersections or the log structure is more complicated, we need to rely on the general construction due to Borne and Vistoli (following ideas of Olsson) in \cite{borne-vistoli}.
%}
\end{example}
There is a natural map $\radice[n]{X}\to X$, which is a coarse moduli space morphism. 
Further, if  $n\mid m$ we have a natural projection $\radice[m]{X}\to \radice[n]{X}$. The infinite root stack is the limit 
$\irs{X}=\varprojlim_n \radice[n]{X},$ where the stabilizers, instead of being finite, are profinite groups.  
As in \cite{CSST} we will mostly think of $\irs{X}$ as a 
pro-object in algebraic stacks.

\subsubsection*{\textbf{The valuativization}}
Let $X$ be a log scheme as above. The valuativization 
$X^{\val} \to X$ is the limit of all \emph{log blow-ups} $X_\cI$ of $X$. Log blow-ups are the log geometric version of 
blow-ups, and they are indexed by coherent sheaves of ideals $\cI\subseteq M$ in the sheaf of monoids $M$. In some cases, the log blow-up along $\cI$ can be constructed by taking the ordinary blow-up of $X$ along the sheaf of ideals $\langle \alpha(\cI)\rangle\subseteq \cO_X$ generated by the image of $\cI$ via the map $\cI\to M\xrightarrow{\alpha} \cO_X$, together with a natural log structure. For example when the log structure of $X$ comes from a boundary divisor,  log blow-ups coincide with ordinary blow-ups along higher codimensional strata of the boundary. There is a partial order on coherent sheaves of ideals $\cI\subseteq M$, and the resulting index set is filtered. The valuativization of $X$ is the  limit $X^\val=\varprojlim_\cI X_\cI$.

If $X$ is a log scheme, one can take the inverse limit of the blow-ups $X_\cI$ in the category of locally ringed spaces. The resulting locally ringed space $X^\val$ is rarely a scheme, but one can still talk about some of its geometric properties (for example separatedness, compactness, etc.), and they turn out to be related to the geometry of the log scheme $X$ (see \cite{katov}). In this paper we will instead regard $X^\val$ as a ``formal'' limit, that is, 
as a pro-object in schemes. 
For technical reasons, it will be important to apply this  valuativization process to stacks as well. In that case, we will take the inverse limit in the 2-category of fibered categories over the category of $k$-schemes.

Since both $\irs{X}$ and $X^{\val}$ are pro-objects in stacks, we can formally define their categories of perfect complexes as  colimits of dg-categories
$$
\Db{\radice[\infty]{X}}:=\varinjlim_n \Db{\radice[n]{X}}, \quad   \Db{X^\val}:=\varinjlim_\cI \Db{X_\cI}. 
$$
We do not claim that this is a reasonable definition of the category of perfect complexes for an arbitrary pro-stack. However, in the case of $\irs{X}$ and $X^{\val}$,  this formula  can be justified based on our finer understanding of these colimits. We refer the reader to the main body of the paper for a  discussion of this point (see (\ref{sec:dg.categories})).

\subsubsection*{\textbf{The main theorem}}  
{Before stating our first main result (Theorem \ref{main}), let us clarify the assumptions we need to make on the log scheme $X.$ These are best expressed}  
%The assumptions we need to make on the log scheme $X$ are best expressed 
by saying that $X$ comes equipped with a \emph{simple log semistable} morphism of log schemes $f:X \rightarrow S$ for a log flat base $S$, see below for the definition. This viewpoint fits well  with the picture that motivated us, where $X$ is the total family of a log smooth degeneration. We  refer the reader  to the main body of the paper for a more  general setting of our result and for an intrinsic characterization of the properties of $X$ which makes no mention of the morphism $f$ (see Section \ref{sec:simple.log.ss}).
\begin{definition*}[Definition \ref{def:simple.log.ss.1}]
Let $X$ and $S$ be fine saturated log schemes over $k$.  
A morphism of log schemes  $f\colon X\to S$ is \emph{simple log semistable} if $f$ is log smooth and vertical, and 
for every geometric point $x\to X$ with image $s=f(x)\to S$ with non-trivial log structure, there are isomorphisms $\overline{M}_{S,s}\cong \bN$ and $\overline{M}_{X,x}\cong \bN^{r+1}$ with $r\geq 0$, such that the map $\overline{M}_{S,s}\to \overline{M}_{X,x}$ is identified with the $r$-fold diagonal map $\bN\to \bN^{r+1}$.
\end{definition*}

The $r$-fold diagonal map $\bN\to \bN^{r+1}$ corresponds to the toric morphism $\bA^{r+1}\to \bA^1$ given in coordinates by $(x_1,\hdots,x_{r+1})\mapsto x_1\cdots x_{r+1}$, which is a kind of ``universal'' local model for the  degenerations that we will be interested in.

From now on we will assume that $X$ comes equipped with a simple log semistable morphism of log schemes $f:X \rightarrow S$ over a base $S$ that is log flat over the base field $k$ (this is a mild technical assumption, see Section \ref{sec:simple.log.ss} for details). 
Morally, our main result shows that  
the categories of perfect complexes over the infinite root stack and the valuativization of $X$  are equivalent. 
In actuality, however, this is not quite the right statement. 
From a geometric point of view, the issue has to do with the fact that the fibers of both the modified families $\radice[\infty]{X}\to S$ and $X^\val\to S$ are not reduced. 
In order to fix this, we need to make the preliminary step of extracting every possible root of the log structure of the base $S$. If the base $S$ is equal to $\Spec R$ for a DVR $R$ with uniformizer $\pi$,  with the log structure given by the divisor $\{\pi=0\}$, forming $\radice[n]{S}$ is closely related to extracting an $n$-th root of $\pi$, as should be apparent from Example \ref{example:root.stacks}. This operation should be familiar, for example, from the theory of semistable reduction \cite{KKMS}.

Let us consider the infinite root stack $\radice[\infty]{S}$. Pulling back the family $f\colon X\to S$ along the natural map $\radice[\infty]{S}\to S$, we obtain a morphism $$X_\infty:=X\times_S \irs{S} \to \radice[\infty]{S}.$$ Note that $X_\infty$ is not a scheme anymore, as it picks  up the stackiness of the base $\irs{S}$.  We consider its infinite root stack 
$\radice[\infty]{X_\infty}$, 
and its valuativization  
$X_\infty^\val$. We  call $X_\infty^\val$ the \emph{reduced valuativization} of $X$. While 
$X_\infty^{\val}$ differs from the ordinary valuativization of $X$,   there is a canonical isomorphism of log stacks $\radice[\infty]{X_\infty}\simeq \radice[\infty]{X}$. 
{Our main result  {states} that these two objects are derived equivalent. This is a  {``global''} version of the  {derived} McKay correspondence for abelian quotient singularities. We refer the reader to \cite{kapranov-vasserot} and \cite{BKR} for background on the classical derived McKay correspondence.} 

\begin{customthm}{A}[Theorem \ref{thm:main}]
\label{main}
With the notations discussed above, there is an equivalence of dg-categories
$$
 \Db{X_{\infty}^{\val}}\simeq  \Db{\radice[\infty]{X}}.
$$
\end{customthm}

A global version of the McKay correspondence in the toroidal case was proved by Kawamata in \cite{Ka}. In a way our techniques yield a generalization of his results, as they apply to a wider variety of log schemes, but the emphasis in our theorems is different. Indeed, we are mostly interested in asymptotic statements where the size of the abelian groups that are acting 
is allowed to go to infinity. On the other hand,  
the key technical input in the proof of this result is 
Kawamata's proof of the abelian McKay correspondence for quotients of $\bA^n$ by finite subgroups of the torus.

{A surprising consequence of Theorem \ref{main} is that,  {if $(Y,D)$ is the log scheme given by a smooth variety $Y$ equipped with a normal crossings divisor $D$, then the category of perfect complexes over $\radice[\infty]{(Y,D)}$ is invariant under certain log blow-ups}. This has interesting applications to the theory of parabolic sheaves, and to the K-theory of {these kinds of} log schemes, which  we will discuss in the next section.}

Before proceeding, we clarify the nature of the objects appearing in our main result by working out explicitly the case of 
$X=\bA^2$ equipped with the toric log structure. 

\begin{example}\label{examp:a2}
Consider the log scheme $X=\bA^2$ equipped with the toric log structure, together with the map $\bA^2\to \bA^1$ given by $(x,y)\mapsto xy$, where the base $\bA^1$ also has the toric log structure. Thus, we are viewing $\bA^2$ as the total space of the log smooth degeneration of a $\bG_m$ to the union of the two coordinate axes. Note that this is the local model for a smoothing of a node in a curve.

The stack $X_\infty$ is  the inverse limit of the stacks $X_n:=\bA^2\times_{\bA^1} \radice[n]{\bA^1}.$ These can be described explicitly as follows. Consider the affine scheme $\overline{X_n}:=\Spec k[x,y,t_n]/(xy-t_n^n)$, which is the standard affine $A_{n-1}$ surface singularity. Then $X_n$ is the quotient stack of $\overline{X_n}$ by the natural action of $\mu_n$ (via the variable $t_n$). 

There are two ways to desingularize the stack $X_n$. 
The first way consists in considering the $n$-th root stack $\radice[n]{\bA^2}=[\Spec k[x_n,y_n]\, /\, \mu_n^2]$, where $x_n^n=x$ and $y_n^n=y$, and the action of $\mu_n^2$ is the natural one. There is an induced map 
$$
\xymatrix{
\radice[n]{\bA^2}
\ar[r] \ar@/^1pc/[rr] & X_n \ar[r] & \radice[n]{\bA^1}.} 
$$
Note that passing from $\bA^2\to \bA^1$ to $\radice[n]{\bA^2}\to \radice[n]{\bA^1}$ produces a local model for families of balanced twisted nodal curves, in the sense of \cite{abramovich-vistoli}. 
The other way to desingularize $X_n$ is to consider a blow-up: the $A_{n-1}$ singularity has a canonical minimal smooth resolution $C_n$, whose exceptional divisor is a chain of $(n-1)$ copies of $\bP^1$. Moreover the resolution is toric, and it inherits an action of $\mu_n$, whose quotient stack $\cC_n:=[C_n\,/\, \mu_n]$ maps to $X_n$. 

The two objects $\irs{X}$ and $X_\infty^\val$ in our main theorem are, in this case, the inverse limits of the systems $\{\radice[n]{\bA^2}\}_{n\in \bN}$ and $\{\cC_n\}_{n\in \bN}$ respectively, where $\bN$ is ordered by divisibility. A version of the derived McKay correspondence in the toric setting implies that $\cC_n$ and $\radice[n]{\bA^2}$ are derived equivalent, and passing to the limit we obtain the statement of the Theorem \ref{main} in this particular case (we refer to Section \ref{sec:local} for details and the general case of $\bA^{r+1}$). \end{example}

\subsection{{Categorified excision for parabolic sheaves}}%{Parabolic bundles and log blow-ups}
Parabolic bundles on curves were introduced by Mehta and Seshadri in the late 70's in order to formulate the Narasimhan-Seshadri correspondence in the case of a punctured Riemann surface. The current understanding of the theory is that, for a pair  $(Y,D)$ given by an algebraic variety $Y$ and a Cartier divisor $D$ (or more generally a fine saturated log scheme), 
 it is possible to define
an abelian category $\Par(Y, D)$ of parabolic coherent sheaves over  $(Y,D)$ with rational weights. Work of Biswas, Simpson, Borne, Vistoli, Talpo and others has clarified that
$\Par(Y, D)$ is equivalent to the abelian category of coherent sheaves over the infinite root stack 
$\irs{(Y,D)}$. The literature on these aspects is vast, but we refer the reader for instance to \cite{biswas1997parabolic}, \cite{borne2007fibres}, \cite{borne-vistoli}, \cite{TV}, \cite{talpo2014moduli} and \cite{realweights} for additional information. 

For technical convenience, we will assume that 
$Y$ is smooth, and that $D$ is given by a normal crossings divisor. 
Parabolic sheaves over a pair $(Y, D)$ model coherent sheaves on 
the complement $Y\setminus D$ that have prescribed asymptotic behavior close to the boundary (encoded by filtrations along the boundary components and real numbers, called ``weights''). 
This suggests that the category of parabolic sheaves should be in some way
independent of the actual shape of the boundary $D$, although this is far from clear.

Our main application is a  {first} concrete,  {and quite unexpected,}  articulation of this principle, where we show  that the category of parabolic sheaves is invariant under log blow-ups (up to Morita equivalence). {This should be viewed as a categorified form of excision for parabolic sheaves.} 
 We formulate the precise statement below.

\begin{customthm}{B}[{Proposition} \ref{cor:parabolic}]
\label{mainC}
Let $(Y,D)$ be a pair given by a smooth variety $Y$ over $k$ equipped with a normal crossings divisor $D$. Let $(Y',D')\to (Y,D)$ be a log blow-up, such that $(Y',D')$ is again a smooth variety with a normal crossings divisor. 
Then there is an equivalence 
$$
D^b(\Par(Y,D)) \simeq D^b(\Par(Y',D'))
$$
between the bounded derived categories of  coherent  parabolic sheaves with rational weights.\end{customthm}

{
Theorem \ref{mainC} is closely related to earlier results by other authors.  By results of Fujiwara and Kato (see \cite{ill} Theorem 6.2),  Kummer \'etale cohomology with finite coefficients   is 
invariant  under log 
blow-ups.  An analogous statement in the setting of log syntomic cohomology was proved by Nizio\l \mbox{} in \cite{niziol2008semistable}. Our result has immediate K-theoretic consequences. Namely, it implies that the algebraic K-theory of the Kummer-flat site of $(Y,D)$ is invariant under a large class of log blow-ups. We refer the reader to Corollary \ref{cor:kth} in the main text for more details on this.}

Abramovich and Wise have shown  in  \cite{abramovichi} that logarithmic Gromov-Witten invariants {do not change under} log blow-ups. 
It has been suggested that a parallel theory of logarithmic Donaldson-Thomas invariants should also exist, although many aspects are still to be understood.  
It is reasonable to  expect that, if such a theory existed,  parabolic sheaves would play in that  context a  role  analogous to ordinary coherent sheaves in classical Donaldson-Thomas theory. Then Theorem \ref{mainC} might perhaps {indicate} that logarithmic Donaldson-Thomas invariants should also be insensitive to log blow-ups, exactly as logarithmic Gromov-Witten invariants.

\subsection{Restriction to the central fiber}
Degeneration techniques 
have been ubiquitous in  algebraic 
geometry since the nineteenth century. Generally speaking, their purpose is to study complex geometric objects by breaking them down into simpler pieces. 
Formally this is achieved by 
creating singularities via flat deformations. 
In the last fifteen years 
these methods have been applied with great success to 
areas of 
enumerative geometry 
such as 
Gromov-Witten and Donaldson-Thomas theory, see for instance \cite{Li1, Li2, AF}. 
One important lesson to be learned from these results is that the 
central fiber of a deformation, by itself, is not sufficient to encode all the relevant geometric information. 
Additional data are needed, and they take the shape of various different \emph{enhancements} of 
the central fiber.

Most work in the literature is about the simplest case of a flat degeneration, where the central fiber breaks down as the union of two irreducible components $C_1$ and $C_2$ intersecting transversally. 
 We have two main {techniques of enhancement} there. 
The first involves working with a full system of 
\emph{expanded degenerations}, as in \cite{Li1, Li2}, which are modified central fibers with extra irreducible components wedged in between $C_1$ and 
$C_2.$ These can be viewed as the central fibers of modified deformations   
obtained by blowing-up  the original total family along the singular locus of the central fiber.
In the second  approach the central fiber is enhanced instead by introducing stackyness 
along the common divisor of $C_1$ and $C_2,$ as in \cite{AF}. These methods are not mutually exclusive, they can be combined together (see \cite[Section 2.2]{AF}).

A complementary point of view  on enhancements comes from the study of moduli problems on degenerate algebraic varieties, and in particular on nodal curves, which are the example  that  has been  studied most 
extensively in the literature. 
Here a similar geometry also emerges. Due to the presence of singularities 
moduli problems are typically non-compact, but modular compactifications can be achieved  through enhancements: points on the boundary of the compactification 
parametrize objects that  
do not live 
over the degenerate algebraic variety itself, 
but rather on some  expanded or stacky modifications  of it. 
This point of view can be traced back to work of Gieseker on moduli of vector bundles on nodal curves \cite{gie}. In order to construct modular compactifications, Gieseker  looked at parametrizations of vector bundles on a system of expanded degenerations, where extra copies of 
$\bP^1$ are wedged in between the two branches of the node. More recent appearances of this perspective, in the context of moduli of principal bundles on nodal curves, can be found for instance in work of Martens and Thaddeus \cite{matha2} and Solis \cite{solis}, which make use of  a mix of expanded and stacky techniques.

%%%%%%%%%%%%%%%%%%%%%%%%%

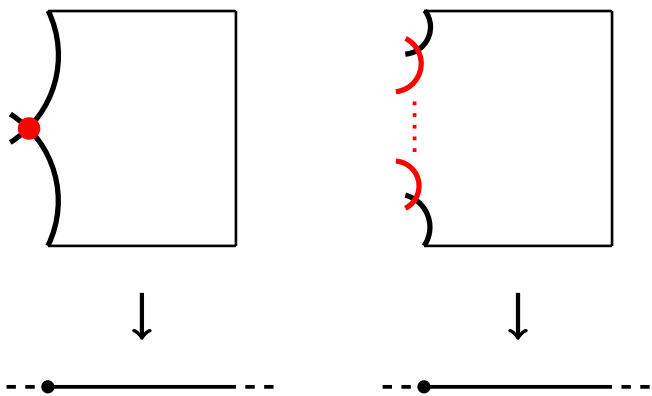
\begin{figure}[h]
\begin{center}
\begin{tikzpicture}[scale=2.5]

\clip (-.35,-0.5) rectangle (3.2, 1.78);

\draw [line width=.35mm] (0,.5) -- (1,.5);

\draw [line width=.35mm] (0,1.75) -- (1,1.75);

\draw [line width=.35mm] (1,.5) -- (1,1.75);

\draw [line width=.5mm] (0,-.25) -- (1,-.25);

\draw [line width=.5mm]  (1.05,-.25) -- (1.1,-.25);

\draw [line width=.5mm]  (1.15,-.25) -- (1.2,-.25);

\draw [line width=.5mm]  (-.12,-.25) -- (-.07,-.25);

\draw [line width=.5mm]  (-.22,-.25) -- (-.17,-.25);

\draw [line width=.5mm] (2,-.25) -- (3,-.25);

\draw [line width=.5mm]  (3.05,-.25) -- (3.1,-.25);

\draw [line width=.5mm]  (3.15,-.25) -- (3.2,-.25);

\draw [line width=.5mm]  (2-.12,-.25) -- (2-.07,-.25);

\draw [line width=.5mm]  (2-.22,-.25) -- (2-.17,-.25);

\draw  [line width=.35mm] (2,.5) -- (3,.5);

\draw [line width=.35mm] (2,1.75) -- (3,1.75);

\draw [line width=.35mm] (3,.5) -- (3,1.75);

\tkzDefPoint(0,.5){A}\tkzDefPoint(.05,0.8125){B}\tkzDefPoint(-.2,1.2){C}
\tkzCircumCenter(A,B,C)\tkzGetPoint{O}
\tkzDrawArc[color=black, line width=.7mm](O,A)(C)

\tkzDefPoint(0,1.75){D}\tkzDefPoint(.05,1.4375){E}\tkzDefPoint(-.2,1.05){F}
\tkzCircumCenter(D,E,F)\tkzGetPoint{P}
\tkzDrawArc[color=black, line width=.7mm](P,F)(D)

\draw[ultra thick, ->] (.5, .25) -- (.5, 0);

\draw[ultra thick, ->] (2.5, .25) -- (2.5, 0);

\tkzDefPoint(2,1.75){A}\tkzDefPoint(2.03,1.67){B}\tkzDefPoint(2-.1,1.52){C}
\tkzCircumCenter(A,B,C)\tkzGetPoint{O}
\tkzDrawArc[color=black, line width=.7mm](O,C)(A)

\tkzDefPoint(2-.1,1.6){A}\tkzDefPoint(2-0.3,1.4){B}\tkzDefPoint(2-.15,1.32){C}
\tkzCircumCenter(A,B,C)\tkzGetPoint{O}
\tkzDrawArc[color=red, line width=.7mm](O,C)(A)

\tkzDefPoint(2-.1,.77){A}\tkzDefPoint(2.03,.6){B}\tkzDefPoint(2,.5){C}
\tkzCircumCenter(A,B,C)\tkzGetPoint{O}
\tkzDrawArc[color=black, line width=.7mm](O,C)(A)

\tkzDefPoint(2-.15,.95){A}\tkzDefPoint(2-.03,.8){B}\tkzDefPoint(2-.1,.7){C}
\tkzCircumCenter(A,B,C)\tkzGetPoint{O}
\tkzDrawArc[color=red, line width=.7mm](O,C)(A)

\fill (0,-.25) circle (1pt);
\fill (2,-.25) circle (1pt);

\fill (-.1, 1.125)[color=red] circle (1.7pt);

\draw[color=red,  dash pattern=on \pgflinewidth off 3pt, line width=.5mm] (2-.05, 1) -- (2-.05, 1.275);

\end{tikzpicture}
\captionsetup{labelformat=empty}
\caption{The picture shows  two  kinds of enhancements of the same degenerate central fiber: the red dot on the left stands for $[\mathrm{Spec\,} k/\mu_n]$; on the right we have instead $n-1$ extra copies of $\mathbb{P}^1_k$ wedged between the two branches of the node,  also in red.}
\end{center}
\end{figure}

%%%%%%%%%%%%%%%%%%%%%%%%%%%%%%%%%%%%%%%%%%%%%%%%%%%%%%

The  enhancements of the central fiber encode  additional information coming from   nearby fibers. Informally speaking, this has the effect of partially smoothening  out the family. 
The modern perspective is that this can be  accomplished by equipping the central fiber with its  natural log structure and, in the context of Gromov-Witten theory, this point of view has been pursued for instance in  \cite{C, AC, GS}.  
However, as we pointed out,  it is often useful  to work with 
 more geometric incarnations of 
logarithmic objects. We propose that the central fibers of the infinite root stack and of the reduced valuativization can serve this purpose, and   encompass many  previous constructions of enhanced central fibers considered in the literature.  

More precisely, let $f:X \to S$ be a flat degeneration, where
$S= \Spec R$ with $R$ a DVR. 
We claim that the   fibers of $X_\infty^{\val}$ and $\irs{X}$ over the ``closed point''
$0 \in \irs{S}$ (i.e. the unique lift of the closed point of $S$ to the infinite root stack) model  respectively the expanded and the stacky central fibers of the family $f.$ For illustrative purposes, let us focus on the case where the central fiber of $f$ is the union of two irreducible components $C_1$ and $C_2$ glued along a common divisor $D.$ 
Then the fiber $(\irs{X})_0$ is a formal limit 
of root constructions along $D $, and thus captures    
the stacky enhancements of $(X)_0$ considered   in \cite{AF} and elsewhere. On the other hand  $(X_\infty^{\val})_0$ is a pro-object consisting of all  reduced central fibers of the modifications of $f$ given by blow-ups of $X$ supported on $D$. This 
recreates the shape of the expanded central fibers appearing in \cite{gie} and in \cite{Li1}.

Our third main result shows that, from the 
viewpoint of their sheaf theories, these two kinds of enhancements  contain the same amount of information.

\begin{customthm}{C}[Theorem \ref{thm:main1}]
\label{mainB}
Assume that $f\colon X\to S$ is a simple log semistable degeneration, where $S$ is the spectrum of a DVR, and let $0\in \irs{S}$ be 
the closed point. Then the equivalence $\Phi\colon\Db{X_{\infty}^{\val}}\to  \Db{\radice[\infty]{X}} $ restricts to an equivalence $$\Phi_0 \colon  \Db{(X_\infty^\val)_0}\to \Db{(\radice[\infty]{X})_0}.$$
\end{customthm}
 Theorem \ref{mainB} can be viewed as a global version of  earlier  results  on specializations of the derived McKay correspondence to the fibers (as in \cite{BP}), with the usual proviso that  contrary to the existing literature   we are mostly interested   in asymptotic statements. We remark  that our proof depends in a substantial way on these earlier references. 

The purpose of our result is twofold. On the one hand, it shows that two familiar kinds of enhancements of degenerate central fibers  are in fact closely related, as they give rise to equivalent sheaf theories.  This implies  for instance  that, in compactifications of moduli of vector bundles \` a la Gieseker,
perfect complexes 
on either of the two enhancements can be used.

On the other hand, we propose  that 
the (central fibers of) the reduced valuativization and  the  infinite root stack offer  a viable generalization of the existing 
enhancement packages considered in the literature. In particular, we believe that they capture the 
shapes of the enhancements required to handle degeneration patterns 
that are more complicated than the ones most studied in earlier references,  
where the central fiber breaks down as the union of two 
irreducible components intersecting transversely. 
We plan to return to applications of this point of view in future work. 

\begin{remark}
We point out that, in concrete geometric applications (e.g. in \cite{gie} and \cite{Li1}), the relevant geometric objects 
are rather \emph{finite} enhancements of the central fibers, where  the number of additional irreducible components and the size of the isotropy groups are kept finite. As we explained, infinite root stacks and valuativizations are much bigger objects made up of the inverse systems of all finite enhancements and the maps between them. 
\end{remark}

\begin{example}\label{example:elliptic}
Theorem \ref{mainB} can also help 
to fit  in a broader  context existing  observations scattered in the literature. Let us focus on one concrete example, which served as one of our initial motivations for this project. Let $X_0$ be a nodal rational curve of genus one, and  
let 
$\cX_n := [X_0/\mu_n]$ be the stacky quotient of $X_0$ by the action of $\mu_n \subseteq \bG_m.$ Consider   
 the Cartier dual of $\mu_n$, denoted $\bZ_n$, and let 
$E_n$ be the $\bZ_n$-Galois cover of $X_0$. The curve $E_n$ is sometimes called the N\' eron $n$-gon, and consists of a cycle of 
$n$ rational curves intersecting transversely. These objects fit into a  pair  of Galois covers: 
$
(i) \quad X_0 \rightarrow \cX_n \, ,  \quad   (ii)  \quad E_n \rightarrow X_0 \, .
$

As explained  in \cite{Si}, $\cX_n$ and 
$E_n$ have the same homological mirror, which is an $n$-punctured symplectic torus in both cases. Thus, as a consequence of  homological mirror symmetry, we have an equivalence 
$$
\Db{\cX_n} \simeq \Db{E_n}.    
$$ 
A different perspective on 
this equivalence comes from Mukai's  picture of 
derived equivalences of dual abelian varieties. The curve $X_0$ is isomorphic to its own compactified Jacobian, and this yields a non-trivial self-equivalence   of 
$
\Db{X_0}  
$  
which was studied for instance in \cite{burban2005fourier} and  in \cite{Si2}. 
From this perspective, the Galois covers $(i)$ and 
$(ii)$ 
can be interpreted as dual isogenies: 
$
 E_n \rightarrow X_0,  \quad  X_0   \sim  (X_0)^{\vee} \rightarrow \cX_n \sim (E_n)^\vee.  
 $ 
 This indicates that  $E_n$ play the role of 
 the (compactified) Jacobian of $\cX_n$, an observation which was also made by Niles (see Question 7.3 in \cite{niles2013moduli}). Mukai's theory then also suggests  that $E_n$ and $\cX_n$ should be derived equivalent.

Theorem  \ref{mainB} offers a new angle on
this phenomenon, which appears as a manifestation of a much more general comparison result. 
Indeed, let  $f: X \rightarrow \bA^1$ be a flat degeneration of elliptic curves, where the central fiber $X_0$ is a nodal rational curve. Then  Theorem \ref{mainB} applied to $f$ recovers precisely  the equivalence $\Db{\cX_n} \simeq\Db{E_n}$ (or, more accurately, an asymptotic version of it). 
\end{example}

\subsection{Future directions} 
Our results open up many avenues for future research,  
and we  conclude this  introduction by briefly outlining some of them. 
 We expect that Theorem \ref{main} will be useful in clarifying the relationship between several different compactifications and enhancements of moduli spaces appearing in the literature. An important example is provided by the Deligne--Mumford 
moduli stack of genus $g$ stable curves 
$\overline{\cM_g}$. The stack $\overline{\cM_g}$ has a natural log structure given by the boundary divisor. In \cite{chiodo2015n} Chiodo proves that it is possible to 
construct a universal compactified Jacobian over $\overline{\cM_g}$, on condition of enhancing $\overline{\cM_g}$ 
via root constructions along the boundary divisor. 
Holmes achieves a similar result in \cite{holmes2014eron} by looking instead at  towers of log blow-ups of $\overline{\cM_g},$   
which are closely related to the valuativization of $\overline{\cM_g}$. The precise relationship between these two approaches remains elusive at the moment, but it is tempting to guess that Holmes's and Chiodo's constructions might be related by a derived equivalence patterned after our Theorem \ref{main}. Making this idea precise is the subject of work in progress by two of the authors together with Holmes and Smeets.

It would be interesting to establish sharper  versions of our Theorem \ref{mainC}. At the moment we work under quite restrictive assumptions on the log scheme $X$. We believe that these assumptions could be substantially relaxed. Additionally, it would be desirable to extend Theorem \ref{mainC} to categories of parabolic sheaves with real (as opposed to rational) weights. 
As a preliminary step, this requires developing a  new geometric approach to parabolic sheaves with real weights, since root stack techniques are only available 
when weights are rational. We believe that this is of independent interest, and this is also the subject of work in progress. 

\subsection*{Acknowledgements}

We are grateful to Tom Bridgeland for pointing us to the papers \cite{Ka} and \cite{BP}, and to Jonathan Wise for bringing Kato's construction of valuativizations of log schemes of \cite{katov} to our attention in the first place. We are also happy to thank Kai Behrend, David Ben-Zvi, John Calabrese, Sabin Cautis, David Holmes, Luc Illusie, Nathan Ilten, Pavel Safronov, Arne Smeets, Martin Ulirsch, and  Angelo Vistoli for useful exchanges. Finally, we are grateful to the referee for several useful comments and suggestions. The second author wishes to thank the University of Oxford and Wadham College for providing excellent working conditions while part of this project was carried out. 

\subsection*{Notations and conventions}
Every scheme or algebraic stack will be quasi-compact and quasi-separated, and defined over a base field $k$ which is algebraically closed, of characteristic zero, equipped with the trivial log structure.  All fibered categories will be fibered in groupoids. We will state and use several universal mapping properties in a $2$-category: when we write ``unique'' in that context, we will implicitly mean ``in the $2$-categorical sense''. All pullback and pushforward functors between derived categories will be implicitly derived. All monoids will be commutative, and we will typically write the operation additively. We will also assume monoids to be \emph{integral} or \emph{cancellative}, i.e. if $p+q=r+q$, then $p=r$.

%%%%%%%%%%%%%%%%%%%%%%%%%%%%%%%%%%%%%%%%%%%

\section{Preliminaries}

\subsection{Logarithmic schemes}

We give a minimal reminder of the basic definitions regarding log schemes and root stacks. See \cite{kato,abramovich,borne-vistoli,TV} or the Appendix of \cite{CSST} for more details.

A \emph{log scheme} (``log'' is short for ``logarithmic'') is a scheme $X$ together with a sheaf of monoids $M$ on the small \'etale site, and a homomorphism of sheaves of monoids $\alpha\colon M\to \cO_X$ {that restricts} to an isomorphism $\alpha^{-1}\cO_X^\times\to \cO_X^\times$, where $\cO_X$ is equipped with the multiplication. This definition also makes sense if $X$ is only a locally ringed space, using the classical site of $X$ instead of the small \'etale site. %If we want to include the log scheme $X$ in the notation for $\alpha$ and $M$, we will write $\alpha_X$ and $M_X$.
With mild assumptions, a log structure in the above sense can also be seen as a symmetric monoidal functor $L\colon \overline{M}:=M/\cO^\times\to \Div_X$ with trivial kernel (i.e. the only section $a$ such that $L(a)$ is invertible is zero), where $\Div_X$ is the symmetric monoidal stack on $X$ of pairs $(L,s)$ consisting of a line bundle with a global section. For details about this reformulation we refer to \cite[Section 3]{borne-vistoli}.

A \emph{morphism} of log schemes $(X,M,\alpha)\to (Y,N,\beta)$ is a morphism of schemes $f\colon X\to Y$, together with a homomorphism of sheaves of monoids $f^{-1}N\to M$ compatible with $f^{-1}\cO_Y\to \cO_X$. A \emph{strict morphism} is a morphism as above such that the pullback of the log structure of $Y$ is isomorphic to the log structure of $X$ (via the map $f^{-1}N\to M$).

\begin{example}
Let us consider two kinds of log schemes that are especially important in applications. First of all, if  $P$ is a monoid, the scheme $\Spec k[P]$ has a natural log structure induced by the homomorphism $P\to k[P]=\Gamma(\cO_{\Spec k[P]})$. Second, if $X$ is any variety over $k$ equipped with a normal crossings divisor $D\subseteq X$, we obtain a log scheme $(X,D)$ by considering the subsheaf $M\subseteq \cO_X$ of regular functions that  vanish only possibly on $D$. Picking local equations for the branches of $D$ passing through a point $x\in X$ gives local charts for the log structure, with monoid $\bN^n$, where $n$ is the number of branches. In this case the log structure encodes the fact that we see $D$ as a sort of ``boundary'' in $X$.
\end{example}

A \emph{chart} for a log scheme $X$ is a homomorphism of monoids $P\to \cO_X(X)$ that induces the homomorphism $\alpha\colon M\to \cO_X$ in the following sense: from the map $P\to \cO_X(X)$ one obtains a homomorphism of sheaves of monoids $\widetilde{\alpha}\colon \underline{P}\to \cO_X$ where $\underline{P}$ is the constant sheaf with stalks $P$. The induced homomorphism $\underline{P}\oplus_{\widetilde{\alpha}^{-1}\cO_X^\times}\cO_X^\times\to \cO_X$ is a log structure on $X$, that we require to be isomorphic to $\alpha\colon M\to \cO_X$ (here $\oplus$ denotes the pushout in the category of sheaves of monoids). Equivalently, a chart is a strict morphism of log schemes $X\to \Spec k[P]$.

A log scheme $X$ is \emph{quasi-coherent} if  locally it has charts $X\to \Spec k[P]$. It is \emph{coherent} (resp. \emph{integral}, resp. \emph{fine}, resp. \emph{fine and saturated}, resp. \emph{locally free}) if   locally it has charts as before, where the monoid $P$ is finitely generated (resp. integral, resp. finitely generated and integral, resp. finitely generated, integral and saturated, resp. free). If a log scheme is quasi-coherent, the other properties can be checked on the stalks of $\overline{M}$ on geometric points of $X$. A log scheme $X$ is \emph{Zariski} if the log structure comes from the Zariski topos (see \cite[Section 2.1]{niziol}). 

{In the sequel we will often consider also algebraic stacks equipped with a log structure. One can define a log structure on an algebraic stack $X$ exactly as for schemes, but by using the lisse-\'etale site instead. By descent for fine log structures, this turns out to be the same as a log structure on a presentation, together with descent data.} {All the log structures on schemes and stacks in this paper will be integral and quasi-coherent (and often fine and saturated).}  All fibered products of fine saturated log schemes and stacks will be taken in the fine saturated category.

\subsection{{Root stacks}}
The $n$-th \emph{root stack} $\radice[n]{X}$ of a fine saturated log scheme (or algebraic stack) $X$ is the stack over schemes$/k$ that universally parametrizes extensions $\frac{1}{n}\overline{M}\to \Div_X$ of the symmetric monoidal functor $L\colon \overline{M}\to \Div_X$. In characteristic $0$ it is a Deligne--Mumford stack, and the projection $\radice[n]{X}\to X$ is a coarse moduli space morphism. This construction adds to a point $x\to X$ a stabilizer $\mu_n^r$, where $r$ is the rank of the free abelian group $\overline{M}_x^\gp$.

Quasi-coherent sheaves on the root stack $\radice[n]{X}$ correspond to quasi-coherent parabolic sheaves on $X$ with weights in $\frac{1}{n}\overline{M}^\gp$ \cite[Theorem 6.1]{borne-vistoli}. If $n\mid m$ we have a natural map $\radice[m]{X}\to \radice[n]{X}$, and with these maps root stacks form an inverse system of Deligne--Mumford stacks, with index system the set $\bN$ partially ordered by divisibility. If the log structure of $X$ is locally free, these maps are all flat.

The inverse limit $\irs{X}=:\varprojlim_n \radice[n]{X}$ is called the \emph{infinite root stack} of $X$. It is a stack over schemes$/k$ with an fpqc presentation, but it is not algebraic. Nonetheless, there is a good theory of quasi-coherent sheaves on it, and they correspond to parabolic sheaves with arbitrary rational weights \cite[Theorem 7.3]{TV}. The infinite root stack can be thought of as a purely algebro-geometric incarnation of a log structure itself: the association $X\mapsto \irs{X}$ gives a conservative faithful functor from fine saturated log schemes$/k$ to stacks$/k$ \cite[Section 5]{TV}.

The formation of root stacks (finite or infinite) is compatible with strict base change.

\subsection{Log blow-ups and valuativizations}\label{sec:valuativization}

The \emph{valuativization} of a log scheme first appeared in the unpublished \cite{katov} by K. Kato (see also \cite{illusie}), and it was further studied and used by Olsson \cite{olsson} and Pahnke \cite{pahnke}. Nizio\l \mbox{} \cite{Ni1} uses a related construction of valuative sites for a log scheme, where one allows log blow-ups to be coverings (see Remark \ref{rmk:niziol}).

\begin{definition}
A monoid $P$ is said to be \emph{valuative} if for every $a\in P^\gp$, either $a\in P$ or $-a\in P$ (recall that all monoids here are integral, so $P\to P^\gp$ is injective).

A log locally ringed space (or log algebraic stack) $X$ is said to be {valuative} if all the stalks of the sheaf $\overline{M}$ are valuative monoids.
\end{definition}

The valuativization $X^\val\to X$ of a log locally ringed space $X$ (if it exists) is a valuative log locally ringed space equipped with a morphism to $X$, such that if $Y$ is any valuative log locally ringed space, every log morphism $Y\to X$ factors uniquely as $Y\to X^\val\to X$ (so $(-)^\val$ would be the right adjoint of the inclusion functor from valuative log locally ringed spaces into log locally ringed spaces).

For a quasi-coherent log scheme $X$, the valuativization exists, and can be constructed as the inverse limit $\varprojlim_\cI X_\cI$ of all log blow-ups of $X$, taken in the category of log locally ringed spaces. The result is rarely a scheme, but it still has some geometry that one can study, and relate to the properties of the log scheme $X$. This construction is a logarithmic variant of Riemann--Zariski spaces. 

In the next sections we will review the definition and basic properties of log blow-ups and valuativizations. 

\subsubsection{Log blow-ups}

Let us briefly recall what log blow-ups are, and how they are constructed (more details can be found for example in \cite[Section 4]{niziol}, and references therein). Recall that an \emph{ideal} of a monoid $P$ is a subset $I\subseteq P$ such that $P+I\subseteq I$. Let us consider a subsheaf $\cI\subseteq M$ of ideals of the monoid $M$. We will refer to this as a ``sheaf of ideals of the log structure''.

\begin{definition}[{\cite[Definition 3.1]{niziol}}]
We will say that $\cI$ is \emph{coherent} if locally for the \'etale topology of $X$ there is a chart $P\to M(X)$ for the log structure and a finitely generated ideal $I\subseteq P$ of the monoid $P$, such that $\cI$ coincides with the subsheaf of $M$ generated by the image of $I$.

We will say that $\cI$ is \emph{invertible} if it can be locally generated by a single element, or equivalently if it is induced locally by a principal ideal of a chart $P\to M(X)$.
\end{definition}

If $\cI$ is coherent, then it is invertible if and only if $\cI_x \subseteq M_x$ is a principal ideal for every (geometric) point of $X$.
 
Given a coherent sheaf of ideals $\cI$ of the log structure on $X$, the log blow-up $(X_\cI,M_\cI)$ of $X$ along $\cI$ is a log scheme with a map $f_\cI\colon X_\cI \to X$ with the following universal property: the ideal generated by $f_\cI^{-1}\cI$ in $M_{\cI}$ is invertible, and every morphism of log schemes (or log locally ringed spaces) $f\colon Y\to X$ such that the ideal generated by $f^{-1}\cI$ on $Y$ is invertible factors uniquely as $Y\to X_\cI\to X$. Locally on $X$, where there is a chart $P\to M(X)$ and a finitely generated ideal $I\subseteq P$ inducing $\cI$, we can construct the log blow-up as the pullback along $X\to \Spec k[P]$ of the blow-up $\Proj (\bigoplus_n \langle I\rangle ^n)\to \Spec k[P]$ {along  the ideal $\langle I\rangle \subseteq k[P]$ generated by $I$}.

\begin{remark}
There are two different versions of log blow-ups, one in the category of fine log schemes, and the other one in the category of fine saturated log schemes. The two versions differ by a ``normalization'' step (a ``saturation'' at the level of monoids), which is performed in order to land in the fine saturated category (see \cite[Section 4]{niziol} for details). 

{It does not matter whether we use one or the other version to construct valuativizations, because a valuative monoid is necessarily saturated. In the main body of the paper we will mostly make use of the ``saturated'' variant, since we will be describing log blow-ups via subdivisions of fans (and these are all saturated, as toric varieties described via fans are all normal).}
\end{remark}

There is a filtered partial ordering on coherent sheaves of ideals of $M$ on $X$: we will write $\cJ\geq \cI$ if there exists a coherent sheaf of ideals $\cK\subseteq M$ such that $\cJ=\cK + \cI$ as sheaves of ideals of $M$. Recall that if $I$ and $J$ are ideals of $P$, their sum $I+J\subseteq P$ is the ideal consisting of elements of the form $i+j$ with $i\in I$ and $j\in J$ (in the more usual setting of rings, this corresponds to the product of two ideals). By the universal property, whenever $\cJ\geq \cI$ the map $X_\cJ\to X$ factors uniquely as $X_\cJ\to X_\cI\to X$, so we get an inverse system in the category of log schemes $\{X_\cI\}_\cI$ indexed by coherent sheaves of ideals $\cI$ of the log structure. 
{For later reference, we note that the formation of log blow-ups is compatible with strict base change (where we blow-up the pullback sheaf of ideals).}

\begin{remark}\label{rmk:toric.blow.up}
If $X$ is a (normal) toric variety, that we see as a log scheme via the natural log structure, then we also have ``toric blow-ups'', that are given by subdivisions of the fan $\Sigma$ of $X$. It turns out that every (saturated) log blow-up of $X$ is of this form: by \cite[Proposition 4.3]{niziol}, the log blow-up along a coherent sheaf of ideals $\cI\subseteq M$ is the same as the normalization of the blow-up of $X$ along the coherent sheaf of ideals (in the usual sense, this time) $\langle \alpha(\cI)\rangle\subseteq \cO_X$ generated by $\cI$ in $\cO_X$, via the map $\alpha\colon M\to \cO_X$. If $X=\Spec k[P]$ is affine, the global sections of this image will be a homogeneous ideal of the monoid algebra $k[P]$, and the blow-up along this ideal can be described via a subdivision of the corresponding cone. The subdivisions on the affine pieces will be compatible on the overlaps, and this gives a subdivision of $\Sigma$ realizing the log blow-up.

In particular, toric blow-ups are cofinal among log blow-ups of a normal toric variety.
\end{remark}

\subsubsection{Valuativization of quasi-coherent log schemes} Let us describe how to construct the valuativization $X^\val$ of a quasi-coherent log scheme: consider the inverse limit $\varprojlim_\cI X_\cI$ of all log blow-ups, as a locally ringed space. For a coherent sheaf of ideals $\cI\subseteq M$, denote by $\pi_\cI\colon X^\val\to X_\cI$ the projection. By pulling back the log structures of $X_\cI$, on $X^\val$ we have a filtered direct system of sheaves of monoids, and we set $M^\val:=\varinjlim_\cI \pi_\cI^{-1}M_{\cI}$. This sheaf has a natural induced map to $\cO_{X^\val}=\varinjlim_\cI \pi_{\cI}^{-1}\cO_{X_\cI}$, that gives a log structure on the locally ringed space $X^\val$. The projection $X^\val \to X$ (and all the maps $\pi_\cI$) are naturally log morphisms, and in fact $X^\val$ is the limit of diagram $\{X_\cI\}_{\cI}$ in the category of log locally ringed spaces.

\begin{remark}
Note that, since every fine saturated log scheme $X$ admits a log blow-up $X_\cI$ with Zariski log structure \cite[Theorem 5.4]{niziol}, in the above construction we can assume that the log structure of $X$ itself (and thus of every log blow-up, cfr. \cite[Proposition 4.5]{niziol}) is Zariski.
\end{remark}

The proof that the inverse limit of the log blow-ups is a valuativization is a consequence of the following lemma.

\begin{lemma}[{\cite[Lemma 7.1.2]{olsson}}]\label{lem:valuative}
Let $P$ be an integral monoid. Then $P$ is valuative if and only if every finitely generated ideal $I\subseteq P$ is principal. \qed
\end{lemma}

\begin{proposition}[{\cite[Theorem 1.3.1]{katov}}]\label{prop:valuativization}
Let $X$ be a quasi-coherent log scheme. Then the inverse limit of all log blow-ups along coherent sheaves of ideals $X^\val=\varprojlim_\cI X_\cI$ as described above is a valuativization for $X$.
\end{proposition}

\begin{proof}[Sketch of proof]
One shows that $X^\val$ is valuative by checking that for every point $p\in X^\val$ the stalk $M^\val_p=\varinjlim_\cI (M_\cI)_{\pi_\cI(p)}$ is a valuative monoid. In order to do this, by the previous lemma it suffices to check that every  finitely generated ideal $J\subseteq M^\val_p$ is principal.  This follows from the following reasoning. The generators of $J$ come from some log blow-up $X_\cI$ of $X$, denote by $J_\cI\subseteq (M_\cI)_{\pi_\cI(p)}$ the ideal that they generate (so that the ideal generated by the image of $J_\cI$ in $M^\val_p$ will be exactly $J$). We can then find a coherent sheaf of ideals $\cI'$ on $X$ dominating $\cI$, and whose image in $(M_\cI)_{\pi_\cI(p)}$ dominates $J_\cI$. Hence, on the log blow-up $X_{\cI'}$ (that dominates $X_\cI$) we can conclude that the ideal generated by the image of $J_\cI \subseteq (M_\cI)_{\pi_\cI(p)}$ is principal. Therefore, at the limit, $J\subseteq M^\val_p$ will also be principal. 

It is also easy to check that if $Y$ is a valuative log locally ringed space and $f\colon Y\to X$ is a log morphism, then this factors uniquely as $Y\to X^\val\to X$ (again by making use of the previous lemma, and the universal property of log blow-ups).
\end{proof}

\begin{example}\label{example:valuativization}
Since any monoid of rank $0$ or $1$ is valuative, any log scheme where the stalks of $\overline{M}$ have rank $\leq 1$ is valuative. In higher rank the situation is drastically different, as integral finitely generated monoids of rank at least $2$ are never valuative. Because of this, even the valuativization of $\bA^2$ with its natural log structure is quite complicated.
\end{example}

\subsection{{Valuativization of log stacks}} In this paper {it will be important to consider also the} valuativization of some kinds of log stacks. {To this purpose we will presently explain a variant of the process that we recalled above, that applies to stacks as well.} For a log stack $X$, in case log blow-ups make sense (for example if $X$ is algebraic, as we explain below), we will regard the inverse limit $\varprojlim_\cI X_\cI$, in a rather formal way, just as a fibered category with a log structure. 

\begin{definition}\label{def:log.str.fibered}
Let $\cC$ be a fibered category over schemes$/k$. A log structure $(\xi_{(S,\phi)}, \psi_{(a,\beta)})$ on $\cC$ is given by:
\begin{itemize}
\item a log structure $\xi_{(S,\phi)}=(M_{(S,\phi)}, \alpha_{(S,\phi)})$ on a scheme $S$, for every pair $(S,\phi)$ where $\phi$ is an object $\phi\in \cC(S)$ (i.e. a map $\phi\colon S\to \cC$),
\item an isomorphism $\psi_{(a,\beta)}\colon a^*\xi_{(S',\phi')}\to \xi_{(S,\phi)}$ for every $2$-commutative diagram
$$
\xymatrix@=.55cm{
S\ar[dd]_a \ar[rr]^\phi & &  \cC\\
&   \ar@{}[u]!<-6ex,0ex>^(.10){}="a"^(.80){}="b" \ar@{=>}^\beta "a";"b" &\\
S'\ar[rruu]_{\phi'} & &
}
$$

\item such that the obvious compatibility with respect to compositions $(S,\phi)\to (S',\phi')\to (S'',\phi'')$ is satisfied.
\end{itemize}
We will say that a log structure on $\cC$ is \emph{quasi-coherent, integral, finitely generated, fine, saturated, valuative} if the corresponding property is satisfied for every one of the log structures $\xi_{(S,\phi)}$.
\end{definition}

\begin{remark}
{The category of fine log structures on a fibered category $\cC$ can also be described as the category of functors $\cC\to LOG_k$ (commuting with the projections to the category of schemes$/k$) to Olsson's classifying stack of fine log structures over $\Spec k$}. Note that since $LOG_k$ is a category fibered in groupoids, every such functor is automatically cartesian.
\end{remark}

We will call a fibered category with a log structure a \emph{log fibered category}. One can also define morphisms of log structures in the obvious way, as well as morphisms of log fibered categories (as a functor of fibered categories over schemes$/k$, plus compatible morphisms of log structures for every scheme mapping to the source fibered category). If $\cC$ is a scheme or an algebraic stack, this recovers the usual definition of log structure, and the usual properties. In those cases one can restrict the attention to the appropriate small sites (the small \'etale site for a scheme, the lisse-\'etale site for an algebraic stack), where everything becomes more manageable.

\begin{example}\label{rmk:irs.log.str}
Let $X$ be a fine saturated log scheme over $k$. There is a tautological quasi-coherent log structure on the infinite root stack $\irs{X}$: for a scheme $S$, a morphism $S\to \irs{X}$ corresponds by definition to a Deligne--Falings structure  $\Lambda\colon \overline{M}_\bQ \to \Div_S$, that in turn gives a log structure in the sense of Kato $\alpha_\infty\colon M_\infty\to \cO_S$ (cfr. \cite[Theorem 3.6]{borne-vistoli}). This log structure can also be obtained as a direct limit of the pullbacks of the log structures $M_n$ of the finite root stacks $\radice[n]{X}$ along the projections $\irs{X}\to \radice[n]{X}$.
\end{example}

We can define the valuativization of any log fibered category $\cC$, by the universal mapping property with respect to morphisms from valuative log fibered categories.

\begin{definition}
Let $\cC$ be a log fibered category. The valuativization $\cC^\val$ of $\cC$ is a valuative log fibered category with a morphism of log fibered categories $\cC^\val\to \cC$, such that for every valuative log fibered category $\cD$ mapping to $\cC$ there exists a unique factorization $\cD\to \cC^\val\to \cC$.
\end{definition}

{Note that in the definition above, and in the rest of this section, we are using ``unique'' to mean ``unique in the $2$-categorical sense''.} 
Of course  it is not clear why such a thing should exist in full generality.  We will construct it when $\cC$ is a fine log algebraic stack, or the infinite root stack of a fine and saturated log algebraic stack.

\begin{remark}\label{rmk:val.base.change}
It is easy to see that if $\cC^\val\to \cC$ is a valuativization of $\cC$, and $\cD\to \cC$ is a strict morphism of log fibered categories (i.e. for a map $S\to \cD$, the induced log structure on $S$ is isomorphic to the one given by the composite $S\to \cD\to \cC$), then the projection $\cD\times_\cC \cC^\val\to \cD$ is a valuativization for $\cD$.
\end{remark}

\begin{proposition}
{For a log algebraic stack $X$, there exists a valuativization $X^\val\to X$, obtained as limit of log blow ups of $X$.}
\end{proposition}
\begin{proof}
{First, note that if $\cI$ is a coherent sheaf of ideals of the log structure on $X$, then there is a log blow-up $X_\cI\to X$. This can be constructed by taking a groupoid presentation $R\rightrightarrows U$ for $X$, and by defining $X_\cI$ as the stack associated to the groupoid $R_\cI\rightrightarrows U_\cI$ (note that these maps are still smooth, as they are a base change of the  structure maps $s,t\colon R\to U$, because these maps are strict). The stack $X_\cI$ has a natural log structure, and the resulting map $X_\cI\to X$ has the universal property of log blow-ups, as for log schemes.}

Hence we can consider the inverse limit $X^\val=\varprojlim_\cI X_\cI$ as a fibered category (here $\cI$ runs through all coherent ideal sheaves of the log structure on $X$). The log structure on $X^\val$ is described as follows: if $\phi\colon S\to X^\val$ is a morphism, it corresponds by construction to a family $(\phi_\cI,\eta_{\cJ\geq \cI})$ of morphisms $\phi_\cI\colon S\to X_\cI$ for every coherent sheaf of ideals $\cI$, and compatible natural equivalences $\eta_{\cJ\geq \cI}$ between the composites $S\xrightarrow{\phi_\cJ}X_\cJ\to X_\cI$ and $S\xrightarrow{\phi_\cI}X_\cI$ every time $\cJ\geq \cI$. By pulling back the log structure $(M_\cI,\alpha_\cI)$ of $X_\cI$ to $S$, we obtain an inductive system of log structures on $S$. We define $\xi_{(S,\phi)}$ as the limit $\varinjlim_\cI \phi_\cI^{-1}M_{\cI}$, with the induced map to $\cO_S$.
Routine verifications show that there are also natural isomorphisms $\psi_{(a,\beta)}$ as in the definition above, so that we get a log structure on $X^\val$. 
{The fact that this log fibered category is a valuativization is proved in the same way as Proposition \ref{prop:valuativization}.
}
\end{proof}

Let us check that if $X$ is a log scheme, this version of the valuativization is compatible with the locally ringed space of Proposition \ref{prop:valuativization}.

\begin{proposition}
Let $X$ be a quasi-coherent log scheme, and consider the valuativization ${X}^\val$ as a (log) locally ringed space. We can consider the induced fibered category $F_{{X}^\val}$ over schemes, by defining the category $F_{{X}^\val}(S)$ to be the set of morphisms $S\to {X}^\val$ as locally ringed spaces.

If we temporarily denote by $\varprojlim_\cI X_\cI$ the valuativization as a fibered category, then there is a canonical isomorphism $$\varprojlim_\cI X_\cI \simeq F_{{X}^\val}$$ of log fibered categories over schemes$/k$. 
\end{proposition}

\begin{proof}
This is clear from $\Hom_{\LRS}(S,X^\val)=\Hom_{\LRS}(S,\varprojlim_\cI X_\cI)=\varprojlim_{\cI} \Hom_{\Sch}(S,X_\cI)$, where $\LRS$ denotes the category of locally ringed spaces. It is also clear that the log structures will coincide, since they are defined in the exact same way for the inverse limit of fibered categories and for the locally ringed space.
\end{proof}

\begin{remark}
One can ask if anything is lost by considering the valuativization formally as an inverse limit of functors, rather than as a locally ringed space as in \cite{katov} and Proposition \ref{prop:valuativization} above. More generally, one can associate a functor over schemes to an arbitrary locally ringed space, in the same way. Despite the fact that it is unlikely that this process will be full and/or faithful in general, it is still possible that for a quasi-coherent log scheme $X$, the locally ringed space $X^\val$ could be reconstructed from the functor that it represents on schemes.

We will make no actual use of the geometry of the locally ringed space $X^\val$ in this paper, so we leave the question open to further investigation.
\end{remark}

\subsubsection{Valuativization of an infinite root stack} As mentioned earlier, we will also consider valuativizations of infinite root stacks of fine saturated log algebraic stacks. These infinite root stacks are not algebraic, so a little more care is required in this case. What allows everything to work out fine is the fact that they are ``pro-algebraic''.

Let $X$ be a fine saturated quasi-compact log scheme (or log algebraic stack) over $k$, and consider the infinite root stack $\irs{X}$. This has a tautological log structure, as explained in Example \ref{rmk:irs.log.str}. Recall that the infinite root stack is an inverse limit $\irs{X}=\varprojlim_n \radice[n]{X}$ of the log algebraic stacks $\radice[n]{X}$.

{
There are at least three ways of constructing the valuativization of $\irs{X}$: we could
\begin{itemize}
\item construct log blow-ups along coherent sheaves of ideals of $\irs{X}$ and then pass to the limit with respect to the ideals, or
\item consider the valuativizations $\radice[n]{X}^\val$ (this makes sense, since $\radice[n]{X}$ is a log algebraic stack), and then pass to the limit with respect to the index $n$, or
\item mix the two approaches, and consider the inverse system of log blow-ups $\radice[n]{X}_\cI$ of the finite root stacks along coherent sheaves of ideals of their log structure.
\end{itemize}
All three constructions can be carried through, and will yield the same log fibered category for $\irs{X}^\val$.} {In the following 
we will use the third approach, since it is best suited for our purposes.}

Consider for every $n\in \bN$ the set $I_n$ of coherent ideal sheaves $\cI_n^\alpha$ of the log structure on $\radice[n]{X}$, and form the set $I=\{(n,\cI_n^\alpha)\}_{n\in \bN, \cI_n^\alpha\in I_n}$. This set has a filtered partial ordering, by declaring that $(m,\cI_m^\beta)\geq (n,\cI_n^\alpha)$ if $n\mid m$ and, if $(\cI_n^\alpha)^m$ denotes the sheaf of ideals on $\radice[m]{X}$ generated by the pullback of $\cI_n^\alpha$, then we have $\cI_m^\beta \geq (\cI_n^\alpha)^m$ as sheaves of ideals on $\radice[m]{X}$.

For every pair of elements of $I$ such that $(m,\cI_m^\beta)\geq (n,\cI_n^\alpha)$, by the universal property of log blow-ups we have a map $\radice[m]{X}_{\cI_m^\beta}\to \radice[n]{X}_{\cI_n^\alpha}$, and we can consider the inverse limit $\varprojlim_{(n,\cI_n^\alpha) \in I}\radice[n]{X}_{\cI_n^\alpha}$.

\begin{proposition}\label{prop:diagonal.valuativization}
The inverse limit $\varprojlim_{(n,\cI_n^\alpha) \in I}\radice[n]{X}_{\cI_n^\alpha}$ is a valuativization for $\irs{X}$.
\end{proposition}

\begin{proof}
This proof is also along the lines of the proof of Proposition \ref{prop:valuativization}.

Let us show that the log structure of the inverse limit is valuative: given a scheme $S$ with a map $S\to \varprojlim_{(n,\cI_n^\alpha) \in I}\radice[n]{X}_{\cI_n^\alpha}$ and a geometric point $s\to S$, we want to show that the monoid $M_{S,s}=\varinjlim_{(n,\cI_n^\alpha)} ((M_n)_{\cI_n^\alpha})_{s_{n,\alpha}}$ is valuative (where we denote by $s_{n,\alpha}$ the image of $s$ in $\radice[n]{X}_{\cI_n^\alpha}$). 

%Let us first note that by localizing on $X$ around the image of $s$, we can assume that the log structure has a global chart $X\to \Spec k[P]$ with $P$ fine, saturated and sharp, sending $s$ to the homogeneous maximal ideal in $k[P]$. Moreover, since every object is pulled back from $\Spec k[P]$, we can assume that $X=\Spec k[P]$ altogether.

Let us consider a finitely generated ideal $J\subseteq M_{S,s}$. Its finitely many generators all come from $((M_n)_{\cI_n^\alpha})_{s_{n,\alpha}}$ for some $(n,\cI_n^\alpha) \in I$, let us denote by $J_{(n,\cI_n^\alpha)}$ the ideal they generate in the monoid $((M_n)_{\cI_n^\alpha})_{s_{n,\alpha}}$. Then we can find a sheaf of ideals $\cI_n^\beta$ of the log structure of $\radice[n]{X}$ dominating $\cI_n^\alpha$, and whose image in $((M_n)_{\cI_n^\alpha})_{s_{n,\alpha}}$ dominates the ideal $J_{(n,\cI_n^\alpha)}$. Consider the log blow-up $\radice[n]{X}_{\cI_n^\beta}$. By construction, the image of $J_{(n,\cI_n^\alpha)}$ in the monoid $((M_n)_{\cI_n^\beta})_{s_{n,\beta}}$ is a principal ideal.
%With our assumptions, there will be a corresponding coherent sheaf of ideals $\cJ$ of the log structure of the stack $\radice[n]{X}_{\cI_n^\alpha}$, and we can consider its log blow-up $(\radice[n]{X}_{\cI_n^\alpha})_\cJ$.
Since the map from the inverse limit factors through the projection $\radice[n]{X}_{\cI_n^\beta}\to \radice[n]{X}_{\cI_n^\alpha}$, we can conclude that $J$ is also principal. By Lemma \ref{lem:valuative}, it follows that the log structure of $S$ is valuative.

It is also easy to check that if $\cC$ is a valuative log fibered category mapping to $\irs{X}$, then we have a unique factorization $\cC\to \irs{X}^\val\to \irs{X}$, by using Lemma \ref{lem:valuative}.
\end{proof}

\begin{remark}\label{rmk:cofinal.subsystem}
Let us also explicitly note that for every cofinal subsystem $J\subseteq I$ we will also have $\varprojlim_{(n,\cI_n^\alpha) \in J}\radice[n]{X}_{\cI_n^\alpha}\simeq \irs{X}^\val$. 
\end{remark}

\begin{remark}\label{rmk:niziol}
In \cite[Section 2.2]{Ni1}, Nizio{\l} considers valuative variants of the Kummer-flat and Kummer-\'etale sites (among others), by allowing as covers also log blow-ups. This valuativization procedure should be a ``geometric counterpart'' to these valuative sites. 

More precisely, the valuativization $X^\val$ of $X$ itself corresponds, loosely speaking, to the valuative Zariski site (denoted by $X_\val$ in \cite{Ni1}), and the valuativization of the infinite root stack $\irs{X}^\val$ should correspond to the valuative Kummer-\'etale site (which coincides with the ``full'' log \'etale site). In fact, one can define a small \'etale site for the infinite root stack, and the associated topos turns out to be equivalent to the Kummer-\'etale topos (see \cite[Section 6.2]{TV}). Introducing log blow-ups into the picture should give the full \'etale site instead.

Thus, the valuativization of the infinite root stack $\irs{X}^\val$ can be thought of as a geometric incarnation of the full log \'etale site of a fine saturated log scheme $X$.
\end{remark}

\subsection{Toric stacks}\label{sec:toric.stacks}

We recall the basic definitions of stacky fans and toric stacks from \cite{gsa}. Recall that if $N$ is a lattice and $\Sigma$ is a rational polyhedral fan in the vector space $N_\bR=N\otimes_\bZ \bR$, there is an associated normal toric variety $X_\Sigma$.

\begin{definition}
A \emph{stacky fan} is the datum of a rational polyhedral fan $\Sigma$ in a vector space $N_\bR$ for a lattice $N$, and a homomorphism of lattices $f\colon N\to N'$, with finite cokernel.
\end{definition}

To a stacky fan one can associate a toric stack as follows: the homomorphism of lattices $f$ gives a homomorphism of tori $D(f^\vee)\colon D(N^\vee)\to D((N')^\vee)$, where we are denoting by $N^\vee$ the lattice $\Hom(N,\bZ)$, and by $D(-)$ the Cartier dual $\Hom(-,\bG_m)$. %Because $f$ has finite cokernel, the map $f^\vee$ is injective, and $D(f^\vee)$ is surjective.
Set $G:=\ker D(f^\vee)$.%, and denote by $X_\Sigma$ the toric variety given by the fan $\Sigma$.
Then, since $G$ is a subgroup of the torus $D(N^\vee)$ of the toric variety $X_\Sigma$, there is a natural action of $G$ on $X_\Sigma$.

\begin{definition}
The \emph{toric stack} associated with the stacky fan $(\Sigma, f\colon N\to N')$ is the quotient stack $[X_\Sigma/G]$.
\end{definition}

Toric stacks and stacky fans will be used in the main text to describe root stacks and log blow-ups of some local models for our constructions. We end this subsection with a few examples that can be useful to keep in mind.

\begin{example}
Fix a positive integer $n$. Consider the action of $\mu_n$ on $\bA^2$ given by $\xi\cdot (x,y)=(\xi x,\xi^{-1}y)$. The quotient stack $\cX=[\bA^2/\mu_n]$ is a smooth Deligne--Mumford stack, and its coarse moduli space is the $A_{n-1}$ singularity $X=\bA^2/\mu_n=\Spec k[x,y,t_n]/(xy-t_n^n)$.

The quotient stack $\cX$ is a toric stack: the lattice $N$ here is the lattice $\bZ^2$, the fan $\Sigma$ is the first quadrant in $\bR^2$, and $N'$ is the lattice given by points of the form $\frac{1}{n}(a,b)$ in $\bR^2$, with $a+b\equiv 0\; (\mathrm{mod} \; n)$.
The induced morphism of tori $\bG_m^2\cong D(N^\vee)\to D((N')^\vee)\cong \bG_m^2$ can be written (by appropriately choosing the isomorphisms) as $(s,t)\mapsto (st^{-(n-1)},t^n)$, and its kernel can be identified with $\mu_n\subseteq \bG_m$ embedded anti-diagonally $\xi\mapsto (\xi,\xi^{-1})$. The toric variety $X_\Sigma$ is $\bA^2$, and the result of the procedure outlined above is precisely the quotient stack $[\bA^2/\mu_n]$.

Note also that the coarse moduli space $\bA^2/\mu_n$ is the toric variety given by the same fan $\Sigma$, but with respect to the bigger lattice $N'$.
\end{example}

\begin{example}\label{ex:toric.root}
Consider $\bA^2$ as a toric variety, with its natural log structure. Its fan $\Sigma$ is the first quadrant in the vector space $\bR^2$ associated with the lattice $\bZ^2$. All its root stacks $\radice[n]{\bA^2}$ are toric stacks (and this holds in broader generality for toric varieties). In order to write a stacky fan for $\radice[n]{\bA^2}$ it suffices to consider the same fan $\Sigma$ of $\bA^2$, but as a fan in the lattice $n\bZ^2$, together with the inclusion of lattices $N=n\bZ^2\subseteq \bZ^2=N'$. The corresponding toric variety $X_\Sigma$ is again $\bA^2$ (mapping to the original $\bA^2$ via $(x,y)\mapsto (x^n,y^n)$), and it has an action of the kernel $\mu_n^2$ of the induced map of tori $\bG_m^2\to \bG_m^2$ given by $(s,t)\mapsto (s^n,t^n)$. The resulting quotient is $[\bA^2/\mu_n^2]$, which is indeed naturally identified with $\radice[n]{\bA^2}$. Taking its coarse moduli space, i.e. by considering $\Sigma$ as a fan in the bigger lattice $N'$, we get back the original $\bA^2$.

Also, note that subdivisions of $\Sigma$ give both toric blow-ups of the original $\bA^2$, and log blow-ups of the root stack $\radice[n]{X}$. Indeed, a subdivision of $\Sigma$ as a fan with lattice $N$ gives a toric blow-up $X\to X_\Sigma=\bA^2$, and the action of $\mu_n^2$ lifts to $X$. The quotient stack $[X/\mu_n^2]\to[\bA^2/\mu_n^2]$ gives a log blow-up of the stack $\radice[n]{\bA^2}$.
\end{example}

\subsection{{Dg-categories and categories of sheaves}} 
\label{sec:dg.categories}

Throughout the paper a \emph{dg-category} will be a $k$-linear differential graded category. We refer the reader   to \cite{keller2007differential} and \cite{drinfeld2004dg} for additional information on 
dg-categories and   their basic properties. All limits and colimits of diagrams of dg-categories are to be understood as homotopy limits  or colimits in the category of dg-categories, equipped with its Morita model structure of \cite{Tab} and \cite{toen2007homotopy}. Most dg-categories appearing in this paper will be \emph{triangulated}, see \cite[Section 4.4]{bertrand2011lectures} for a discussion of this notion. 
Let $\cC$ be a dg-category and let $A$ and $B$ be objects of 
$\cC$. We denote $\Hom_\cC(A, B)$ the complex of morphisms between $A$ and $B.$
We will also make use of the concept of a 
t-structure of a triangulated dg-category. By definition, a triangulated dg-category $\cC$ has a t-structure if its homotopy category $Ho(\cC)$ has a t-structure in the ordinary sense. A reference 
for the theory of t-structures of stable $(\infty,1)$-categories is  \cite[Section 1.2.1]{Lu2}. The discussion there applies without variations to the case of triangulated dg-categories. Given an abelian category $A$, we denote the canonical dg enhancement of its bounded derived category by $D^b_{dg}(A)$.

In the richer homotopical context of dg-categories,  quasi-coherent sheaves can be defined in complete generality for arbitrary (higher and derived) stacks, and they enjoy good descent properties. A discussion of these aspects can be found in the introduction of \cite{BFN}. 
If $X$ is a stack, we  denote by  $\Qcoh(X)$ the 
triangulated dg-category of quasi-coherent sheaves on $X$, and by 
$\Db{X}$ its full subcategory of perfect 
complexes. The tensor product of quasi-coherent sheaves equips  $\Qcoh(X)$ with a symmetric monoidal structure, and $\Db{X}$ can be defined as the full subcategory of dualizable objects. Passing to the homotopy categories 
we recover the ordinary unbounded derived category of quasi-coherent sheaves, and the triangulated category of perfect complexes on $X$.  

\begin{remark}
Throughout the paper we will work with the category of perfect complexes over our scheme or stacks, rather than with the bounded derived category. In the presence of singularities, $\Perf(-)$ is in some respects better behaved than $D^b(-)$. We remark that the two notions coincide 
when working with smooth geometric objects, under reasonable assumptions. This is the case for instance for a locally Noetherian regular scheme, 
see Illusie's Expos\'  e I in SGA6 \cite{sga6} or, for a more recent reference (and a stronger result),  Corollary 3.0.5 of \cite{BNP}. 
\end{remark}

By \cite[Theorem 1.3.4]{G},  the dg-category of quasi-coherent sheaves $\Qcoh(-)$ satisfies faithfully flat descent. Namely, let $Y \to X$ be a  faithfully flat cover, and let $Y^\bullet$ be the simplicial object given by the \v{C}hech nerve of $Y \to X$. Then we have $$ \varprojlim \Qcoh(Y^\bullet) \simeq \Qcoh(X),$$ where the (homotopy) limit on the left is computed in the category of symmetric monoidal dg-categories. It is easy to see that dualizable objects in the limit are the same as the limit of dualizable objects. Thus, passing to dualizable objects on both sides, we obtain  $$ \varprojlim \Db{Y^{\bullet}} \simeq \Db{X}.$$  As a consequence $\Db{-}$ also satisfies faithfully flat descent. This will be an important technical ingredient in our arguments.

\subsubsection*{{Quasi-coherent sheaves on pro-stacks}} 
We will work with pro-objects in stacks. We make the following formal definition of the category of perfect complexes on a pro-stack.  
\begin{definition}
\label{properf}
Let $\{X_i\}_{i\in I}$ be a pro-object in   stacks. We will formally define $$\Db{\varprojlim_i X_i}:=\varinjlim_i \Db{X_i}$$ as a dg-category.
\end{definition}

We do not claim that Definition \ref{properf} is a reasonable  definition of the category of perfect complexes over an arbitrary pro-object in stacks. However, in the concrete case of the kind of pro-stacks that we will be interested in, this definition can be justified. We devote the rest of this section to a discussion of this aspect.  \begin{remark}\label{limit of affine}
If the transition maps $X_j\to X_i$ are affine, so that the inverse limit is also naturally an algebraic stack $X$, and if additionally every $X_i$ is coherent, and affine over a quasi-compact scheme or stack $X_0$, 
then we have an equivalence of abelian categories
$\Coh(\varprojlim_i X_i)\simeq \varinjlim\Coh(X_i)$ by the results of \cite[IV, \S 8]{EGA}, and an equivalence of categories $\Db{\varprojlim_i X_i}\simeq \varinjlim_i \Db{X_i}$ (see \cite[Tag 09RC]{stacks-project}). We remark however that this result does not apply to our situation, since the structure maps in the inverse systems that we will consider are not affine. 
\end{remark}

We will work with two kinds of inverse systems,  given by infinite root stacks and valuativizations,  which were introduced in the previous section. Infinite root stacks are inverse limits of the inverse system $\{\radice[n]{X}\}_{n\in \bN}$ of finite root stacks of the log scheme or stack $X.$ 
These stacks have groupoid presentations $R_n\rightrightarrows U_n\to \radice[n]{X}$ that fit very nicely together (the projections lift to affine morphisms $R_m\to R_n$ and $U_m\to U_n$). Because of this we have a concrete description of quasi-coherent and coherent sheaves on $\radice[\infty]{X}$ (see \cite[Section 4]{TV}), and 
thus also of  the category of perfect complexes $\Db{\radice[\infty]{X}}$. We will check below (Proposition \ref{prop.db.irs}) that, under some assumptions, 
this gives the same dg-category as the direct limit of the categories of perfect complexes $\Db{\radice[n]{X}}$ of the finite root stacks. Thus, in the case of the infinite root stack,  
Definition \ref{properf} gives a reasonable notion, since it captures perfect complexes on the actual limit of 
the inverse system. 

The structure maps of the second inverse system, which computes the valuativization, are blow-ups.  Thus they are very far from being affine as required by Remark \ref{limit of affine}. In this case it is maybe less clear that our definition for the category of perfect complexes is reasonable. We do not explore this question any further, but we point out the following related result. If $X$ is a log regular quasi-compact log scheme $X$, then Nizio{\l} \cite[Corollary 3.9]{Ni1} proves that the direct limit of the category of locally free sheaves on the Kummer-flat site of log blow-ups of $X$ is equivalent to the category of locally free sheaves on the valuative Kummer-flat site of $X$. Thus, at least under these 
assumptions and in the ``Kummer-flat'' context, vector bundles on the valuativization are modeled by the colimit of the categories of vector bundles on log blow-ups. 
 
To conclude this section we will check that, as anticipated, under reasonable assumptions  we have an equivalence  
$$\Db{\radice[\infty]{X}}\simeq \varinjlim_n \Db{\radice[n]{X}}.$$
We start by proving a general result on colimits of pretriangulated dg-categories equipped with a t-structure.
\begin{theorem}\label{colim}
Let $I$ be a filtered category, and let $\{\cC_i\}_{i \in I}$ be a filtered system of dg-categories. Assume that all the structure maps $$\alpha_{i \to j}:\cC_i \to \cC_j$$ are fully faithful. 
Then the colimit $\cC := \varinjlim \cC_i$ is the following  dg-category 
\begin{itemize}
\item  The objects of $\cC$ are given by 
the union $\bigcup_{i\in I} \mathrm{Ob}(\cC_i)$.
\item The Hom complex between  $A_i \in \cC_i$ and  $A_j\in \cC_j$ is given  by 
$$\Hom_\cC(A_i, A_j) = \Hom_{\cC_l} (\alpha_{l \to i}(A_i),\alpha_{l \to j}(A_j))$$ where $l$ is any object of $I$ that is the source of morphisms $ j \leftarrow l \rightarrow i$. 
\end{itemize}
\end{theorem}
\begin{proof}
Cohn has shown in \cite{Cohn} that the 
underlying  $\infty$-category of the category of dg-categories, with the Morita model structure, is the $\infty$-category of $k$-linear stable categories. Furthermore, as stabilization is a left adjoint, we
can compute a filtered colimit of stable $k$-linear $\infty$-categories in the category of $\infty$-categories. A description of filtered colimits of $\infty$--categories is given in  \cite{R}. Then the fully faithfulness of the structure maps, together with the result of  \cite{R}, implies our claim.  
\end{proof}

We need a refinement of this result to the case where the categories $\cC_i$ are the dg enhancements of bounded derived categories, and the structure maps are exact. Notations and assumptions in the next lemma are carried over from   Theorem \ref{colim}. 

\begin{lemma} \label{commute colim}
Assume that $\cC_i = D^b_{dg}(A_i)$, where 
$\{A_i\}_{i \in I}$ is a filtered system of abelian categories. Assume also that the structure maps preserve the abelian categories, that is, for all $i \to j$ in $I$,     
$
\alpha_{i \to j}(A_i) \subset A_j
$. 
Then there is an equivalence 
$$\varinjlim_i \cC_i = \varinjlim_i D^b_{dg}(A_i) \simeq  D^b_{dg}(\varinjlim_i A_i).$$ 
\end{lemma} 
\begin{proof}
It is easy to see from the explicit description of $ \cC= \varinjlim \cC_i$ given in Theorem \ref{colim}  that the colimit of the canonical $t$-structures in $\cC_i$ forms a $t$-structure on $\cC$ with heart
$\varinjlim_i A_i$.  The $\delta$-functor $ \varinjlim A_i \to \cC$ induces a functor 
$D^b_{dg}(\varinjlim A_i)\to \varinjlim \cC_i =  \cC$, as $D^b_{dg}(\varinjlim A_i)$ is the triangulated hull  of $\varinjlim A_i$.  The fact that this  is an equivalence follows from \cite[Proposition 1.3.3.7]{Lu2} and the explicit descriptions of the Hom complexes in $\cC$. 
\end{proof}

\begin{proposition}\label{prop.db.irs}
Let $X$ be a noetherian fine saturated log algebraic stack over $k$, with locally free log structure. Then there is an equivalence of dg-categories
$$
\varinjlim_n \Db{\radice[n]{X}}\simeq \Db{\radice[\infty]{X}}.
$$
\end{proposition}

\begin{proof}
Applying Lemma \ref{commute colim}, we obtain that the colimit $\varinjlim_n D^b_{dg}(\Coh(\radice[n]{X}))$  is equivalent to $D^b_{dg}(\varinjlim_n \Coh(\radice[n]{X}))$. Since $X$ is noetherian and the log structure is locally free, we can apply \cite[Proposition 6.1]{TV} (together with \cite[Proposition 4.19]{TV}), and conclude that the direct limit $\varinjlim_n \Coh(\radice[n]{X})$ is equivalent to the abelian category $\Coh(\radice[\infty]{X})$. This implies that
$$
\varinjlim_n D^b_{dg}(\Coh(\radice[n]{X})) \simeq 
D^b_{dg}(\Coh(\radice[\infty]{X})). 
$$ 

Note that there is an inclusion 
$\Db{\radice[\infty]{X}} \subseteq D^b_{dg}(\Coh(\radice[\infty]{X}))$. Indeed, by  \cite[Proposition 4.19]{TV}, the structure sheaf  $\cO_{\radice[\infty]{X}}$ is coherent. Thus, using 
\cite[Section 1.5]{Ga}, $\radice[\infty]{X}$ is eventually coconnective and, by the  proof of \cite[Proposition 1.5.3]{Ga}, this implies that $\Perf(  \radice[\infty]{X}) \subseteq 
D^b_{dg}(\Coh(\radice[\infty]{X}) )$. Further 
$\Db{\radice[\infty]{X}} $ can be characterized as 
the full subcategory of dualizable objects inside $D^b_{dg}(\Coh(\radice[\infty]{X}))$. In symbols, we write 
$$
\Db{\radice[\infty]{X}}  =( D^b_{dg}(\Coh(\radice[\infty]{X})))^{\mathrm{dual}}
$$ 
From the explicit description of  $\varinjlim_n D^b_{dg}(\Coh(\radice[n]{X}))$ given by Theorem \ref{colim}, and the fact that the structure maps are monoidal, we  
conclude that the dualizable objects in $D^b_{dg}(\Coh(\radice[\infty]{X}))$ are given by the image of  the dualizable objects in the categories $D^b_{dg}(\Coh(\radice[n]{X}))$. In formulas 
$$\varinjlim_n \Db{\radice[n]{X}} =  \varinjlim_n (D^b_{dg}(\Coh(\radice[n]{X})))^{\mathrm{dual}} \simeq (D^b_{dg}( \Coh(\radice[\infty]{X}) ))^{\mathrm{dual}} \simeq \Db{\radice[\infty]{X}},$$ which is what we wanted to prove.
\end{proof}

Note that by \cite[Theorem 7.3]{TV} the category of quasi-coherent sheaves on $\radice[\infty]{X}$ is equivalent to the category $\Par(X)$ of parabolic sheaves on $X$ with rational weights. Therefore we can also identify $D^b_{dg}(\Coh(\radice[\infty]{X}))$  with the (dg enhancement of the) bounded derived category of the category of coherent parabolic sheaves with rational weights on $X$.

%%%%%%%%%%%%%%%%%%%%%%%%%%%%%%%%%%%%%%%%%%%

\section{Setup and statement of the main theorem}\label{sec:simple.log.ss}

In this section we specify our setup and explain the construction of the two objects that we will prove to be ``derived equivalent''. We will then state our main result (Theorem \ref{thm:main}), and start with the proof in the next section.

\begin{assumption}
Let $X$ be a fine saturated log algebraic stack over $k$ with affine stabilizer groups, which is moreover \emph{log flat} over the base field $k$ (equipped with the trivial log structure), and such that the log structure is \emph{locally free}, i.e. for every geometric point $x\to X$ we have an isomorphism $\overline{M}_{X,x}\cong \bN^{s}$ for some $s\in \bN$.
\end{assumption}

Note that such an $X$ has local charts of the form $X\to \bA^s$, where $\bA^s$ has its canonical toric log structure, and log flatness is equivalent to requiring that these maps be flat in the classical sense (see for example \cite[Theorem 4.15]{Ols}).  In fact, we require log flatness of $X$ exactly in order to have flatness of these chart morphisms. This is ultimately used  (in Section \ref{sec:colocopro}) to ensure that certain fibered products along these morphisms coincide with the derived fiber product, and therefore we can compare the dg-category of perfect complexes on such a fiber product with perfect complexes of the factors (Lemma \ref{lem:fibered.prod}). 

\begin{example}
A smooth variety $Y$ equipped with a normal crossings divisor $D$ gives a log scheme $(Y,D)$, which is of this form. More generally, if $Z\to Y$ is a flat morphism (where $Y$ is as above), then we can equip $Z$ with the pullback log structure, and the resulting log scheme will still satisfy our assumptions.
\end{example}

One of the sides of our derived equivalence is given by root stacks of $X$. For every $n\in \bN$ we can form $\radice[n]{X}$, and the inverse limit of this system is the infinite root stack
$\irs{X}=\varprojlim_n \radice[n]{X}$.

Describing the other side is slightly more complicated. Instead of forming root stacks we are going to take log blow-ups, but before doing that we need to perform a preliminary construction. We consider a second log structure on $X$ (or, more precisely, on the underlying algebraic stack $\underline{X}$ of $X$) as follows.

Let $C\subseteq X$ be the reduced closed substack where the log structure of $X$ is concentrated (i.e. the stalk $\overline{M}_{X,x}$ is not zero), and denote by $\bN_C$ the constant sheaf over $C$ with stalks equal to the monoid $\bN$ (pushed forward to $X$). In our setting, there is a unique homomorphism of sheaves of monoids $\bN_C\to \overline{M}_X$ on $X$ such that for every geometric point $x\to C$, the induced homomorphism $$(\bN_C)_x=\bN \to \overline{M}_{X,x}\cong \bN^s$$ is the diagonal map $n\mapsto (n,\hdots,n)$ (note that this does not depend on the chosen isomorphism $\overline{M}_{X,x}\cong \bN^s$). By composing this homomorphism with the symmetric monoidal functor $\overline{M}_X\to \Div_X$ that encodes the given log structure of $X$, we obtain a second log structure on the algebraic stack $\underline{X}$. Let us denote by $X'$ the resulting log algebraic stack. There is a natural morphism of log algebraic stacks $X\to X'$, that on charts is described by the diagonal homomorphism $\bN\to \bN^s$.

Locally where $X$ has a chart $X\to \bA^s$, we can consider the toric morphism $\bA^s\to \bA^1$ given by $(x_1,\hdots, x_s)\mapsto x_1\cdots x_s$, and the pullback on $X$ of the log structure of $\bA^1$ via the composite $X\to \bA^s\to \bA^1$. These locally defined log structures on $X$ glue together to give the global one described in the previous paragraph. In particular $X'$ is again a fine saturated log algebraic stack.

\begin{remark}\label{rmk:a1gm}
If $(Y,D)$ is a smooth variety with a normal crossings divisor, then this second log structure can be described as follows: consider the morphism $Y\to [\bA^1/\bG_m]$ determined by the line bundle $\cO_X(D)$ with its tautological global section (recall that morphisms to $[\bA^1/\bG_m]$ are classified by line bundles with global sections). The alternative log structure on $Y$ is obtained by pulling back the natural log structure of $[\bA^1/\bG_m]$ via this map. If the divisor $D$ has a global equation $\{f=0\}$, then $Y\to [\bA^1/\bG_m]$ factors through $Y\to \bA^1$ determined by the regular function $f$ on $Y$, and the log structure can be pulled back from this $\bA^1$.
\end{remark}

Before taking root stacks and log blow-ups, we will preliminarily extract every possible root out of the log structure of $X'$. {For later reference let us introduce special notations for the root stacks of $X'$.}

\begin{definition}
\label{def:xn}
{We denote by $X_n$ the $n$-th root stack $\radice[n]{X'},$ and by $X_\infty$ the infinite root stack $\irs{X'}$.}
\end{definition}

Since taking root stacks commutes with strict base change, locally where $X$ has a chart of the form $X\to \bA^s$ we have cartesian diagrams
$$
\xymatrix{
 X_n\ar[r]\ar[d] & X\ar[d]  && X_\infty\ar[r]\ar[d] & X\ar[d]\\
\radice[n]{\bA^1}\ar[r] & \bA^1
 && \irs{\bA^1}\ar[r] & \bA^1}
$$
where $X\to \bA^1$ is the composite $X\to \bA^s\to \bA^1$, as discussed above. 

\begin{remark}
As mentioned in the introduction, the significance of this operation is best understood by thinking of $X$ as the total space of a simple kind of semistable degeneration. We will be more precise about these degenerations in Section \ref{sec:simple.log.sss}. Suppose for simplicity that we are in the local situation, so there is a flat chart $X\to \bA^s$, that we can compose as before with $\bA^s\to \bA^1$. Assume also that the resulting morphism $X\to \bA^1$ is log smooth.

If we think of $X$ as a degeneration over $\bA^1$ with singular central fiber, we want to modify $X$ in order to add stackyness over the singular locus of the central fiber (via root stacks), or to allow the central fiber to degenerate by inserting additional components in the singular locus (via log blow-ups). In order to do this, there are at least two problems with our picture: root stacks $\radice[n]{X}$ of $X$ will have stacky structure along the whole central fiber (not only on the singular locus), and neither root stacks nor log blow-ups will have reduced fibers over the base $\bA^1$ itself. This can already be seen in the simple case of $\bA^2\to \bA^1$ sending $(x,y)$ to $xy$.

Both problems are solved if, before modifying in either way, we extract roots of the coordinate of the base $\bA^1$, by base changing to root stacks $\radice[n]{\bA^1}$ for various values of $n$ (or for $n=\infty$, at the limit). Note that  $\radice[n]{\bA^1}$ has a natural \'etale presentation given by $\bA^1$ itself acted on by $\mu_n$ in the obvious way, and the composite $\bA^1\to \radice[n]{\bA^1}\to \bA^1$ is the map sending $t$ to $t^n$. Base changing through this kind of finite maps is of course familiar from the context of semistable reduction.
\end{remark}

{On one side of our derived equivalence we will have the infinite root stack $\irs{X_\infty}$.} Note that there is a canonical isomorphism of stacks $\irs{X_\infty}\to \irs{X}$. For the other side of our derived equivalence, we {will} consider the valuativization $X_\infty^\val=\varprojlim_{(n,\cI_n^\alpha)} (X_n)_{\cI_n^\alpha}$, where $I=\{(n,\cI_n^\alpha)\}$ is the direct system of coherent sheaves of ideals of the log structure on the stacks $X_n=\radice[n]{X'}$ that we already considered in (\ref{sec:valuativization}), and $(X_n)_{\cI_n^\alpha}$ denotes the log blow-up along $\cI_n^\alpha$. 

Let us also consider the fibered product $Y:=\radice[\infty]{X}\times_{X_\infty} X_\infty^\val$, and the resulting cartesian diagram
$$
\xymatrix{
& Y\ar[ld]\ar[rd] & \\
\radice[\infty]{X} \ar[rd] & & X_\infty^{\val} \ar[ld] \\ 
& X_\infty & 
}
$$
where everything also lives over the original $X$.

Recall from {Section \ref{sec:dg.categories}}  that we define the categories of perfect complexes of $\irs{X}$ and $(X_\infty)^\val$ by taking a direct limit of the categories of perfect complexes of the algebraic stacks that make up their defining inverse system (note that for this we are using the whole inverse system, and not just the limit). In other words
$$\Db{\irs{X}}=\varinjlim_n \Db{\radice[n]{X}}$$ and $$\Db{X_\infty^\val}=\varinjlim_{(n,\cI_n^\alpha)} \Db{(X_n)_{\cI_n^\alpha}}.$$
As explained {earlier}, the first formula can actually be seen as a statement rather than a definition, at least if $X$ is noetherian (Proposition \ref{prop.db.irs}).

Now we observe that there is a natural functor $\Phi\colon \Db{X_\infty^\val}\to  \Db{\radice[\infty]{X}}$, which is morally the Fourier-Mukai functor given by the structure sheaf $\cO_Y$ of the fibered product. Since $Y$ is just a fibered category in our formalism, instead of trying to make precise sense of this, let us note that $\Phi$ can be defined as a colimit of ``honest'' Fourier-Mukai functors.

Attached to the filtered direct system $I=\{(n,\cI_n^\alpha)\}$, we have three corresponding inverse systems of stacks: we can assign to $(n,\cI_n^\alpha)$ either:
\begin{itemize}
\item the stack $X_n$ (note that this disregards $\cI_n^\alpha$). For $(m,\cI_m^\beta)\geq (n,\cI_n^\alpha)$ we have a natural map $X_m\to X_n$, and the inverse limit is the stack $X_\infty$;
\item the stack $\radice[n]{X}$ (this disregards $\cI_n^\alpha$ as well). For $(m,\cI_m^\beta)\geq (n,\cI_n^\alpha)$ we have a natural map $\radice[m]{X}\to \radice[n]{X}$, and the inverse limit is the stack $\irs{X}$;
\item the log blow-up $(X_n)_{\cI_n^\alpha}$. For $(m,\cI_m^\beta)\geq (n,\cI_n^\alpha)$ we have a natural map $(X_m)_{\cI_m^\beta}\to (X_n)_{\cI_n^\alpha}$, and the inverse limit is the stack $X_\infty^\val$.
\end{itemize}
We can also set $Y_{(n,\cI_n^\alpha)}:=\radice[n]{X}\times_{X_n}(X_n)_{\cI_n^\alpha}$. Together with the natural transition maps this gives a fourth inverse system, and the inverse limit is canonically isomorphic to $Y=\radice[\infty]{X}\times_{X_\infty} X_\infty^\val$.

Moreover, for every $(n,\cI_n^\alpha)\in I$ we have a Fourier-Mukai functor $$\Phi_{(n,\cI_n^\alpha)}\colon \Db{(X_n)_{\cI_n^\alpha}}\to \Db{\radice[n]{X}}$$ given by pull-push along the projections $Y_{(n,\cI_n^\alpha)}\to  (X_n)_{\cI_n^\alpha}$ and $Y_{(n,\cI_n^\alpha)}\to  \radice[n]{X}$.

\begin{remark}
Note that since we are working with no smoothness or finite type assumptions, it is not a priori clear that the functor $\Phi_{(n,\cI_n^\alpha)}\colon \Qcoh((X_n)_{\cI_n^\alpha})\to \Qcoh(\radice[n]{X})$ will carry perfect complexes to perfect complexes. This will follow from our proof of the main theorem (see in particular the proof of Proposition \ref{prop:fm.local}).

For simplicity we will gloss over this detail in the rest of the paper, and pretend that we already know that the functors $\Phi_{(n,\cI_n^\alpha)}\colon \Db{(X_n)_{\cI_n^\alpha}}\to \Db{\radice[n]{X}}$ are well-defined.
\end{remark}

We will show in the next proposition that these functors are compatible with the structure maps (given by pullback) of the two direct systems of dg-categories $\{\Db{\radice[n]{X}} \}_{(n, \cI_n^\alpha)\in I}$ and $\{\Db{(X_n)_{\cI_n^\alpha}}\}_{(n,\cI_n^\alpha)\in I}$, so that we can consider the colimit
$$
\Phi:=\varinjlim_{(n,\cI_n^\alpha)\in I} \Phi_{(n,\cI_n^\alpha)}\colon\Db{X_\infty^\val}\to  \Db{\irs{X}}.
$$

\begin{proposition}\label{prop:FM.compatible}
Let $(n,\cI_n^\alpha)$ and $(m,\cI_m^\beta)$ be two elements of the index system $I$, such that $(m,\cI_m^\beta)\geq (n,\cI_n^\alpha)$, i.e. $m\geq n$ by divisibility, and $\cI_m^\beta\geq (\cI_n^\alpha)^m$, where $(\cI_n^\alpha)^m$ denotes the sheaf of ideals generated by the pullback of $\cI_n^\alpha$ to $\radice[m]{X}$. In this situation we have maps $v\colon \radice[m]{X}\to \radice[n]{X}$ and $w\colon (X_m)_{\cI_m^\beta}\to (X_n)_{\cI_n^\alpha}$.

Then the induced diagram of dg-categories
$$
\xymatrix@C=2cm{
\Db{(X_n)_{\cI_n^\alpha}} \ar[r]^{\Phi_{(n,\cI_n^\alpha)}}\ar[d]_{w^*} & \Db{\radice[n]{X}}\ar[d]^{v^*} \\
\Db{(X_m)_{\cI_m^\beta}}\ar[r]^{\Phi_{(m,\cI_m^\beta)}} & \Db{\radice[m]{X}}
}
$$
is commutative.
\end{proposition}

\begin{proof}
Consider the commutative diagram
$$
\xymatrix{
& \radice[m]{X}\ar[dd]^<<<<<v|!{[d]}\hole \ar[dl] & & Y_{(m,\cI_m^\beta)}\ar[dd]^u\ar[dl]^{p_{m}}\ar[ll]_{q_{m}}\\
{X_{m}} \ar[dd] & & (X_m)_{\cI_m^\beta}\ar[ll] \ar[dd]^<<<<<w& \\
& \radice[n]{X} \ar[dl] & & Y_{(n,\cI_n^\alpha)}\ar[dl]^{p_n} \ar[ll]_>>>>>>>>{q_n}|!{[l]}\hole\\
{X_n} & & (X_n)_{\cI_n^\alpha}\ar[ll] & 
}
$$
where as above $Y_{(n,\cI_n^\alpha)}$ and $Y_{(m,\cI_m^\beta)}$ are the fibered products of the bottom and top face of the cube, respectively.

We have to show that there is a natural equivalence of functors
$$
v^*\circ \Phi_{(n,\cI_n^\alpha)}\simeq \Phi_{(m,\cI_m^\beta)}\circ w^*\colon \Db{(X_n)_{\cI_n^\alpha}}\to \Db{\radice[m]{X}}.
$$
By writing out the Fourier-Mukai functors, this translates to an equivalence
$$
v^* \circ (q_{n})_* \circ  (p_n)^* \simeq (q_{m})_* \circ (p_{m})^* \circ w^*.
$$
Note that $(q_{m})_* \circ (p_{m})^* \circ w^* \simeq  (q_{m})_* \circ u^* \circ (p_n)^*$.  Thus it is enough to define a natural equivalence 
$$
v^* \circ (q_n)_* \simeq (q_{m})_* \circ u^*. 
$$
Consider the diagram 
$$
\xymatrix{
Y_{(m,\cI_m^\beta)}\ar[r]^u  \ar[d]_{q_{m}} & Y_{(n,\cI_n^\alpha)} \ar[d]^{q_n}\\
\radice[m]{X}\ar[r]^v & \radice[n]{X}
}
$$
{to which we want to apply a flat base change formula. Although the map $v$ is flat, because the log structure of $X$ is locally free, this diagram is not cartesian.}

Let $Z$ be the fiber product of  
$$
Y_{(n,\cI_n^\alpha)} \stackrel{q_n} \longrightarrow   \radice[n]{X}  
 \stackrel{v}\longleftarrow \radice[m]{X} ,  
 $$
and consider the commutative diagram 
$$
\xymatrix{
Y_{(m,\cI_m^\beta)} \ar[rd]^c\ar@/^1pc/[rrd]^u\ar@/_1pc/[ddr]_{q_{m}}& &\\
&Z\ar[r]^d  \ar[d]_{e} & Y_{(n,\cI_n^\alpha)}\ar[d]^{q_n}\\
&\radice[m]{X}\ar[r]^v & \radice[n]{X}
}
$$
We claim that $c_*\cO_{Y_{(m,\cI_m^\beta)}} \cong \cO_Z$. Note first that this is true if $X=\bA^s$ for some $s$: in that case, the morphism $c$ is a proper representable birational toric morphism of toric stacks. By passing to a toric presentation of the target, call it $U\to Z$, we reduce to proving the claim for $c'\colon U'=U\times_Z Y_{(m,\cI_m^\beta)}\to U$, which is a morphism with the same properties but between toric varieties. In this case, the statement follows from the fact that toric varieties have rational singularities (see for example \cite[Theorem 5.2]{cox}): in fact, if $\phi\colon W\to U'$ is a resolution of singularities, then the composite $c'\circ \phi\colon W\to U'\to U$ is a resolution as well, and we have both $\phi_*\cO_W\cong \cO_{U'}$ and $(c'\circ \phi)_*\cO_W=c'_*(\phi_*\cO_W)\cong \cO_U$, from which we get $c'_*\cO_{U'}\cong \cO_U$, as desired.

In the general case, the claim follows from localizing to where $X$ has a flat chart $X\to \bA^s$, and applying base change. We refer to Corollary 1.4.5 of \cite{drinfeld}, where base change and projection formula are proven for maps of ``QCA'' stacks (short for ``quasi-compact with affine automorphism groups''), which is a large class of derived stacks that includes in particular the ones appearing in our proof. 

By the projection formula, there is a natural equivalence $c_*\circ c^* \simeq \id$. Thus we have 
$$
(q_{m})_*\circ  u^* \simeq e_*\circ c_*\circ c^*\circ d^* \simeq e_*\circ d^* \simeq v^*\circ (q_n)_*,
$$
where the last equivalence is given again by base change (note that, since $v$ is flat, the ordinary fibered product in the diagram coincides with the derived one). This concludes the proof.
\end{proof}

We are now ready to state our main result.

\begin{theorem}\label{thm:main}
Let $X$ be a quasi-compact fine saturated log algebraic stack over $k$ with affine stabilizer groups, which is log flat over $k$ and has locally free log structure. Then the ``Fourier-Mukai'' functor $\Phi\colon \Db{X_\infty^\val}\to  \Db{\radice[\infty]{X}}$ described above is an equivalence of dg-categories.
\end{theorem}

The proof of the theorem will consist of a reduction to an analogous ``finite-index'' statement for a cofinal subset of the inverse systems $\{\radice[n]{X}\}_{(n,\cI_n^\alpha)\in I}$ and $\{(X_n)_{\cI_n^\alpha}\}_{(n,\cI_n^\alpha)\in I}$. We will deduce this finite-index version from the local case of $\bA^s$ equipped with the standard degeneration $ \bA^{s}\to \bA^1$, given by $(x_1,\hdots, x_{s})\mapsto x_1\cdots x_s$, of $\bG_m^{s-1}$ to the union of the $s$ coordinate hyperplanes in $\bA^{s}$. In this local case, as explained in the introduction, our result will be a consequence of the derived McKay correspondence, proven for abelian quotient singularities by Kawamata \cite[Theorem 4.2]{Ka}. Discussing the local case first is helpful in order to prepare for the construction of the cofinal system for the blow-ups $(X_n)_{\cI_n^\alpha}$ in the global case, so we start with that part in the next section, and treat the globalization process in Section \ref{sec:globalizing}. 

{
 {However, before proceeding with the proof, we show that Theorem \ref{thm:main} implies that {the derived} category of parabolic sheaves on a smooth variety equipped with a normal crossings divisor is invariant under log blow-ups. This is one of our main applications.} 
}

\subsection*{{Parabolic sheaves on smooth varieties with normal crossings boundary}}
Theorem \ref{thm:main} can be applied in particular to log schemes of the form $X=(Y,D)$ where $Y$ is a smooth variety over $k$ equipped with a normal crossings divisor $D$, for which it implies that there is an equivalence of dg-categories $\Db{(Y,D)_\infty^\val} \simeq \Db{\irs{(Y,D)}}$. From this we obtain the following consequence.

\begin{proposition}\label{cor:parabolic}
Let $(Y,D)$ be the log scheme given by a smooth variety $Y$ over $k$ equipped with a normal crossings divisor $D$, and let $(Y',D')\to (Y,D)$ be a {log blow-up, such that $(Y',D')$ is again a smooth variety with a normal crossings divisor.} Then there is an equivalence of dg-categories
$$
D^b_{dg}(\Par(Y,D))\simeq D^b_{dg}(\Par(Y',D'))
$$
between the bounded derived categories of {coherent} parabolic sheaves with rational weights.
\end{proposition}

This fits well with the philosophy that in the ``divisorial'' case, some objects defined in terms of the log structure should actually only depend on the open part (i.e. the complement of the boundary). 

{Before proceeding with the proof, let us make one preliminary observation. Recall that a normal crossings divisor $D \subset Y$  comes with a canonical smooth stratification. One way to see this is to consider the normalization map 
$n: \widetilde{D} \to D.$ The locally closed subschemes of $D$ where the cardinality of the fibers of $n$ stays constant give rise to a stratification of $D.$  From now on  we will simply call \emph{strata} the  closed smooth strata on the normal crossings divisor $D,$ viewed as closed subschemes of $Y$.}

{Now, let $(Y', D') \to (Y, D)$ be a log blow-up satisfying the assumptions of the theorem. By the factorization theorem for toroidal birational 
maps between smooth varieties (see Theorem B of \cite{wlod}, and Theorem 8.3 of \cite{abrammats}),  we can factor this log blow-up as a finite sequence of blow-ups and blow-downs with centers a single smooth closed stratum. Thus, in the proof of Proposition \ref{cor:parabolic} we can, and will, assume without loss of generality that $(Y', D') \to (Y, D)$ is a blow-up centered on a single smooth closed stratum.}

\begin{proof}
{We break down the proof in three steps.}

{{\bf First step.}}
Consider the morphism $(Y,D)\to [\bA^1/\bG_m]$ induced by the divisor $D$, described in Remark \ref{rmk:a1gm}. The stacks $(Y,D)_n$ in this case are obtained by pulling back along the maps $ \phi_n\colon [\bA^1/\bG_m]\to  [\bA^1/\bG_m]$, induced by raising to the $n$-th power both on $\bA^1$ and on $\bG_m$.

Let $r$ be the number of branches of $D$ meeting in the stratum that we are blowing up (including branches belonging to the same irreducible component). Consider the composite $$f\colon (Y',D')\to (Y,D)\to  [\bA^1/\bG_m].$$ We would like to say that the log blow-ups of the stacks $(Y',D')_n$ will be cofinal in the log blow-ups of the stacks $(Y,D)_n$, so that $(Y,D)_\infty^\val\simeq (Y',D')_\infty^\val$ (note that a composition of log blow-ups is again a log blow-up \cite[Corollary 4.11]{niziol}). As easy examples show though, the map $f$ is not exactly the morphism $(Y',D')\to  [\bA^1/\bG_m]$ that corresponds to the sum of the components of the divisor $D'$, because the exceptional divisor $E$ will have multiplicity bigger than $1$ if we take the pullback of $D$ to $Y'$. Consequently, it is not clear that there will be induced morphisms $(Y',D')_n\to (Y,D)_n$.

In order to fix this, let us extract $r$-th roots on the base $[\bA^1/\bG_m]$: consider the pullback $(Y,D)_r\to (Y,D)$ along the map $\phi_r$, and take the fibered product $Z=(Y,D)_r\times_{(Y,D)} (Y',D')$ in the category of  fine  saturated log algebraic stacks. Note that $Z$ is a stack that satisfies the assumptions of Theorem \ref{thm:main}, and the projection $Z\to (Y,D)_r$ is a log blow-up (by \cite[Corollary 4.8]{niziol}). Moreover, pullback along the composite $Z\to  (Y,D)_r\to [\bA^1/\bG_m]$ equips $Z$ (or rather, its underlying algebraic stack) with the ``correct'' log structure, whose root stacks are the stacks $Z_n$  {that we introduced in Definition \ref{def:xn}}.  Hence, log blow-ups of the stacks $\{Z_n\}_{n\in \bN}$ are also log blow-ups of the stacks $\{(Y,D)_{rn}\}_{n\in \bN}$, and it is clear that they are cofinal in this direct system. Hence, we have an isomorphism $Z_\infty^\val\simeq (Y,D)_\infty^\val$.

{{\bf Second step.} Next we prove that there is a natural isomorphism $\irs{Z}\simeq \irs{(Y',D')}.$ This follows from the fact that $Z\to (Y',D')$ is an $r$-th root stack morphism, with respect to the strict transform of $D$ on $Y'$ (i.e. all components of the divisor $D'$, except the exceptional divisor $E$), where the roots are taken separately along each component.}
Let us prove this claim: it suffices to work locally, around a point $y\in D\subseteq Y$, that is moreover in the stratum that we are blowing up. Around $y$ there will be a chart $(Y,D)\to \bA^r$ sending $y$ to the origin. Since the formation of log blow-ups and root stacks are compatible with strict base-change, we can verify the claim on the universal model, where we blow up the origin in $\bA^r$.

Let $B\to \bA^r$ be the blow-up of $\bA^r$ at the origin. Using toric language, this can be described as the toric variety associated with the fan $\Sigma$, given by the subdivision of the first orthant in $\bR^r$, where we add the ray $\rho$ generated by the vector $(1,\hdots, 1) \in \bZ^r$. The stack $\bA^r_r$ can be seen in this case as the fibered product $\radice[r]{\bA^1}\times_{\bA^1}\bA^r$, where $\bA^r\to \bA^1$ is the map $(x_1,\hdots, x_r)\mapsto x_1\cdots x_r$. This can be described as a toric stack (whose definition has been recalled in Section \ref{sec:toric.stacks}) by considering the first orthant $\bR^r_{\geq 0}$ as a fan in the lattice $$L_r=\{(a_1,\hdots, a_{r})\in \bZ^{r} \mid a_1+\cdots +a_{r}\equiv 0\; (\mathrm{mod} \; r)\}$$ equipped with the inclusion of lattices $L_r\subseteq \bZ^r$ (this toric picture will come up again in Section \ref{sec:local}). The log blow-up $Z$ of $\bA^r_r$ given by the fine saturated fibered product $\bA^r_r\times_{\bA^r} B$ is described, as a toric stack, by subdiving this fan, using again the ray $\rho$ as above.

Now the claim becomes that this log blow-up $Z$ coincides with the root stack of $B$ along the strict transforms of the coordinate hyperplanes of $\bA^r$. This is clear by using the toric picture: an affine chart of the blow-up $B$ is given by the smooth cone with rays consisting of all the positive coordinates axes except one, and the ray $\rho$. Extracting an $r$-th root of the divisors corresponding to the coordinate hyperplanes corresponds to looking at these cones in a sublattice of $\bZ^{r}$, where for each $i=1,\hdots, r$ (except one index that is omitted) we replace the basis element ${e_i}$ by the vector $r\cdot e_i$ (as in Example \ref{ex:toric.root}), and we do not change anything for the last basis element, the primitive generator of $\rho$. These cones reassemble to give a stacky fan in the lattice $L_r$, which is exactly the one of the log blow-up $Z$, and it also clear that the map $Z\to B$ is the projection from the root stack. This proves the claim.

{{\bf Third step.}} Theorem \ref{thm:main} applied to $(Y,D)$ and to $Z$ yields equivalences of dg-categories $$\Db{\irs{(Y,D)}}\simeq \Db{(Y,D)_\infty^\val}, \quad \Db{\irs{Z}}\simeq \Db{Z_\infty^\val},
$$
that, together with the isomorphisms $Z_\infty^\val\simeq (Y,D)_\infty^\val$ and $\irs{Z}\simeq \irs{(Y',D')}$, induce an equivalence of dg-categories $$\Db{\irs{(Y,D)}}\simeq \Db{\irs{(Y',D')}}.$$
To conclude the proof, note that since the root stacks are smooth and noetherian in this case, we have compatible equivalences 
$\Db{\radice[n]{(Y,D)}} \simeq D^b_{dg}(\radice[n]{(Y,D)})$ (by \cite[Corollary 3.0.5]{BNP}) that induce an equivalence $\Db{\radice[\infty]{(Y,D)}} \simeq D^b_{dg}(\radice[\infty]{(Y,D)})$
(and the same is true for $(Y',D')$). Finally, we obtain an equivalence of dg-categories $$D^b_{dg}(\Par(Y,D))\simeq D^b_{dg}(\Par(Y',D')),$$ since by \cite[Theorem 7.3]{TV} there is an equivalence of abelian categories {$\Coh(\irs{(Y,D)})\simeq \Par(Y,D)$}. 
\end{proof}

\begin{remark}
We point out that the abelian categories of coherent parabolic sheaves over $(Y,D)$ and $(Y',D')$ {are most likely not equivalent}, and the fact that their bounded derived categories are equivalent is not at all clear without making use of our construction. It would be interesting to find an alternative and more explicit way to write down this equivalence, by using the description of parabolic sheaves with rational weights as systems of sheaves indexed by $\overline{M}^\gp_\bQ$ (see \cite{borne-vistoli, TV}).
\end{remark}

{Proposition \ref{cor:parabolic} has immediate K-theoretic consequences, 
that we briefly explain next. Let $(Y, D)$ be as in the statement of Proposition \ref{cor:parabolic}. We denote by $(Y, D)_{\mathrm{kfl}}$ the Kummer-flat site of 
$(Y, D)$  and by  $K((Y, D)_{\mathrm{kfl}})$ the algebraic K-theory spectrum of the Kummer-flat topos of $(Y, D).$  Recall that, following the seminal work of Thomason-Trobaugh \cite{TTro}, the algebraic K-theory of a ringed topos is defined as the  K-theory of its  category of perfect complexes. We refer the reader to \cite{Ni1} for more background on  K-theory in the logarithmic setting.} 

\begin{corollary}
\label{cor:kth}
{Let $(Y,D)$ be the log scheme given by a smooth variety $Y$ over $k$ equipped with a normal crossings divisor $D$, and let $(Y',D')\to (Y,D)$ be a {log blow-up, such that $(Y',D')$ is again a smooth variety with a normal crossings divisor.} Then there is an equivalence of spectra 
$$
K((Y, D)_{\mathrm{kfl}}) \simeq K((Y', D')_{\mathrm{kfl}}).$$}
\end{corollary}

It would be interesting to understand the effect of this equivalence using the explicit description of the K-theory given in \cite{sst}.

\begin{proof}{
Applying K-theory to the equivalence of dg-categories obtained in Proposition \ref{cor:parabolic}, we get an equivalence of spectra
$$
K(\Db{\radice[\infty]{(Y,D)}}) \simeq K(\Db{\radice[\infty]{(Y',D')}}). 
$$ 
In order to finish the proof, it is sufficient to observe that we have equivalences 
$$
K({(Y,D)}_\mathrm{kfl}) \simeq K(\Db{\radice[\infty]{(Y,D)}}), \quad 
K({(Y',D')}_\mathrm{kfl}) \simeq K(\Db{\radice[\infty]{(Y',D')}}). 
$$   
This is a consequence of the following two facts. 
First, note that the K-theory spectra 
$$K(\Db{\radice[\infty]{(Y,D)}}), \quad \text{ and } \quad 
K(\Db{\radice[\infty]{(Y',D')}})$$ coincide respectively with the K-theory of the small fppf topos of $\radice[\infty]{(Y,D)}$, and of $\radice[\infty]{(Y',D')}$  {(see \cite[Section 6]{TV} for a discussion about the small fppf topos of an infinite root stack)}.  Second, by Theorem 6.15 of \cite{TV},  the Kummer-flat topos of $(Y,D)$ is equivalent to the small fppf topos of  
$\radice[\infty]{(Y,D)},$ and the same is true for $\radice[\infty]{(Y',D')}.$}  
\end{proof}

%%%%%%%%%%%%%%%%%%%%%%%%%%%%%%%%%%%%%%%%%%%

\section{The local case} \label{sec:local}

In this section we will assume that $X=\bA^{r+1}$ for $r\geq 0$, and prove Theorem \ref{thm:main} in this case. For this, we will make extensive use of the map $\phi_r\colon \bA^{r+1}\to \bA^1$ defined by $\phi_r(x_1,\hdots, x_{r+1})=x_1\cdots x_{r+1}$, which gives the local picture of the semistable degenerations that we will consider in Section \ref{sec:simple.log.sss}.

Let us first follow through the general construction of the previous section, and describe the objects $\radice[\infty]{X}$ and $X_\infty^\val$ in this particular case. We will denote by $t$ the coordinate of the base $\bA^1$, so that $\phi_r$ is given in terms of algebras by the homomorphism $k[t]\to k[x_1,\hdots, x_{r+1}]$ sending $t$ to the product $x_1\cdots x_{r+1}$. From now we will abbreviate $k[x_1,\hdots,x_{r+1}]$ by $k[x_i]$.

\subsection{Root stacks} {Let us start by  describing explicitly  the main objects appearing in our constructions in the local case, on the root stack side: } 
\begin{itemize}
\item The root stacks of the base $\bA^1$ are given by quotient stacks $$\radice[n]{\bA^1}=[\Spec k[t_n]/\mu_n],$$ where $\Spec k[t_n]\to \radice[n]{\bA^1}\to \bA^1=\Spec k[t]$ corresponds to $t\mapsto t_n^n$ (so that morally $t_n\sim t^{\frac{1}{n}}$), and the action of $\mu_n$ is the obvious one. 

\item The root stack $\bA^{r+1}_n=\radice[n]{(\bA^{r+1})'}$ of ${\bA^{r+1}}$ with respect to the log structure pulled back from $\bA^1$ is isomorphic to the fibered product $\bA^{r+1}\times_{\bA^1}\radice[n]{\bA^1}$, for which we have
$$\bA^{r+1}_n=\radice[n]{\bA^1}\times_{\bA^1}\bA^{r+1}\simeq [\Spec k[x_i, t_n]/(x_1\cdots x_{r+1} - t_n^n) \, / \, \mu_n]$$ where $\mu_n$ acts on $t_n$ as for $\radice[n]{\bA^1}$.

\item The root stack $\radice[n]{\bA^{r+1}}$ can be described as the quotient
$$
\radice[n]{\bA^{r+1}}=[\Spec k[x_{i,n}] \, / \, \mu_n^{r+1}]
$$
for new variables $x_{1,n},\hdots, x_{r+1,n}$, where $\mu_n^{r+1}$ acts diagonally, in the obvious way, and the composite $\Spec k[x_{i,n}]\to \radice[n]{\bA^{r+1}}\to \bA^{r+1}=\Spec k[x_i]$ is given algebraically by $x_i\mapsto x_{i,n}^n$ (so that morally $x_{i,n} \sim x_i^{\frac{1}{n}}$).
\end{itemize}
\begin{remark}\label{rmk:toric}
{Note that} everything in sight is either a toric variety or a toric stack. %We refer the reader to \cite{gsa} for  the basics about toric stacks.

The toric stack 
$\bA^{r+1}_n$ can be described via another global quotient. The action of $\mu_n^{r+1}$ on $\bA^{r+1}=\Spec k[x_{i,n}]$  induces an action of the kernel  of the morphism $$\mu_n^{r+1}\to \mu_n, \quad (\xi_1, \ldots, \xi_{r+1}) \mapsto \xi_1 \cdot \ldots \cdot \xi_{r+1},$$  which we identify with $\mu_n^r$ via the map $(\eta_1,\hdots, \eta_r)\mapsto (\eta_1,\hdots, \eta_r, (\eta_1\cdots \eta_r)^{-1})$.

The quotient $\bA^{r+1}/\mu_n^r$  has an action of {the target} $\mu_n$, and we have a canonical isomorphism $\bA^{r+1}_n\simeq [( \bA^{r+1}/\mu_n^r) \, /\, \mu_n]$. In fact the quotient  $\bA^{r+1}/\mu_n^r$ is exactly $\Spec k[x_i, t_n]/(x_1\cdots x_{r+1} - t_n^n)$.

{The theory of stacky fans (recalled in Section \ref{sec:toric.stacks}) allows to encode the difference between the toric variety $\bA^{r+1}/\mu_n^r$ and the toric stack $\bA^{r+1}_n$ in a combinatorial way. Let $\cF_n$ be the fan of $\bA^{r+1}/\mu_n^r$ in a lattice $L_n$ (see below for more details about these). Then the stacky quotient $\bA^{r+1}_n$ is described by the same fan in the same lattice, with the difference that we are also keeping track of an inclusion of lattices $L_n\subseteq \bZ^{r+1}$ (denoted by $N\to N'$ in Section \ref{sec:toric.stacks}).%, or equivalently an inclusion $(\bZ^{r+1})^\vee\to L_n^\vee$.
}
The kernel of the morphism between the Cartier duals $\bG_m^{r+1}=D(L_n^\vee)\to D((\bZ^{r+1})^\vee)=\bG_m^{r+1}$ is exactly the copy of $\mu_n$ appearing in the quotient stack description of $\bA^{r+1}_n$.
\end{remark}

\begin{example}\label{rmk:r=1}
The simplest case beyond the trivial one is the case $r=1$, where the degeneration $\bA^2\to \bA^1$ mapping $(x,y)$ to $xy$ is the universal smoothing of a node in a curve, see also Example \ref{examp:a2} from the {i}ntroduction. 

In this case the fibered product $\bA^2_n=\bA^2\times_{\bA^1}\radice[n]{\bA^1}$ is the quotient $[(\bA^2/\mu_n)\,  / \, \mu_n]$, where $\bA^2/\mu_n$ is the quotient by the action $\xi\cdot (x_n,y_n)=(\xi x_n,\xi^{-1} y_n)$, which is the affine scheme $\Spec k[x,y,t_n]/(xy-t_n^n)$. This is the standard form of the $A_{n-1}$ surface singularity, and the root stack morphism $\radice[n]{\bA^2}\to \bA^2_n$ pulls back to the ``stacky resolution'' $[\bA^2/\mu_n]\to \bA^2/\mu_n$. The construction of the blow-up side that follows will, in this particular case, reduce to considering the crepant resolution of singularities of $\bA^2/\mu_n$.
\end{example}

\subsection{Log blow-ups} 
{Next we turn to the other side of our equivalence, which is given by  the valuativization.} Let us first spend some words on log blow-ups and the valuativization of $X=\bA^{r+1}$ itself. As for any toric variety, the valuativization $(\bA^{r+1})^\val$ is the inverse limit of all toric blow-ups (see Remark \ref{rmk:toric.blow.up}). These are conveniently described by subdivisions of the fan $\cF$ associated to the toric variety, that in this case is the cone given by the first orthant $\cF=\bR_{\geq 0}^{r+1}$ of the vector space $\bR^{r+1}$, with lattice $\bZ^{r+1}$. So the valuativization is described as the inverse limit of the schemes that we get by subdividing the fan $\cF$ further and further. 

{Rather than in $\bA^{r+1}$ itself, we will be  mostly interested in the toric stacks obtained from $\bA^{r+1}$ by extracting roots on the base $\bA^1$. Pulling back the degeneration $\bA^{r+1}\to \bA^1$ along the $n$-th root stack morphism $\radice[n]{\bA^1}\to \bA^1$,
has the effect of replacing the lattice $\bZ^{r+1}$ by the inclusion $L_n\subseteq \bZ^{r+1}$ (the resulting object will be a toric stack), where $L_n$ is the lattice of vectors $(a_1,\hdots, a_{r+1})$ such that $a_1+\cdots +a_{r+1}\equiv 0\; (\mathrm{mod} \; n)$. This operation does not change the fan.} Let us denote by $\cF_n$ the fan given by the first orthant in $\bR^{r+1}$, but with lattice the subgroup $L_n$ described above. Because of this description, subdivisions of the original fan $\cF$ will induce subdivisions of the fan $\cF_n$.

{For our purposes it will be important to 
consider certain specific kinds of subdivisions of the fan 
$\cF_n,$ which we describe next.} Consider the hyperplane given by the equation $
x_1+\cdots+x_{r+1}=n
$  in $\bR^{r+1}$, and the standard simplex that the first orthant determines on it. With respect to the lattice $L_n$, this simplex is isomorphic to the standard simplex $n\Delta^r$ of dimension $r$, but whose sides have $n+1$ lattice points each. {We can subdivide the fan $\cF_n$ by choosing  a unimodular (i.e. such that every lattice point is contained in a face, or equivalently, every lattice point in the interior of $n\Delta^r$ is a vertex) triangulation $\cT$ of the simplex $n\Delta^r$, and construct the induced fan in the space $\bR^{r+1}$ with lattice $L_n$, by taking the cones over the simplices in the triangulation $\cT$.} 

\begin{remark}
The operation of constructing a fan by taking cones over some polyhedra sitting in a hyperplane should be familiar from the theory of   toric singularities \cite{altmann1997versal}. It also plays a key role in 
Hori-Vafa mirror symmetry of toric LG models, we refer the reader to Section 3 of  \cite{pascaleff2016topological} for a discussion of these aspects.
\end{remark}

Denote by $(\bA^{r+1}_n)_\cT$ the {toric} blow-up of $\bA^{r+1}_n$ corresponding to this fan (we trust that there will be no confusion with the notation for a log blow-up along a sheaf of ideals of the log structure). Note that the stack $\bA^{r+1}_n\simeq  [( \bA^{r+1}/\mu_n^r) \, /\, \mu_n]$ is a global quotient of a toric variety $( \bA^{r+1}/\mu_n^r)$ (as mentioned in Remark \ref{rmk:toric}), whose fan coincides with $\cF_n$, and this gives a canonical way to define the log blow-up on an atlas, as described in (\ref{sec:valuativization}).

\begin{remark}\label{rmk:crepant}
We point out that the log blow-up $(\bA^{r+1}/\mu_n^r)_\cT\to ( \bA^{r+1}/\mu_n^r)$ is a crepant resolution of singularities for any unimodular triangulation $\cT$ as above: since it is given by a subdivision of the fan, it is a proper birational morphism, the total space is smooth since every cone in its fan is smooth, and both varieties have trivial canonical divisor. We refer the reader to  \cite{altmann1997versal} for additional information. 
\end{remark}

{By extracting further roots 
and subdividing further, we obtain a {direct} system of triangulations}. More precisely, if $m\in \bN$ and $\cT$ is a unimodular triangulation of $n\Delta^r$, then $\cT$ will induce a (non-unimodular) triangulation of the simplex $(mn)\Delta^r$, that we can refine further to a unimodular one in several possible ways. If $\cT'$ is such a refinement, we write $(mn,\cT')\geq (n,\cT)$, and we will simply say that $\cT'$ is a \emph{refinement} of $\cT$ (see also Remark \ref{rmk:rational.triangulations} below). In this way the set of pairs $\{(n,\cT_n^\alpha)\}$ where $n\in \bN$ and $\cT_n^\alpha$ is a unimodular triangulation of $n \Delta^r$ becomes a filtered partially ordered set.

To each $(n,\cT_n^\alpha)$ we can associate  the log algebraic stack $(\bA^{r+1}_n)_{\cT_n^\alpha}$, and if $(mn,\cT_{mn}^\beta)\geq (n,\cT_n^\alpha)$ we have an induced morphism $(\bA^{r+1}_{mn})_{\cT_{mn}^\beta}\to (\bA^{r+1}_n)_{\cT_n^\alpha}$. Thus we have an inverse system of log algebraic stacks indexed by the pairs $(n,\cT_n^\alpha)$.

\begin{remark}\label{rmk:rational.triangulations}
The family of triangulations $\{\cT_n^\alpha\}_{(n,\alpha)}$ can also be thought of as the family of ``(finite) rational triangulations'' of the standard simplex $\Delta^r$ (with two lattice points on each side), such that:
\begin{itemize}
\item every vertex of the triangulation is a rational linear combination of the vertices of $\Delta^r$, and
\item if $n$ is the least common denominator of the ones appearing in these (reduced, i.e. where the denominator can not be simplified further) rational linear combinations, then the induced lattice triangulation of the simplex $n\Delta^r$ is unimodular.
\end{itemize}
Moreover we will have $\cT_{mn}^\beta\geq \cT_n^\alpha$ if and only if, as rational triangulations of $\Delta^r$, $\cT_{mn}^\beta$ is a refinement of $\cT_n^\alpha$.
\end{remark}

Let us consider the inverse limit of the system $\{(\bA^{r+1}_n)_{\cT_n^\alpha}\}_{(n,\cT_n^\alpha)}$.

\begin{proposition}
There is a canonical isomorphism of log fibered categories $$(\bA^{r+1}_\infty)^\val\simeq \varprojlim_{(n,\cT_n^\alpha)} (\bA^{r+1}_n)_{\cT_n^\alpha}.$$
\end{proposition}

\begin{proof}
For every $(n,\cT_n^\alpha)$ the stack $(\bA^{r+1}_n)_{\cT_n^\alpha}$ is a log blow-up of $\bA^{r+1}_n$, and the structure maps of the inverse system are the natural maps between log blow-ups. Thus to conclude we only have to check that these log blow-ups are cofinal in the inverse system of log blow-ups of the stacks $\bA^{r+1}_n$ (as in (\ref{sec:valuativization}), in particular Proposition \ref{prop:diagonal.valuativization}).

Let $\cY\to \bA^{r+1}_n$ be a log blow-up. The pullback to the toric atlas $\bA^{r+1}/\mu_n^r$ of $\bA^{r+1}_n$ (Remark \ref{rmk:toric}) is then a log blow-up as well, and is therefore (by Remark \ref{rmk:toric.blow.up}) dominated by some toric blow-up, that will be defined by some subdivision $\Sigma$ of the fan $\cF_n$ of $\bA^{r+1}/\mu_n^r$ (which, recall, is the first orthant in $\bR^{r+1}$ with lattice $L_n=\{(a_1,\hdots, a_{r+1})\in \bZ^{r+1} \mid a_1+\cdots +a_{r+1}\equiv 0\; (\mathrm{mod} \; n)\}$).  
Now consider the subdivision induced by $\Sigma$ on the simplex $n\Delta^r$, seen as the portion of the hyperplane $H=\{a_1+\cdots+a_{r+1}=n\}$ lying in the first orthant, by considering for each cone $\sigma\in \Sigma$ the intersection $\sigma\cap H$ (which is going to be a rational polytope contained in $n\Delta^r$). This might not be a lattice subdivision, but by increasing $n$ we can make it so, and we can then refine it to a unimodular triangulation $\cT_n^\alpha$. Therefore we obtain a canonical morphism $(\bA^{r+1}_n)_{\cT_n^\alpha} \to \cY$ factoring $(\bA^{r+1}_n)_{\cT_n^\alpha}\to \bA^{r+1}_n$. This shows that the log blow-ups $(\bA^{r+1}_n)_{\cT_n^\alpha}$ are cofinal in the system of log blow-ups of the stacks $\bA^{r+1}_n$, and concludes the proof.
\end{proof}

\begin{corollary}
The natural maps $(\bA^{r+1}_\infty)^\val\to (\bA^{r+1}_n)_{\cT_n^\alpha}$ induce an equivalence of dg-categories
$$
\varinjlim_{(n,\cT_n^\alpha)} \Db{(\bA^{r+1}_n)_{\cT_n^\alpha}}\simeq \Db{(\bA^{r+1}_\infty)^\val}
$$
and the same holds for any cofinal subsystem of the family $\{(\bA^{r+1}_n)_{\cT_n^\alpha}\}_{(n,\cT_n^\alpha)}$.\qed 
\end{corollary}

\subsection{{Compatible cofinal systems of triangulations}}
{The statement of Theorem \ref{thm:main} in the local case is that there is an equivalence of dg-categories 
$
\Db{(\bA^{r+1}_\infty)^\val}\to \Db{\radice[\infty]{\bA^{r+1}}},$ which is obtained as a limit of Fourier-Mukai functors induced by a push-pull formalism, see Section \ref{sec:simple.log.ss}.} We will deduce this by proving that we can choose a sequence $\{k_n\}_{n\in \bN}$ of positive integers and, for every $n,$ a unimodular triangulation $\cT_{k_n}$ of the standard simplex $k_n\Delta^r$ such that 
\begin{itemize}
\item $k_n$ is cofinal in $\bN$ ordered by divisibility, so that $\varprojlim_n \radice[k_n]{\bA^{r+1}}\simeq \radice[\infty]{\bA^{r+1}}$, and
\item the triangulations $\cT_{k_n}$ are cofinal in the system of all rational triangulations of $\Delta^r$ (in the sense of Remark \ref{rmk:rational.triangulations}), so that $\varprojlim_n (\bA^{r+1}_{k_n})_{\cT_{k_n}}\simeq (\bA^{r+1}_\infty)^\val$.
\end{itemize}
We will also check (Proposition \ref{prop:mckay}) that for every $n$ and every unimodular triangulation $\cT$ of the simplex $n\Delta^r$,  the Fourier-Mukai functor $\Db{(\bA^{r+1}_n)_{\cT}}\to \Db{\radice[n]{\bA^{r+1}}}$ given by push-pull along the diagram $$
\xymatrix{
& \radice[n]{\bA^{r+1}}  \times_{\bA_n^{r+1}} (\bA^{r+1}_n)_{\cT} \ar[ld]\ar[rd] & \\
\radice[n]{\bA^{r+1}} \ar[rd] & & (\bA^{r+1}_n)_{\cT}  \ar[ld] \\ 
& \bA_n^{r+1} & 
}
$$is an equivalence of dg-categories. {By the previous discussion, this will imply our result.}

\begin{remark}\label{rmk:r.1}
Note that if $r=1$ (i.e. we are looking at the degeneration $\phi_1\colon \bA^2\to \bA^1$ of $\bG_m$ to the union of the two coordinate axes), then there is a unique unimodular triangulation of the simplex $n\Delta^1$ for all $n$, corresponding to the minimal resolution of the surface $A_{n-1}$ singularity, see Remark \ref{rmk:r=1}. In this case we do not have to worry about carefully choosing this cofinal system of triangulations (which, on the other hand, for higher values of $r$ are far from unique).
\end{remark}

{In view of adapting this strategy to the global case in the next section, we will also take care to choose the unimodular triangulations $\cT_{k_n}$ to be completely symmetric with respect to permutations of the vertices of $k_n\Delta^r$ (so that the local blow-ups will be compatible, and we will be able to glue them together). We stress that using these symmetric triangulations is not necessary to prove Theorem \ref{thm:main} in the local case. However they will play an important role in the next Section, where we will give a proof of Theorem \ref{thm:main} in the general case.} 

{
We devote the next subsection to the construction of symmetric triangulations, and resume the proof of Theorem \ref{thm:main} in the local setting in (\ref{sec:cooftheco}).}

\subsubsection{{Symmetric triangulations}}
Now we take a short detour  to establish the existence of triangulations of the lattice simplices $n\Delta^r$ with some {symmetry properties that  will be important later on.} {Given an integer $n$}, in Lemma \ref{lem:sym} below we establish in particular the existence of a positive integer $mn \in \bN$ divisible by $n$ and a unimodular triangulation of the simplex $(mn)\Delta^r$, that refines every possible lattice triangulation of $n\Delta^r$ and is moreover symmetric with respect to permutations of the vertices.

The basic idea to obtain such a triangulation is the following: we can ``superimpose'' every possible lattice triangulation of $n\Delta^r$, of which there will be finitely many, and think of this picture as giving a rational triangulation of $\Delta$ (as in Remark \ref{rmk:rational.triangulations}). By looking at the coordinates of the vertices of this new triangulation we can determine the integer $m$, and then we can refine the corresponding lattice triangulation to be unimodular and symmetric. {Before giving of a formal proof, we clarify this geometry in the simplest non-trivial example where  $r=2$ and $n=2$.}

\begin{example}\label{ex:triangulation}
Assume $r=2$ and $n=2$, so that we are looking at the lattice polytope $2\Delta^2\subseteq \bR^3$. This is a triangle, with 3 lattice points on each side and none in the interior, that can be realized as the closed region on the hyperplane $x+y+z=2$ in $\bR^3$ {lying in the first orthant}.
\bigskip
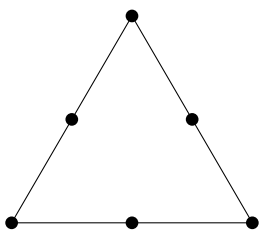
\begin{center}
\begin{tikzpicture}[scale=0.8]
\draw (0,0) -- (4,0) -- (4/2 , 0.86*4 ) -- cycle;
\fill (0,0) circle (3pt);
\fill (1 , 0.86*2 ) circle (3pt);
\fill (3 , 0.86*2 ) circle (3pt);
\fill (2 , 0) circle (3pt);
\fill (4,0) circle (3pt);
\fill (4/2 , 0.86 *4 )  circle (3pt);
\end{tikzpicture}
\end{center}
\bigskip
Up to symmetries, there are exactly $2$ different unimodular triangulations of this simplex, given by the following pictures.
\bigskip
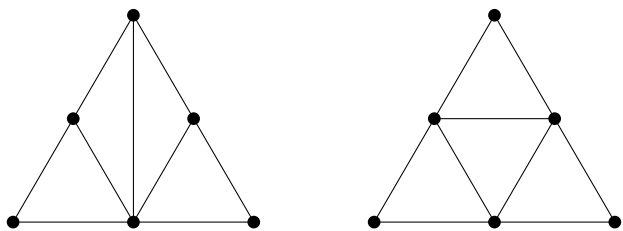
\begin{center}
\begin{tikzpicture}[scale=0.8]
\draw (0,0) -- (4,0) -- (4/2 , 0.86 *4 ) -- cycle;
\draw (2,0) -- (1 , 0.86*2 );
\draw (2,0) -- (3 , 0.86*2 );
\draw (2,0) -- (4/2, 0.86 *4);
\fill (0,0) circle (3pt);
\fill (1 , 0.86*2 ) circle (3pt);
\fill (3 , 0.86*2 ) circle (3pt);
\fill (2 , 0) circle (3pt);
\fill (4,0) circle (3pt);
\fill (4/2 , 0.86 *4 )  circle (3pt);
\draw (6,0) -- (10,0) -- (4/2+6 , 0.86 *4 ) -- cycle;
\draw (2+6,0) -- (1+6 , 0.86*2 );
\draw (2+6,0) -- (3+6 , 0.86*2 );
\draw (1+6 , 0.86*2 ) -- (3+6 , 0.86*2 );
\fill (0+6,0) circle (3pt);
\fill (1+6 , 0.86*2 ) circle (3pt);
\fill (3 +6, 0.86*2 ) circle (3pt);
\fill (2+6 , 0) circle (3pt);
\fill (4+6,0) circle (3pt);
\fill (4/2+6 , 0.86 *4 )  circle (3pt);
\end{tikzpicture}
\end{center}
\bigskip
If we superimpose them (and all the ones obtained by symmetries), we obtain
\bigskip
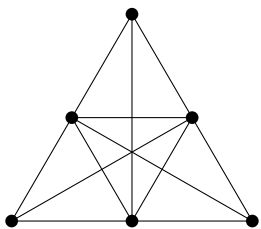
\begin{center}
\begin{tikzpicture}[scale=0.8]
\draw (0,0) -- (4,0) -- (4/2 , 0.86 *4 ) -- cycle;
\draw (2,0) -- (1 , 0.86*2 );
\draw (2,0) -- (3 , 0.86*2 );
\draw (2,0) -- (4/2, 0.86 *4);
\draw (1 , 0.86*2 ) -- (3 , 0.86*2 );
\draw (0,0) -- (3 , 0.86*2 );
\draw (1 , 0.86*2 ) -- (4,0);
\fill (0,0) circle (3pt);
\fill (1 , 0.86*2 ) circle (3pt);
\fill (3 , 0.86*2 ) circle (3pt);
\fill (2 , 0) circle (3pt);
\fill (4,0) circle (3pt);
\fill (4/2 , 0.86 *4 )  circle (3pt);
\end{tikzpicture}
\end{center}
\bigskip
which will become a lattice triangulation if we look at the lattice $\frac{1}{6}\bZ^3$ (so $mn=12$), obtaining the following picture.
\bigskip
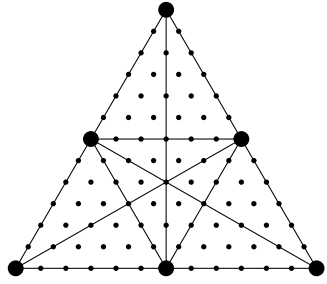
\begin{center}
\begin{tikzpicture}
\draw (0,0) -- (4,0) -- (4/2 , 0.86 *4 ) -- cycle;
\draw (2,0) -- (1 , 0.86*2 );
\draw (2,0) -- (3 , 0.86*2 );
\draw (2,0) -- (4/2, 0.86 *4);
\draw (1 , 0.86*2 ) -- (3 , 0.86*2 );
\draw (0,0) -- (3 , 0.86*2 );
\draw (1 , 0.86*2 ) -- (4,0);
\fill (0,0) circle (3pt);
\fill (1 , 0.86*2 ) circle (3pt);
\fill (3 , 0.86*2 ) circle (3pt);
\fill (2 , 0) circle (3pt);
\fill (4,0) circle (3pt);
\fill (4/2 , 0.86 *4 )  circle (3pt);
\foreach \j in {0,...,12}
{
\foreach \i in {0,...,\j}
{
\fill (12 * 4/24-\j*4/24+ \i* 4/12, 12* 0.86 * 4/12 -\j* 0.86 * 4/12) circle (1pt);
}
}
\end{tikzpicture}
\end{center}
We can complete this to a unimodular symmetric triangulation of $12\Delta^2$, for example as shown in the following picture.
\bigskip
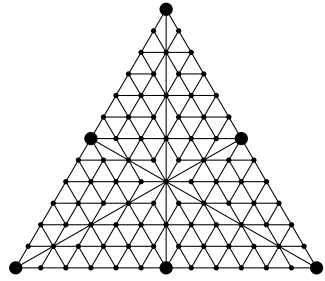
\begin{center}
\begin{tikzpicture}[scale=1]
\draw (0,0) -- (4,0) -- (4/2 , 0.86 *4 ) -- cycle;
\draw (2,0) -- (1 , 0.86*2 );
\draw (2,0) -- (3 , 0.86*2 );
\draw (2,0) -- (4/2, 0.86 *4);
\draw (1 , 0.86*2 ) -- (3 , 0.86*2 );
\draw (0,0) -- (3 , 0.86*2 );
\draw (1 , 0.86*2 ) -- (4,0);
\fill (0,0) circle (2.5pt);
\fill (1 , 0.86*2 ) circle (2.5pt);
\fill (3 , 0.86*2 ) circle (2.5pt);
\fill (2 , 0) circle (2.5pt);
\fill (4,0) circle (2.5pt);
\fill (4/2 , 0.86 *4 )  circle (2.5pt);
\foreach \j in {0,...,12}
{
\foreach \i in {0,...,\j}
{
\fill (12 * 4/24-\j*4/24+ \i* 4/12, 12* 0.86 * 4/12 -\j* 0.86 * 4/12) circle (1pt);
}
}
\draw (4/24, 0.86*4/12 ) -- ( 4/12+4/24+4*4/12, 0.86*4/12 );
\draw (5*4/12+4/12+4/24, 0.86*4/12 ) -- (6*4/12+ 4/12+4/24+4*4/12, 0.86*4/12 );
\draw (2*4/24, 2*0.86*4/12 ) -- ( 4-2*4/24, 2*0.86*4/12 );
\draw (4/12,0) -- (4/12+5*4/24, 5*0.86*4/12);
\draw (1+4/12, 6*0.86*4/12 ) -- (2+4/24, 11*0.86*4/12);
\draw (2*4/12, 0 ) -- ( 2+2*4/24, 10*0.86*4/12 );
\draw (5*4/12+ 4/12+2*4/24+4*4/12, 0 ) -- (3-4/24, 5*0.86*4/12);
\draw (3-4/12, 6*0.86*4/12)-- (2-4/24,11*0.86*4/12);
\draw ( 4-4*4/24, 0) -- ( 2-4/12, 10*0.86*4/12 );
\draw (4*4/12,0) -- (2+2*4/12, 8*0.86*4/12 );
\draw (2+2*4/12,0) -- (1+4/12, 8*0.86*4/12 );
\draw (4*4/24, 4*0.86*4/12) -- (4-4*4/24, 4*0.86*4/12);
\draw (1+4/12, 8*0.86*4/12) -- (3-4/12, 8*0.86*4/12);
\draw (1+2*4/12, 10*0.86*4/12) -- (3-2*4/12, 10*0.86*4/12);
\draw (2+2*4/12,0) -- (3+4/12, 4*0.86*4/12);
\draw (2+4*4/12,0) -- (4-2*4/24, 2*0.86*4/12);
\draw (2*4/12,0) -- (4/12, 2*0.86*4/12);
\draw (4*4/12,0) -- (2*4/12, 4*0.86*4/12);
\draw (2-3*4/24, 9*0.86*4/12) -- (2-4/24, 9*0.86*4/12);
\draw (2+4/24, 9*0.86*4/12) -- (2+3*4/24, 9*0.86*4/12);
\draw (3*4/12,0) -- (3*4/12 - 4/24, 0.86*4/12);
\draw (2*4/12, 2*0.86*4/12) -- (3*4/24, 3*0.86*4/12);
\draw (2+3*4/12, 0) -- (2+3*4/12+4/24, 0.86*4/12);
\draw (2+3*4/12+2*4/24, 2*0.86*4/12) -- (4-3*4/24, 3*0.86*4/12);
\draw (3*4/24, 3*0.86*4/12 ) -- (3*4/24+4*4/12,3*0.86*4/12 );
\draw (2+4/24, 3*0.86*4/12 ) -- (4-3*4/24,3*0.86*4/12 );
\draw (5*4/24, 5*0.86*4/12 ) -- (2-4/24,5*0.86*4/12 );
\draw (2+4/24, 5*0.86*4/12 ) -- (4-5*4/24,5*0.86*4/12 );
\draw (7*4/24, 7*0.86*4/12 ) -- (2-4/24,7*0.86*4/12 );
\draw (2+4/24, 7*0.86*4/12 ) -- (4-7*4/24,7*0.86*4/12 );
\draw (2-4/12,0) -- (2-4*4/24, 2*0.86*4/12);
\draw (2-5*4/24, 3*0.86*4/12) -- (1-4/24, 5*0.86*4/12);
\draw (2+4/12, 0) -- (2-4/24, 3*0.86*4/12);
\draw (2-2*4/24, 4*0.86*4/12) -- (1+4/24, 7*0.86*4/12);
\draw (2+3*4/12, 0) -- (2+4/12, 4*0.86*4/12);
\draw (2+4/24, 5*0.86*4/12) -- (2-3*4/24, 9*0.86*4/12);
\draw (2+4/12,0) -- (2+2*4/12, 2*0.86*4/12);
\draw (2+5*4/24, 3*0.86*4/12) -- (2+7*4/24, 5*0.86*4/12);
\draw (2-4/12,0) -- (2+4/24, 3*0.86*4/12);
\draw (2+4/12, 4*0.86*4/12) -- (3-4/24,7*0.86*4/12);
\draw (3*4/12 ,0) -- (5*4/12, 4*0.86*4/12);
\draw (2-4/24 ,5*0.86*4/12) -- (2+3*4/24, 9*0.86*4/12);
\end{tikzpicture}
\end{center}
\medskip
\end{example}

Let us fix some notations for the following lemmas. 
Let $N$ be an $r$-dimensional lattice,  
$N \cong \bZ^r$, and let $\Delta = n \Delta^r$ be a lattice simplex.  With small abuse of notation, we denote $\Delta$ also the convex envelope of $\Delta$ in  
$N_\bR = N \otimes_\bZ \bR $. We can equip $N_\bR$ with a scalar product by choosing a basis $E$ of $N$, and stipulating that $E$ is an orthonormal basis of $N_\bR$. The automorphism group of $\Delta$, as a polytope, coincides with the group of permutations on the set of vertices of $\Delta.$ By labeling the vertices of $\Delta$ with numbers from $0$ to $n,$  we can identify the automorphism group of $\Delta$ with the symmetric group $S_{r+1}$ on the letters  
$
\{0, \ldots, r\}.
$

The action of $S_{r+1}$ on $\Delta$ can be described geometrically in a simple way. Let $i$ and 
$j$ be vertices of $\Delta$. The pair $i,j$ determines a hyperplane $H_{i,j}$  in $N_\bR$:   $H_{i,j}$ is the hyperplane passing through the midpoint between $i$ and $j$ and all the vertices of $\Delta$ except $i$ and $j$. 
We can realize the transposition $(i, j) \in S_{r+1}$ as the reflection of $\Delta$ through the hyperplane 
$H_{i,j}$. Thus we can also identify the group $S_{r+1}$ with the group of automorphisms of 
$\Delta \subset N_\bR$ generated by these reflections.

\begin{lemma}
\label{lem:sym}
With the notation as above, there exists an $m \in \bN$ such that, if $N'$ is the overlattice 
$$
N \subset N' = \frac{1}{m}N,   
$$
and $\Delta'$ is the image of $\Delta$ in $N',$ then  
$\Delta'$ admits a unimodular triangulation 
$\cT$ with the following properties: 
\begin{enumerate}
\item $\cT$ refines all triangulations of $\Delta,$ 
\item $\cT$ is invariant under the action of $S_{r+1}$ on $\Delta'.$
\end{enumerate}
\end{lemma}
\begin{proof}
Consider the set of all line segments in $N_{\bR}$ joining two lattice points contained in $\Delta$.  In order to find a triangulation refining all triangulations of $\Delta$ we need, first of all, to consider a refinement $M$ of 
the lattice $N$ such that all the intersection points of these line segments lie in $M$. Note that the intersection points will have rational coordinates: i.e. there exists an 
$l \in \bN$ such that the intersections points all lie in 
$$
M=\frac{1}{l}N \subset N_\bR. 
$$
Call $\Delta_M$ the image of $\Delta$ in $M$. The intersections of the simplices belonging to all the triangulations of $\Delta$ gives rise to a lattice polyhedral decomposition of $\Delta_M,$ which we denote $\cP.$ By construction $\cP$ is the coarsest polyhedral decomposition of $\Delta_M$ that refines simultaneously all the triangulations of $\Delta.$ Note that $\cP$ is invariant under the action of $S_{r+1}.$

In order to refine this to a unimodular triangulation $\cT$ that is invariant under the action of $S_{r+1}$, we need to consider a refinement of the lattice $N$ which is, in principle, different from $M$. Namely, we need to consider an overlattice $M'$ of $N$ having the property that  all the intersections points of the hyperplanes $H_{i,j}$, with themselves and with the faces of $\Delta$, lie in $M'.$  
These intersection points are rational  and thus, as before, there exists $l' \in \bN$ such that the overlattice 
$$
M' = \frac{1}{l'}N \subset N_\bR 
$$   
has the required property. 

Now let $m = \mathrm{lcm}\{l, l'\}$, and let $N'$ be the refinement of $N$ given by 
$$
N \subset N' = \frac{1}{m}N \subset N_\bR.
$$
Consider the polyhedral decomposition of $\Delta'$ given by the closed chambers cut out by the hyperplanes $H_{i,j}$. Let $C_H$ be the set of these chambers. Since $N'$ is a refinement of $M'$, all the elements of  $C_H$ are lattice polyhedra for $N'.$  
Further, each of the elements of $C_H$ is (the closure of) a fundamental domain for the action of $S_{r+1}$ on $\Delta'.$ In particular, the cardinality of $C_H$ is  equal to the order of $S_{r+1}$, and $C_H$ is a 
$S_{r+1}$-torsor. 

Fix a chamber $C$ in $C_H.$ The polyhedral decomposition $\cP$ constructed above  restricts to a polyhedral decomposition of $C,$ which we denote $\cP_C$. Now choose  a unimodular triangulation of 
$C$, denoted $\cT_C$, that refines $\cP_C.$ Since 
$C_H$ is a torsor for $S_{r+1}$, we can transport the triangulation 
$\cT_C$ of $C$ to a triangulation of any other $C'$ in $C_H$ by acting with the appropriate element of $S_{r+1}$. This gives rise to a triangulation $\cT$ of $\Delta'.$ By construction $\cT$ refines $\cP,$ and thus it refines all triangulations of $\Delta$. Further $\cT$ is invariant under the action of $S_{r+1}$, and this concludes the proof. 
\end{proof}

A triangulation $\cT$ as in the statement of Lemma \ref{lem:sym} has additional properties that will be important later on. Namely, restricting $\cT$ to any of the faces of $\Delta$  yields a triangulation of that face that also satisfies properties $(1)$ and $(2)$ from Lemma \ref{lem:sym}.  We formulate a precise statement in  Lemma \ref{lem:sym2} below.

Let $\Delta_s$ be a an 
$s$-dimensional face of $\Delta.$ 
Denote $\Delta_s'$ the image of $\Delta_s$ in $N'$, where $N'$ is an overlattice of $N$ as in the statement of Lemma \ref{lem:sym}. Every triangulation $\cT$ of $\Delta$ (resp. $\Delta'$) restricts to a triangulation of $\Delta_s$ 
(resp. $\Delta_s'$) which we denote by $\cT_s.$
\begin{lemma}
\label{lem:sym2}
If $\cT$ is a unimodular triangulation of 
$\Delta'$ satisfying the conditions of Lemma \ref{lem:sym}, then it has the following additional property: for every $s \leq r$, and for every $s$-dimensional face $\Delta_s$ of $\Delta$, $\cT_s$ is a triangulation of $\Delta_s'$ such that \begin{enumerate}
\item $\cT_s$ refines all triangulations of $\Delta_s$,  
\item $\cT_s$ is invariant under the action of $S_{s+1}$ on $\Delta_s'.$
\end{enumerate}
\end{lemma}
\begin{proof}
Note that property $(2)$ is clearly satisfied: any $S_{r+1}$-invariant triangulation of $\Delta'$ restricts to an $S_{s+1}$-invariant triangulation of $\Delta'_s,$ as $S_{s+1}$ can be realized as a subgroup of $S_{r+1}$. Next, as in the proof of Lemma \ref{lem:sym}, denote $\cP$ be the polyhedral decomposition of $\Delta'$ obtained by intersecting all simplices in all triangulations of $\Delta.$ Note that its restriction $\cP_s$ to $\Delta_s'$ is the polyhedral decomposition  of $\Delta_s'$ obtained by intersecting all simplices in all triangulations of $\Delta_s.$ By construction, $\cT_s$ refines $\cP_s$ and is therefore a simultaneous refinement of all triangulations of $\Delta_s.$ 
\end{proof}

\subsubsection{{Construction of the cofinal system $\{(k_n, \cT_{k_n})\}_{n\in\bN}$}}
\label{sec:cooftheco}

Let us get back to the construction of a cofinal system $\{(k_n, \cT_{k_n})\}_{n\in\bN}$ of indices $k_n$ and unimodular triangulations $ \cT_{k_n}$ of $k_n\Delta^r$. {Recall that} we want to find a cofinal subsystem of pairs $\{(k_n,\cT_{k_n})\}_{n\in \bN}$, where $\cT_n$ is a unimodular triangulation of $n\Delta^{r}$, i.e. we want {$\{(k_n,\cT_{k_n})\}_{n\in \bN}$ to have the following properties:} 
\begin{itemize}
\item $k_{n-1} \mid k_n$, 
\item $\cT_{k_{n}}$ is a refinement of $\cT_{k_{n-1}}$ for all $n$, and 
\item for every pair $(m, \cT)$ there exists an $n_0$ such that $m\mid k_{n_0}$ and $\cT_{k_{n_0}}$ is a refinement of $\cT$ (and this will then be true for every $(k_n,\cT_{k_n})$ with $n\geq n_0$). Note that this in particular entails cofinality of $\{k_n\}$ in $\bN$ by divisibility.
\end{itemize} 
We will construct the sequence $\{(k_n,\cT_{k_n})\}_{n\in \bN}$ inductively, using Lemma \ref{lem:sym}, and we will also choose $\cT_{k_n}$ to be symmetric with respect to permutations of the vertices of the simplex $k_n\Delta^r$.

Start by setting $k_1=1$, and taking $\cT_1$ to be the (trivial) unique unimodular triangulation of the simplex $\Delta^r$. Assume that we have constructed $(k_{n-1},\cT_{k_{n-1}})$. By Lemma \ref{lem:sym}, there exists an index $m(k_{n-1})$ and a (symmetric) unimodular triangulation $\cT_{m(k_{n-1})}$ of $m(k_{n-1})\Delta^r$ that refines every possible unimodular triangulation of the simplex $k_{n-1}\Delta^r$. Moreover, if we increase this index to $n\cdot m(k_{n-1})$, there will be some (symmetric) unimodular triangulation $\cT_{n\cdot m(k_{n-1})}$ of $n\cdot m(k_{n-1})\Delta^r$ that refines the triangulation $\cT_{m(k_{n-1})}$. Set $k_n:=n\cdot m(k_{n-1})$ and $\cT_{k_n}:=\cT_{n\cdot m(k_{n-1})}$. 

Let us check that the resulting system $\{(k_n,\cT_{k_n})\}_{n\in \bN}$ is cofinal (in both $\bN$ ordered by divisibility, and in the system of unimodular triangulations of the simplices $n\Delta^r$). {Note that the sequence $h_n=n!$ 
is cofinal  for divisibility in $\bN.$}  Clearly we have $n!\mid k_n$, so the sequence $\{k_n\}_{n\in \bN}$ 
is also cofinal in $\bN$ ordered by divisibility. As for the triangulations, consider a pair $(n,\cT)$, where $\cT$ is a unimodular triangulation of the simplex $n\Delta^r$. Then by construction we have both that $n\mid k_{n}\mid k_{n+1}$, and the unimodular triangulation $\cT_{k_{n+1}}$ is a refinement of any unimodular triangulation of ``level'' lower than (i.e. dividing) $k_n$. In particular, since $n\mid k_n$, the given $\cT$ is such a triangulation, so $(n,\cT)\leq (k_{n+1}, \cT_{k_{n+1}})$ in the ordering.

\subsection{{Conclusion of the proof}}

In order to {finish the proof of} Theorem \ref{thm:main} in the local case, what is left to verify is that for every $n$ and every unimodular triangulation $\cT$ of the simplex $n\Delta^r$, the Fourier-Mukai functor $\Db{(\bA^{r+1}_n)_{\cT}}\to \Db{\radice[n]{\bA^{r+1}}}$ is an equivalence. By descent for $\Db{-}$ it is sufficient to show that both statements are true after making a further \'etale base change on the base $\radice[n]{\bA^1}$.

So let us consider the natural \'etale atlas $\bA^1=\Spec k[t_n]\to [\Spec k[t_n] \, /\, \mu_n]=\radice[n]{\bA^1}$, and pullback the whole diagram
$$
\xymatrix{
\radice[n]{\bA^{r+1}}
\ar[rd] & & (\bA^{r+1}_n)_\cT \ar[ld] \\ 
& \bA^{r+1}_n \ar[d] & \\
& \radice[n]{\bA^1} &
}
$$
to the scheme $\bA^1$, obtaining the diagram
$$
\xymatrix{
[{\bA^{r+1}} \, / \, \mu_n^{r}] 
\ar[rd] & & (\bA^{r+1}/\mu_n^r)_\cT \ar[ld] \\ 
& \bA^{r+1}/\mu_n^r \ar[d] & \\
& {\bA^1} &
}
$$
where $\bA^{r+1}=\Spec k[x_{i,n}]$ is the atlas of the root stack $\radice[n]{\bA^{r+1}}=[\bA^{r+1}\, /\, \mu_n^{r+1}]$, and $\mu_n^r$ is acting via the inclusion $\mu_n^r\subseteq \mu_n^{r+1}$ as the kernel of the map $$
\mu_n^{r+1}\to \mu_n, \quad (\xi_1,\hdots,\xi_{r+1})\mapsto \xi_1\cdots \xi_{r+1}.$$ Pulling back along $\bA^1\to \radice[n]{\bA^1}=[\bA^1\, / \,\mu_n]$ removes this $\mu_n$ from the stabilizers of all the stacks in the diagram. 

Note also that, by construction, the log blow-up  $(\bA^{r+1}_n)_\cT\to \bA^{r+1}_n$ restricts to the log blow-up $(\bA^{r+1}/\mu_n^r)_\cT \to \bA^{r+1}/\mu_n^r$ determined by the \emph{same} subdivision of the fan that the toric stack $ \bA^{r+1}_n$ and the toric variety $\bA^{r+1}/\mu_n^r$ share (see Remark \ref{rmk:toric}), which is the fan given by the first orthant in $\bR^{r+1}$, for the lattice $L_n\subseteq \bR^{r+1}$ of points $(a_1,\hdots, a_{r+1})$ in $\bZ^{r+1}$ such that $a_1+\cdots+a_{r+1}\equiv 0 \; (\mathrm{mod} \; n)$.

As explained above, by \'etale descent our local statement reduces to proving that the Fourier-Mukai functor $$\Db{(\bA^{r+1}/\mu_n^r)_\cT}\to \Db{[{\bA^{r+1}} \, / \, \mu_n^{r}]}$$ induced by the structure sheaf of the fiber product $[{\bA^{r+1}} \, / \, \mu_n^{r}]\times_{  \bA^{r+1}/\mu_n^r }(\bA^{r+1}/\mu_n^r)_\cT$ is an equivalence. {This is a special case  of the derived McKay correspondence for abelian quotient singularities, a result proven by Kawamata in \cite{Ka}.} We record this result in the form of the following proposition.

\begin{proposition}\label{prop:mckay}
For every $n$ and every unimodular triangulation $\cT$ of the simplex $n\Delta^r$, the Fourier-Mukai functor $ \Db{(\bA^{r+1}/\mu_n^r)_\cT}\to \Db{[{\bA^{r+1}} \, / \, \mu_n^{r}]}$ is an equivalence of dg-categories.
\end{proposition}

\begin{proof}
We use Theorem 4.2 of \cite{Ka}, case (4). Note that here both stacks are smooth and noetherian, so $\Db{-}=D^b_{dg}(-)$ \cite[Corollary 3.0.5]{BNP}. 

In the notation of the paper, we take $X=\bA^{r+1}/\mu_n^r$, where $\mu_n^r\subseteq \mu_n^{r+1}$ as described above, and $Y= (\bA^{r+1}/\mu_n^r)_\cT$, with $B=C=0$. 
%Denote by $\mathcal{W}$
The normalization of the fiber product of  \emph{schemes}
$
\mathcal{W}:=[{\bA^{r+1}} \, / \, \mu_n^{r}]\times_{\bA^{r+1}/\mu_n^r}(\bA^{r+1}/\mu_n^r)_\cT
$
coincides with the fibered product of fine saturated log schemes, since everything is toric. %and consider the diagram 
%$$
%\xymatrix{
%&  \mathcal{W} \ar[ld]_-q  \ar[rd]^-p & \\  [{\bA^{r+1}} \, / \, \mu_n^{r}] 
%& & (\bA^{r+1}/\mu_n^r)_\cT 
%}
%$$
By Theorem 4.2 of \cite{Ka} the functor
$$
q_* \circ p^*: \Db{(\bA^{r+1}/\mu_n^r)_\cT}\to \Db{[{\bA^{r+1}} \, / \, \mu_n^{r}]}
$$
is an equivalence, where $p\colon \mathcal{W} \to (\bA^{r+1}/\mu_n^r)_\cT$ and $q\colon \mathcal{W} \to [{\bA^{r+1}} \, / \, \mu_n^{r}]$ are the projections. This concludes the proof. %All is left  to prove is that $[{\bA^{r+1}} \, / \, \mu_n^{r}]\times_{\bA^{r+1}/\mu_n^r}(\bA^{r+1}/\mu_n^r)_\cT$ is   normal and therefore coincides with $\mathcal{W}$.
%Consider the fiber product
%$$
%\begin{gathered}
%\label{eq}
%\xymatrix{ Z  \ar[r] \ar[d] & (\bA^{r+1}/\mu_n^r)_\cT  \ar[d] \\
%\bA^{r+1} \ar[r]^-p & \bA^{r+1}/\mu_n^r
%}
%\end{gathered}
%$$
%where $p$ is the quotient map. The morphism of fans corresponding to $p$ is given by a finite-index  embedding of lattices. By Lemma 2.2.7 of \cite{Molcho},  
%$Z$ is a toric variety and thus, in particular, $Z$ is normal. 
%Next consider the commutative  diagram
%$$
%\xymatrix{ Z  \ar[r]^-{g}  \ar[d] & [{\bA^{r+1}} \, / \, \mu_n^{r}]\times_{\bA^{r+1}/\mu_n^r}(\bA^{r+1}/\mu_n^r)_\cT \ar[d] \ar[r] & (\bA^{r+1}/\mu_n^r)_\cT \ar[d] \\
%\bA^{r+1} \ar[r]^-h & [{\bA^{r+1}} \, / \, \mu_n^{r}] \ar[r] & \bA^{r+1}/\mu_n^r
%}
%$$
%and note that both the exterior and rightmost square are fiber products, hence also the leftmost square is cartesian.
%Now, since the map $h$ is \' etale, $g$ is 
%\' etale as well. That is, $Z$ is an  
%\' etale cover of $[{\bA^{r+1}} \, / \, \mu_n^{r}]\times_{\bA^{r+1}/\mu_n^r}(\bA^{r+1}/\mu_n^r)_\cT.$ The fact that $Z$ is normal then implies that $[{\bA^{r+1}} \, / \, \mu_n^{r}]\times_{\bA^{r+1}/\mu_n^r}(\bA^{r+1}/\mu_n^r)_\cT$ is also normal, and concludes the proof. 
\end{proof}

This concludes the proof of Theorem \ref{thm:main} in the local case.

\section{Globalizing}\label{sec:globalizing}

In this section we show how the results obtained {in the previous section} in the local case imply the statement of Theorem \ref{thm:main} also for the global case. Consider a log algebraic stack $X$ as in Section \ref{sec:simple.log.ss}. The strategy for proving the theorem will be to construct, as in the previous section, a system of log blow-ups $(X_{k_n})_{\cI_{k_n}}\to X_{k_n}$ (where for $n\in \bN$, we denote by $X_n$ the root stack $\radice[n]{X'}$, see Section \ref{sec:simple.log.ss}) such that
\begin{enumerate}
\item the set $\{k_n\}$ is cofinal in $\bN$ by divisibility, so that $\irs{X}\to \varprojlim_{n} \radice[k_n]{X}$ is an equivalence,
\item the log blow-ups $(X_{k_n})_{\cI_{k_n}}\to X_{k_n}$ are cofinal among log blow-ups of the stacks $X_n$, so that  the induced map $(X_\infty)^\val\to \varprojlim_n (X_{k_n})_{\cI_{k_n}}$ is an equivalence, and
\item the Fourier-Mukai functors $ \Db{(X_{k_n})_{\cI_{k_n}}}\to \Db{\radice[k_n]{X}}$ are equivalences of dg-categories.
\end{enumerate}
After we achieve this, Theorem \ref{thm:main} will be proven.

\begin{remark}
Note that the complication of carefully constructing a cofinal subsystem of log blow-ups is ultimately due, as in the local case (Remark \ref{rmk:r.1}), to non-uniqueness of crepant resolutions for the local models, outside of the case $r=1$.\end{remark}

We want to argue that the log blow-ups that we obtain in the local model, pulled back to $X$ locally where there is a chart $X\to \bA^{r+1}$, will glue together (provided we set things up correctly), and give a system of global log blow-ups of $X$ that has the properties outlined above. For this it is essential that the triangulations that we constructed in the local case (Lemma \ref{lem:sym}) are symmetric with respect to automorphisms of the monoid $\bN^{r+1}$. {In order to keep track of all the compatibilities between various patches on $X$ we will use the technology of  \emph{Artin fans}. We {recall} definitions and basic properties of these objects in the next section. In Section \ref{sec:colocopro}, using subdivisions of Artin fans, we will construct a system of log blow-ups $(X_{k_n})_{\cI_{k_n}}\to X_{k_n}$ satisfying properties $(1)$, $(2)$ and $(3)$ from above. This will conclude the proof of Theorem \ref{thm:main}.}

\subsection{Artin fans and subdivisions}

A convenient way to encode the compatibility of the local log blow-ups is by considering the \emph{Artin fan} $\cA_X$ of $X$, and its subdivisions. Let us recall briefly what these are. For more about Artin fans and subdivisions we refer to reader to \cite{abramovichi,abramovichb,abramovichs}.

Artin fans are log algebraic stacks that are log \'etale over the base field $k$ (or, equivalently, \'etale over Olsson's stack $LOG_k$). All our Artin fans will also have \emph{faithful monodromy}, i.e. the structure map $\cX\to LOG_k$ will also be representable. 
Every fine saturated log scheme (or algebraic stack) $X$ admits an initial morphism $X\to \cA_X\to LOG_k$ to an Artin fan, called the ``Artin fan of $X$''. For example if $X$ is a $T$-toric variety for a torus $T$, then $\cA_X\simeq [X/T]$ for the action giving the toric structure. Artin fans forget all the geometry of a log scheme or algebraic stack, and only retain the information about the combinatorics of the log structure, in an ``algebro-geometric'' manner.

They are perfectly suited for encoding log blow-ups, since these are algebro-geometric operations, but completely determined by the combinatorics. Just as for toric varieties toric blow-ups correspond to subdivisions of the fan, for general fine saturated log stacks, (representable) log blow-ups correspond to ``subdivisions'' of the Artin fan. See the discussion in \cite[Section 2.4]{abramovichi} and \cite[Section 3]{abramovichb} for details.

Given a rational polyhedral cone $\sigma$ in a finite dimensional vector space $N\otimes_\bZ \bR$ (where $N$ is a lattice, dual to $M$), we will denote by $\cA_\sigma:=[\Spec k[\sigma^\vee\cap M]\,/ \,T]$ the Artin fan of the affine toric variety $\Spec k[\sigma^\vee\cap M]$ (here $T$ is the algebraic torus $\Hom(M,\bG_m)$ with character lattice $M$). Artin fans of the form $\cA_\sigma$ are called \emph{Artin cones}.

If $\tau\subseteq \sigma$ is the inclusion of a face, then we naturally have an induced map $\cA_\tau\to \cA_\sigma$, which is a Zariski open immersion. Given a (rational) subdivision $\Sigma$ of the cone $\sigma$, we can then glue together the various stacks $\cA_\tau$ for the cones $\tau \in \Sigma$, and obtain an Artin fan $\cA_\Sigma$ (which can also be described as the global quotient of the toric variety $X_\Sigma$ associated to $\Sigma$ by its torus) with a map $\cA_\Sigma\to \cA_\sigma$. This is called a subdivision of $\cA_\sigma$, and is a log blow-up in the sense of (\ref{sec:valuativization}).

More generally, a morphism of Artin fans $\cA'\to \cA$ is a subdivision if for every Artin cone $\cA_\sigma$ and every morphism of Artin fans (i.e. of log algebraic stacks) $\cA_\sigma \to \cA$, the projection $\cA'\times_\cA \cA_\sigma \to \cA_\sigma$ is isomorphic to a subdivision. Moreover, it is enough to check this for maps $\cA_\sigma\to \cA$ that are part of a fixed \'etale cover of $\cA$.

%\begin{remark}
{We will make use of the following  observation.}  Assume that we have an Artin fan $\cA$, together with an \'etale cover $\{\cA_{\sigma_i}\to \cA\}$ by Artin cones, and for every cone $\sigma_i$ we have a subdivision $\Sigma_i$. Assume also that the subdivisions $\Sigma_i$ are symmetric with respect to automorphisms of $\sigma_i$, and that whenever for some $i$ and $j$ we have a face $\tau$ of both $\sigma_i$ and $\sigma_j$ with a commutative diagram
$$
\xymatrix{
\cA_\tau\ar[r]\ar[d] & \cA_{\sigma_i}\ar[d]\\
\cA_{\sigma_j}\ar[r] & \cA,
}
$$
then the subdivisions $\Sigma_i$ and $\Sigma_j$ coincide when restricted to $\tau$. Then there is a subdivision $\cA_{\Sigma}\to \cA$, that restricts to $\cA_{\Sigma_i}\to \cA_{\sigma_i}$ on the given \'etale cover. This is used also in \cite[Section 3]{abramovichb} (see Section 3.2 in particular).

Here is a sketch of an argument to justify this: 
our Artin fan $\cA$ will have a presentation as the colimit of a diagram $\cR\rightrightarrows \cU$ where both $\cR$ and $\cU$ are disjoint unions $\cU=\bigsqcup_i \cA_{\sigma_i}$ and $\cR=\bigsqcup_j \cA_{\tau_j}$ of Artin cones as above. The subdivisions $\Sigma_i$ and their restriction to the faces $\tau_j$ will give subdivisions $\cU'\to \cU$ and $\cR'\to \cR$, and by our assumptions there is an induced diagram $\cR'\rightrightarrows \cU'$, compatible with $\cR\rightrightarrows \cU$, and whose colimit in \'etale sheaves over $LOG_k$ gives the desired subdivision $\cA_\Sigma$ of $\cA$.
%\end{remark}

Note that compatibility and symmetry of the subdivisions $\Sigma_i$ is fundamental for the above argument. For example, for a fixed cone $\sigma_i$ appearing in $\cU$, in $\cR$ we could have $\cA_{\sigma_i}$ itself, with one of the maps of the groupoid $\cR\rightrightarrows \cU$ restricting to the identity $\cA_{\sigma_i}=\cA_{\sigma_i}$ and the other one restricting to a non-trivial automorphism $\cA_{\sigma_i}{\simeq}\cA_{\sigma_i}$ induced by a non-trivial symmetry of the cone $\sigma_i$. Only thanks to the symmetry of $\Sigma_i$ can we conclude that the groupoid structure will lift to $\cR'$ and $\cU'$.

\subsection{Construction of the log blow-ups $(X_{k_n})_{\cI_{k_n}}$ and conclusion of the proof}
\label{sec:colocopro}

Consider the Artin fan $\cA_X$ of our fixed log 
algebraic stack $X$. By the explicit form of our local charts $X\to \bA^{r+1}$, the Artin fan $\cA_X$ has an \'etale cover by Artin cones of the form $$[\bA^{r+1}/\bG_m^{r+1}]=[\bA^1/\bG_m]\times \cdots \times [\bA^1/\bG_m]$$ for various possible values of $r$. The ``intersections'' of these are Artin cones of the same form, with lower values for $r$.

In order to define a subdivision of $\cA_X$ (or of the Artin fan $\cA_{X_n}$ of the stack $X_n$, that has a similar form), we can specify one subdivision for each one of the cones  corresponding to  the toric varieties $\bA^{r+1}$, provided these are compatible over the (possibly self-)intersections. If the subdivisions are chosen to be symmetric with respect to automorphisms of the monoid $\bN^{r+1}$ and compatible for different values of $r$, these conditions will be automatically satisfied. This is  the reason why we made an effort to construct symmetric unimodular triangulations in Lemma \ref{lem:sym}.

{For clarity, we break down our construction of a suitable cofinal system in the following  three steps: } 
\begin{itemize}
\item {Since $X$ is quasi-compact, there will be a  maximal value for the ranks of the stalks $\overline{M}_{X,x}$ of the sheaf $\overline{M}_X$.\footnote{We are assuming here that $X$ has a non-trivial log structure, and that therefore the maximal rank of the stalks is greater than zero.}    Denote it by $R+1$. Let  $\{(k_n,\cT_{k_n})\}_{n\in\bN}$ be the sequence of positive integers $k_n$ and symmetric unimodular triangulations $\cT_{k_n}$ of $k_n\Delta^{R}$ that was constructed in Section \ref{sec:local}.} 

{Observe that by restricting the sequence of triangulations $\{(k_n,\cT_{k_n})\}_{n\in\bN}$ to any $r$-dimensional face $k_n\Delta^{r}$ of $k_n\Delta^{R}$, we obtain  a cofinal sequence of symmetric rational triangulations of the simplex $\Delta^r$, see Lemma \ref{lem:sym2}. Also note that, since the triangulation $\cT_{k_n}$ is symmetric with respect to permutations of the vertices, the choice of how to realize $k_n\Delta^{r}$ as a face of $k_n\Delta^{R}$ is irrelevant.}

\item Let $n \in \bN$, and consider the stack $X_{k_n}$ and its Artin fan $\cA_{X_{k_n}}$. This has an \'etale cover by Artin cones of the form $\cA_{\sigma^r_{k_n}}$, where for $m \in \bN$ the cone $\sigma^r_m$ is the cone of the toric variety $\bA^{r+1}/\mu_m^r$ of Section \ref{sec:local}, given by the first orthant $\bR_{\geq 0}^{r+1}$ in $\bR^{r+1}$, with respect to the lattice $L_m=\{(a_1,\hdots, a_{r+1})\in \bZ^{r+1} \mid a_1+\cdots +a_{r+1}\equiv 0\; (\mathrm{mod} \; m)\}$. We stress that here $r$ is not fixed, but can take any value from $0$ to $R$.

We define a subdivision of $\cA_{X_{k_n}}$, by subdividing each $\cA_{\sigma^r_{k_n}}$ according to the triangulation $\cT_{k_n}$. Since these triangulations are symmetric and compatible with restriction to faces (by construction), this will define a global subdivision of $\cA_{X_{k_n}}$, that we denote by $\cA_{\cT_{k_n}}\to \cA_{X_{k_n}}$. 
\item By pulling this back along $X_{k_n}\to \cA_{X_{k_n}}$, we obtain a log blow-up $(X_{k_n})_{\cT_{k_n}}\to X_{k_n}$. Moreover, if $n\leq n'$ we  have a map $(X_{k_{n'}})_{\cT_{k_{n'}}}\to (X_{k_{n}})_{\cT_{k_n}}$ covering $X_{k_{n'}}\to X_{k_n}$, so that $\{(X_{k_{n}})_{\cT_{k_n}}\}_{n \in \bN}$ is an inverse system of log blow-ups of $\{X_{k_n}\}_{n\in\bN}$.  
\end{itemize}

{To conclude the proof we need to check that 
$\{(X_{k_{n}})_{\cT_{k_n}}\}_{n \in \bN}$ satisfies the properties $(1)$, $(2)$ and $(3)$ listed at the beginning of this section. Note that property  $(1)$ is clearly satisfied. Indeed, since $\{k_n\}$ is cofinal in $\bN$ by divisibility, the natural map $\irs{X}\to \varprojlim_n \radice[k_n]{X}$ is an equivalence. As a consequence $\varinjlim_n \Db{\radice[k_n]{X}}\to \Db{\irs{X}}$ is an equivalence of dg-categories. Proposition \ref{prop:global.iso.val} below shows that property $(2)$ also holds.}

\begin{proposition}\label{prop:global.iso.val}
We have a natural isomorphism $X_\infty^\val\simeq \varprojlim_n (X_{k_{n}})_{\cT_{k_n}}$, that induces an equivalence of dg-categories $\Db{X_\infty^\val}\simeq \varinjlim_n \Db{(X_{k_{n}})_{\cT_{k_n}}}$.
\end{proposition}

\begin{proof}
Since every $(X_{k_{n}})_{\cT_{k_n}}$ is a log blow-up of $X_{k_n}$, we have a canonical map $(X_\infty)^\val\to  \varprojlim_n (X_{k_{n}})_{\cT_{k_n}}$. In order to check that this is an isomorphism we can localize on $X$ where we have a chart $X\to \bA^{r+1}$ for the log structure. We can also assume that the log structure of $X$ is Zariski, and that the locus in $X$ where the stalk $\overline{M}_{X,x}$ has maximal rank is irreducible (so that $X$ is a \emph{small} logarithmic scheme, see for example \cite[Definition 5.6]{abramovichs}).

Under these conditions, for every $m\in \bN$ the induced strict map $X_m\to \bA^{r+1}_m$ identifies the Artin fan $\cA_{X_m}$ with the Artin cone $\cA_{\sigma_m^r}$ (where $\sigma_m^r$ is the cone of the toric variety $\bA^{r+1}/\mu_m^r$, as in the discussion above). By construction of the log blow-ups $(X_{k_n})_{\cT_{k_n}}$ we have canonical (and compatible) isomorphisms
$$
(X_{k_n})_{\cT_{k_n}}\simeq X_{k_n}\times_{\bA^{r+1}_{k_n}} (\bA^{r+1}_{k_n})_{\cT_{k_n}}
$$
for every $n$. Therefore for the limit we find
$$
\varprojlim_n (X_{k_n})_{\cT_{k_n}}\simeq \varprojlim_n (X_{k_n}\times_{\bA^{r+1}_{k_n}} (\bA^{r+1}_{k_n})_{\cT_{k_n}})\simeq X_\infty\times_{\bA^{r+1}_\infty} (\bA^{r+1}_\infty)^\val
$$
and since forming valuativizations is also compatible with strict base change (Remark \ref{rmk:val.base.change}), the last term is canonically isomorphic to $(X_\infty)^\val$. From this it follows that the natural map $(X_\infty)^\val\to  \varprojlim_n (X_{k_{n}})_{\cT_{k_n}}$ is an isomorphism, as we wanted.
\end{proof}

{The last thing left to prove is that property $(3)$ holds for $\{(X_{k_{n}})_{\cT_{k_n}}\}_{n \in \bN}$. That is, we have to show that  the Fourier-Mukai functors 
$
\Db{(X_{k_n})_{\cT_{k_n}}}\to \Db{\radice[k_n]{X}} 
$  are equivalences. We prove first a couple of lemmas.}

\begin{lemma}\label{lem:fibered.prod}
Let
$$
\xymatrix{
Y'\ar[r]\ar[d] & Y\ar[d]\\
X'\ar[r] & X
}
$$
be a cartesian diagram of perfect algebraic stacks (in the sense of \cite {BFN}), where the map $X'\to X$ is flat. Then the natural functor of dg-categories
$$
\Db{Y}\otimes_{\Db{X}} \Db{X'} \to \Db{Y'}
$$
is an equivalence.
\end{lemma}
\begin{proof}
The statement in the case of quasi-coherent sheaves is proved in Theorem 4.7 of \cite{BFN}. This restricts to an equivalence of categories between compact objects, i.e. perfect complexes, as explained in the introduction of \cite{BNP}. 
\end{proof}

\begin{lemma}
Let $Y\to  X$, $Z \to X$ and $f\colon X' \to X$ be morphisms of perfect algebraic stacks, and assume that $f$ is flat. Let $Y'$ and $Z'$ be the fiber  products $Y\times_X X'$ and $Z\times_X X'$. Denote by $$\Phi\colon \Db{Y}\to \Db{Z}$$ the Fourier-Mukai functor defined by the structure sheaf of $Y\times_X Z$, and by $$\Phi'\colon \Db{Y'}\to \Db{Z'}$$ the one defined by the structure sheaf of $Y'\times_{X'} Z'$.

Then the functor $\Phi\otimes \id_{ \Db{X'} } \colon \Db{Y'}\to \Db{Z'}$ induced by $\Phi$ and the equivalences of dg-categories
$$
\Db{Y'} \stackrel{\simeq} \to \Db{Y}\otimes_{\Db{X}} \Db{X'} 
$$
and
$$ 
\Db{Z'} \stackrel{\simeq}\to 
\Db{Z}\otimes_{\Db{X}} \Db{X'} 
$$
of the previous lemma, is equivalent to the Fourier-Mukai functor $\Phi'$.
\end{lemma}
\begin{proof}
Let us write the functor $\Phi\otimes \id$ as a Fourier-Mukai transform encoded {via pull-push} in a span diagram. The span diagram is given by 
$$
\xymatrix{&(Y \times_X Z) \times_X X' \ar[ld]_-{q_1} \ar[rd]^-{q_2} & \\
Y'&&Z'}
$$
where $q_1$ and $q_2$ are the composites
$$
q_1:(Y \times_X Z) \times_X X' \to Y\times_X X' = Y', \quad  
q_2: (Y \times_X Z) \times_X X' \to Z\times_X X'= Z' , 
$$
and there is a natural equivalence
$
  \Phi \otimes \id \simeq    q_{2*}\circ q_1^*    .
$
Then the lemma follows immediately from the fact that fiber products commute with fiber products, namely
\belowdisplayskip=-14pt
$$
(Y \times_X Z) \times_X X' \simeq 
(Y \times_X X') \times_{X \times_X X'} (Z \times_X X') \simeq Y' \times_{X'} Z'. 
$$
\end{proof}
We are now ready to show 
that  $\{(X_{k_{n}})_{\cT_{k_n}}\}_{n \in \bN}$ satisfies also property $(3).$ 

\begin{proposition}\label{prop:fm.local}
Let $X$ be as in Theorem \ref{thm:main}, and {let $\{(X_{k_n})_{\cT_{k_n}}\}_{n \in \bN}$ be as {in the discussion} above.} Then for every $n \in \bN$, the Fourier-Mukai functor
$$
\Phi_{(n_k, \cT_{k_n})}\colon \Db{(X_{k_{n}})_{\cT_{k_{n}}}}\to \Db{\radice[k_n]{X}}
$$
is an equivalence of dg-categories.
\end{proposition}

\begin{proof}
By descent for $\Db{-}$ and the assumptions, we can assume that we have a flat chart $X\to \bA^{r+1}$ for the log structure and that $X$ is an affine scheme.

For every $n\in \bN$, the map $X_{k_n}\to \bA^{r+1}_{k_n}$ is then a strict and flat (since it is a base change of $X\to \bA^{n+1}$, that is flat by assumption) morphism of log algebraic stacks. By the previous lemmas, since $(X_{k_n})_{\cT_{k_n}}$ and $\radice[k_n]{X}$ are both base change of the corresponding objects over $\bA^{r+1}_{k_n}$, the Fourier-Mukai functor $\Phi_{(n_k, \cT_{k_n})}$ is base change of the Fourier-Mukai functor $\Db{(\bA^{r+1}_{k_n})_{\cT_{k_n}}}\to \Db{\radice[k_n]{\bA^{r+1}}}$, which is an equivalence by the results of Section \ref{sec:local} (in particular Proposition \ref{prop:mckay}). This concludes the proof.

Note that it is essential in this argument that the global log blow-up $(X_{k_n})_{\cT_{k_n}}\to X_{k_n}$ is constructed by globalizing the local construction of Section \ref{sec:local}.
\end{proof}

This concludes the proof of Theorem \ref{thm:main} in the general case.

\section{Simple log semistable families}\label{sec:simple.log.sss}

In this final section we will assume that $X$ is the total space of a simple kind of semistable degeneration over a base $S$. We will show that in this situation, if the base $S$ is log smooth, the derived equivalence of Theorem \ref{thm:main} restricts to the fibers of this degeneration.

This applies in particular to the version of the McKay correspondence that was used in Section \ref{sec:local}: the toric variety $\bA^{r+1}$ can be equipped with the map $\phi_r\colon \bA^{r+1}\to \bA^1$ sending $(x_1,\hdots,x_{r+1})$ to $x_1\cdots x_{r+1}$, and it becomes the total space of a flat degeneration of $\bG_m^r$ to the union of the coordinate hyperplanes in $\bA^{r+1}$.

The ``finite index'' statement in this case is that the derived equivalence of Proposition \ref{prop:mckay} restricts to a derived equivalence on the fibers over $0\in \bA^1$. This was already observed in \cite{BP}, and in fact we will leverage the results of that paper in our proof of Theorem \ref{thm:main1} below.

\begin{definition}\label{def:simple.log.ss.1}
Let $X, S$ be fine saturated log algebraic stacks over $k$. A \emph{simple log semistable morphism} is a vertical log smooth morphism $f\colon X\to S$ such that for every geometric point $x\to X$ with image $s=f(x)\to S$ on which the log structure of $S$ is not trivial, there are isomorphisms $\overline{M}_{S,s}\cong \bN$ and $\overline{M}_{X,x}\cong \bN^{r+1}$ with $r\geq 0$, such that the map $\overline{M}_{S,s}\to \overline{M}_{X,x}$ is identified with the $r$-fold diagonal map $\bN\to \bN^{r+1}$.
\end{definition}

\begin{remark}
Recall that a morphism of log algebraic stacks $X\to S$ is vertical if for every geometric point $x\to X$ with image $s\to S$ the cokernel of the homomorphism $\overline{M}_{S,s}\to \overline{M}_{X,x}$ is a group. The verticality condition implies that the locus of $X$ where the log structure is trivial is the preimage of the corresponding locus of $S$. Note also that the cokernel of the map $\bN\to \bN^{r+1}$ is a group (isomorphic to $\bZ^r$).
\end{remark}

\begin{remark}
Compare with the definition of ``essentially semistable'' morphisms in \cite[Definition 2.1]{olsson2}, where there can be extra factors $\bN$ in the log structure of the total space and the base, and on which $f$ acts as the identity, pointwise. 
\end{remark}

\begin{lemma}\label{def:simple.log.ss}
Let $X,S$ be fine saturated log algebraic stacks over $k$.  A morphism $f\colon X\to S$ is a simple log semistable morphism if and only if strict-smooth locally on $X$ and $S$ around points where the log structures are non-trivial there exists a diagram
$$
\xymatrix{
X\ar[rd]\ar@/_1pc/[rdd]_f & &\\
& S'\ar[r]\ar[d] & \bA^{r+1} \ar[d]^{\phi_r} \\
& S\ar[r] & \bA^1
}
$$
where the square is cartesian, the horizontal maps are strict, the diagonal map is strict and smooth, and $\phi_r\colon \bA^{r+1}\to \bA^1$ is the toric morphism induced by the $(r+1)$-fold diagonal map $\bN\to \bN^{r+1}$ (as in Section \ref{sec:local}).

If moreover $S$ is log smooth, then there locally exists a diagram as above, and such that the horizontal maps are also smooth.
\end{lemma}

\begin{proof}
Clearly every morphism that admits charts as in the statement satisfies the conditions of Definition \ref{def:simple.log.ss.1}.
The converse follows from the fact that we can find local charts around a geometric point $x\to X$ and around its image $s\to S$ using the monoids $\overline{M}_{X,x}$ and $\overline{M}_{S,s}$ and the homomorphism $\overline{M}_{S,s}\to \overline{M}_{X,x}$. See the proof of \cite[Lemma 2.2]{olsson2} for details. Finally, the last claim follows from existence of charts as in the previous sentence, and from the fact that, because of log smoothness of $S$, the chart morphism $S\to \bA^1$ can be chosen to be smooth.
%that if $S$ is log smooth over $k$, then every chart morphism $S\to \bA^1$ obtained from stalks $\overline{M}_{S,s}\cong \bN$ is smooth. This can be checked as follows: we can assume that $S$ is a scheme, and since $S$ is log smooth over $k$, \'etale locally around every point $s\in S$ there is some fine saturated monoid $P$ and a strict and smooth map $\varphi\colon S\to \Spec k[P]$. Moreover, since $\overline{M}_{S,s}\cong \bN$ there is a surjective homomorphism of monoids $P\to \bN$, and choosing a section $\bN\to P$ we obtain a morphism $\Spec k[P]\to \bA^1$, such that the composite $S\to \bA^1$ is strict at $s$. Therefore $\Spec k[P]\to \bA^1$ is also a strict morphism at the point $\varphi(s)$, and this is easily seen to imply that $\Spec k[P]\to \bA^1$ is smooth around $\phi(s)$. Therefore, after possibly localizing further on $S$, the resulting morphism $S\to \bA^1$ is smooth at $s$.
\end{proof}

\begin{example}\mbox{ }
\label{exexex}
\begin{itemize}
\item The toric maps $\phi_r\colon \bA^{r+1}\to \bA^1$ of the definition, given explicitly by the formula $(x_1,\hdots, x_{r+1})\mapsto x_1\cdots x_{r+1}$ are simple log semistable morphisms (and this is a sort of ``universal'' case).
\item The degeneration of elliptic curves to a nodal rational curve of Example \ref{example:elliptic}, where base and total space are equipped with the natural log structures, is a simple log semistable morphism. 
\item Given any smooth variety $X$ with a normal crossings divisor $D\subseteq X$, we have a morphism $X\to [\bA^1/\bG_m]$, given by the divisor $\cO(D)$ with its tautological global section (recall that $[\bA^1/\bG_m]$ can be seen as the stack parametrizing line bundles with a global section). This is a simple log semistable morphism. 
\end{itemize}
\end{example}

Note that if $f\colon X\to S$ is a simple log semistable morphism and $X$ has affine stabilizer groups, then we are in the situation of Section \ref{sec:simple.log.ss}, and we can apply Theorem \ref{thm:main} to the total space $X$. Moreover, in this case the stacks $X_n=\radice[n]{X'}$ are isomorphic to the fibered products $X\times_S \radice[n]{S}$, and analogously at the limit we have $X_\infty\simeq X\times_S \irs{S}$.

\begin{theorem}\label{thm:main1}
Let $X, S$ be fine saturated log algebraic stacks over $k$, with $S$ log smooth and $X$ with affine stabilizer groups, and let $f\colon X\to S$ be a quasi-compact simple log semistable morphism of Deligne--Mumford type (i.e. for a scheme $T$ and a map $T\to S$, the base change $X\times_S T$ is a Deligne--Mumford stack).

Then for every geometric point $\xi \to \irs{S}$, the equivalence $\Phi$ of Theorem \ref{thm:main} restricts to an equivalence of dg-categories on the fibers $\Phi_\xi \colon  \Db{(X_\infty^\val)_\xi}\to \Db{(\radice[\infty]{X})_\xi}$.
\end{theorem}

\begin{remark}
{The log smoothness and ``of Deligne--Mumford type'' assumptions in Theorem \ref{thm:main1} are included in order to use results of \cite{BP} about restricting Fourier-Mukai functors (see below for more details), that are stated for families whose total space is a smooth Deligne--Mumford stack. It is possible that these assumptions can be relaxed.}
\end{remark}

\begin{remark}\label{rmk:trivial.fibers}
Note that the statement is obvious if $\xi\to \irs{S}\to S$ is in the open locus where the log structure of $S$ is trivial, because in that case both fibers $(\radice[\infty]{X})_\xi$ and $(X_\infty^\val)_\xi$ are canonically isomorphic to the fiber $X_\xi$ of $f\colon X\to S$.
\end{remark}
Proving Theorem \ref{thm:main1} will require some preparation. We start by noting that with our hypotheses, in the local charts of Lemma \ref{def:simple.log.ss} we can assume that the maps $S\to \bA^1$ and $X\to S'$ (and therefore also the composite $X\to \bA^{r+1}$) are smooth in the classical sense. Note that this assures that both the root stacks $\radice[k_n]{X}$ and the log blow-ups $(X_{k_n})_{\cT_{k_n}}$ (constructed in the previous section) are smooth for every $n$, {because they locally have a smooth map towards $\radice[k_n]{\bA^{r+1}}$ and $(\bA^{r+1}_{k_n})_{\cT_{k_n}}$ respectively, and these are smooth.} 

Now because of smoothness of $S\to \bA^1$ and since the statement is local on $S$, we can assume that $S\cong \bA^q\times \bA^1$ for some $q$, with the map $S\to \bA^1$ being the projection.
Note also that, as mentioned in Remark \ref{rmk:trivial.fibers}, since the log structure of $\bA^1$ is trivial outside of the origin, the only points over which there is something to prove are the ones in the preimage of $0\in \bA^1$. Let us fix a point $s\in S$ in this preimage (so that $s$ is of the form $(s',0)$ with $s'\to \bA^q$), and denote by $\xi\to \irs{S}$ the unique lift of $s$ to $\irs{S}$.

We have to show that if we restrict our modified families to the fiber over $\xi$, the equivalence $\Db{X_\infty^\val}\simeq \Db{\radice[\infty]{X}}$ will restrict to an equivalence of the categories of perfect complexes of the fibers (in a sense to be specified).

For this purpose, we will use the following result, that is an immediate generalization of (part of) \cite[Theorem 1.1]{BP}.

\begin{proposition}\label{thm:bp}
Let $\cX$ and $\cY$ be smooth Deligne--Mumford stacks, $f\colon \cX\to \bA^N$ and $g\colon \cY\to \bA^N$ two morphisms, and $\cF$ a complex of coherent sheaves on the fibered product $\cZ:=\cX\times_{\bA^N}\cY$, with proper support over $\cX$ and $\cY$.

If the Fourier-Mukai functor $\Phi_{\cF}\colon \Db{\cX}\to \Db{\cY}$ corresponding to $\cF$ is an equivalence, then the pullback of its kernel to $\cX_p\times \cY_p$ (the fibers over a geometric point $p\in \bA^N$ of $f$ and $g$) also induces an equivalence $\Phi_p\colon \Db{\cX_p}\to \Db{\cY_p}$.
\end{proposition}

\begin{proof}
The proof in \cite[Section 2]{BP} generalizes word for word, except that the embeddings $\cX_p\to \cX$ and $\cY_p\to \cY$ are now regular embeddings of codimension $N$ (that still have finite tor dimension). Note also that in that paper things are stated for the bounded derived category, but the result holds also for the category of perfect complexes. 
\end{proof}

\begin{remark}
\label{rem:intert} 
{Note that the two Fourier-Mukai functors 
$\Phi_\cF$ and $\Phi_p$ are  {compatible with} the push-forwards along the closed immersions 
$
i: \cX_p\to \cX, \, \, j: \cY_p\to \cY.$  
That is, there is a natural equivalence 
$
\Phi \circ i_* \simeq j_* \circ \Phi_p. 
$
We refer the reader to \cite[Lemma 2.6]{BP} for a proof of this fact.}
\end{remark}

We will apply Proposition \ref{thm:bp} to diagrams of the form
$$
\xymatrix{
\radice[n]{X}\ar[rd] & & (X_n)_\cI\ar[ld]\\
&X_n\ar[d] &\\
& S\cong \bA^{q}\times \bA^1 &  
}
$$
where as usual $(X_n)_\cI$ denotes a log blow-up ({which will be smooth}), and where we will take $\cF$ to be the pushforward to $\radice[n]{X}\times_{\bA^{q+1}} (X_n)_\cI$ of the structure sheaf of the fibered product $\radice[n]{X}\times_{X_n} (X_n)_\cI$.

The proof of Theorem \ref{thm:main1} will involve again  a limit argument, and will make  use of the sequence $\{(k_n,\cT_{k_n})\}_{n\in \bN}$ of unimodular triangulations and the corresponding log blow-ups $(X_{k_n})_{\cT_{k_n}}$ that were constructed in Section \ref{sec:globalizing}. Denote by $(\irs{X})_\xi$ and $(X_\infty^\val)_\xi$ the fibers of $\irs{X}\to \irs{S}$ and $X^\val_\infty\to \irs{S}$ at the point $\xi\to \irs{S}$. Note that for every $n$ we also have a unique point $\xi_n\to \radice[n]{S}$ over $s\to S$, and we can consider the fibers $(\radice[n]{X})_{\xi_n}$ and $((X_{k_n})_{\cT_{k_n}})_{\xi_{k_n}}$.

\begin{lemma}
We have canonical isomorphisms $$(\irs{X})_\xi\simeq \varprojlim_n (\radice[k_n]{X})_{\xi_{k_n}}$$ and $$(X^\val_\infty)_\xi \simeq \varprojlim_n ((X_{k_n})_{\cT_{k_n}})_{\xi_{k_n}}$$ that induce equivalences of dg-categories
$$
\Db{(\irs{X})_\xi}\simeq \varinjlim_n \Db{(\radice[k_n]{X})_{\xi_{k_n}}}
$$
and 
$$
\Db{(X^\val_\infty)_\xi}\simeq \varinjlim_n \Db{((X_{k_n})_{\cT_{k_n}})_{\xi_{k_n}}}.
$$
\end{lemma}

\begin{proof}
This follows from the fact that the point $\xi\to \irs{S}$ is the inverse limit of the points $\xi_n\to \radice[n]{S}$ (in the obvious sense), together with the cofinality properties of the sequence $\{(k_n,\cT_{k_n})\}_{n\in \bN}$.
\end{proof}

To prepare some notation for the proof of {Proposition \ref{proprop} below}, note that the root stacks $\radice[n]{S}=\bA^q\times \radice[n]{\bA^1}$ have compatible \'etale presentations, given by $\id\times \pi_n \colon \overline{S}_n:=\bA^q\times {\bA^1}\to \bA^q\times \radice[n]{\bA^1}$, where $\pi_n \colon \bA^1\to \radice[n]{\bA^1}$ is the usual presentation. Let us denote by $\cX_{k_n}$ and $\cY_{k_n}$ the pullbacks of $\radice[k_n]{X}$ and $(X_{k_n})_{\cT_{k_n}}$ to $\overline{S}_{k_n}$. By descent for $\Db{-}$, for every $n$ the equivalence
$
\Db{(X_{k_n})_{\cT_{k_n}}}\to \Db{\radice[k_n]{X}}
$ 
induces an equivalence
$$
F_{n}\colon \Db{\cY_{k_n}}\to \Db{\cX_{k_n}}.  
$$ 
{Further, by} Proposition \ref{prop:FM.compatible}, these {equivalences} are compatible with pullbacks along $\cX_{k_{n'}}\to \cX_{k_n}$ and $\cY_{k_{n'}}\to \cY_{k_n}$.

We denote by $p$ the unique point of the atlas $\overline{S}_{k_n}=\bA^q\times \bA^1$ in the preimage of $s\in \bA^q\times \bA^1$ with respect to the composite $\overline{S}_{k_n}\to \radice[k_n]{S}\to S=\bA^q\times \bA^1$. By Proposition \ref{thm:bp} the functor $F_{n}$ will induce an equivalence of dg-categories
$$
(F_{n})_p \colon \Db{(\cY_{k_n})_p}\to \Db{(\cX_{k_n})_p}.
$$
With slight abuse of notation, we will denote by $p$ any such point, independently on the index $n$.

\begin{proposition} \mbox{ }
\label{proprop} \begin{enumerate}
\item 
{The equivalence $\Db{(X_{k_n})_{\cT_{k_n}}}\to \Db{\radice[k_n]{X}}$  restricts to an equivalence on the fibers $$ \Db{((X_{k_n})_{\cT_{k_n}})_{\xi_{k_n}}}\to  \Db{(\radice[k_n]{X})_{\xi_{k_n}}}.$$ }
\item For every $n', n\in \bN$ with $n'\geq n$ the diagram of dg-categories
$$
\xymatrix{
\Db{((X_{k_n})_{\cT_{k_n}})_{\xi_{k_n}}} \ar[r]^-\simeq \ar[d] & \Db{(\radice[k_n]{X})_{\xi_{k_n}}}\ar[d] \\
\Db{((X_{k_{n'}})_{\cT_{k_{n'}}})_{\xi_{k_{n'}}}}\ar[r]^-\simeq & \Db{(\radice[k_{n'}]{X})_{\xi_{k_{n'}}}}
}
$$
is commutative.
\end{enumerate}
\end{proposition}

\begin{proof}
{Claim $(1)$ follows immediately by  applying Proposition \ref{thm:bp}  {after pulling back} to {the presentation $\overline{S}_n=\bA^{q+1}$} of $\radice[n]{S}$. This shows that the equivalence $\Db{(X_{k_n})_{\cT_{k_n}}}\to \Db{\radice[k_n]{X}}$  induces an equivalence $ \Db{((X_{k_n})_{\cT_{k_n}})_{\xi_{k_n}}}\to  \Db{(\radice[k_n]{X})_{\xi_{k_n}}}.$} 

{Let us prove next Claim $(2)$, that is,} that these equivalences for different values of $n$ are compatible, as in Proposition \ref{prop:FM.compatible}. 
As usual, by descent for $\Db{-}$ this can be checked after passing to the \'etale presentations $\overline{S}_{k_n}\to \radice[k_n]{S}$ and $\overline{S}_{k_{n'}}\to \radice[k_{n'}]{S}$, so that instead of the diagram in the statement we can look at the diagram
$$
\xymatrix@C=2cm{
\Db{( \cY_{k_n})_p} \ar[r]^{(F_n)_p}\ar[d]_{v_p^*} & \Db{(\cX_{k_n})_p}\ar[d]^{w_p^*} \\
\Db{( \cY_{k_{n'}})_p}\ar[r]^{(F_{n'})_p} & \Db{(\cX_{k_{n'}})_p}
}
$$
where the notation is explained above, except for the vertical maps, that are pullbacks along $v_p\colon (\cY_{k_{n'}})_p\to (\cY_{k_{n}})_p$ and $w_p \colon (\cX_{k_{n'}})_p\to (\cX_{k_{n}})_p$.

For every $n$, let us denote by $j_n\colon (\cX_{k_n})_p\to\cX_{k_n}$ and $i_n\colon (\cY_{k_n})_p\to \cY_{k_n}$ the inclusions of the  fibers. Note also that the maps $v_p$ and $w_p$ between the fibers are not  induced by $v\colon \cY_{k_{n'}} \to  \cY_{k_{n}}$ and $w\colon  \cX_{k_{n'}} \to  \cX_{k_{n}}$ via base change. Rather, there are commutative diagrams
$$
\begin{gathered}
\xymatrix{
(\cY_{k_{n'}})_p \ar[r]^ {v_p} \ar[d]_{i_{n'}} & (\cY_{k_{n}})_p\ar[d]^{i_{n}}  & (\cX_{k_{n'}})_p\ar[r]^ {w_p} \ar[d]_{j_{n'}} &  (\cX_{k_{n}})_p\ar[d]^{j_{n}} \\
\cY_{k_{n'}} \ar[r]^ {v} \ar[d] & \cY_{k_{n}}\ar[d] & 
 \cX_{k_{n'}} \ar[r]^ {w} \ar[d] & \cX_{k_{n}}\ar[d]\\
\bA^q\times \bA^1\ar[r] &\bA^q\times \bA^1 & \bA^q\times\bA^1\ar[r] & \bA^q\times\bA^1
}
\end{gathered}
$$
where the map $\bA^q\times\bA^1\to \bA^q\times\bA^1$ is given by $\id\times (t\mapsto t^d)$ for $d=k_{n'}/k_{n}$, and the squares are not cartesian. The preimage of the point $p$ in $\bA^q\times \bA^1$ consists of exactly one point, but is non-reduced, and the point $p$ itself gives a section (corresponding to the reduction of the fiber).

Let us prove that there is an equivalence 
$$
(w_p)^* \circ (F_n)_p  \simeq (F_{n'})_p \circ (v_p)^* : \Qcoh((\cY_{k_{n}})_p) \rightarrow \Qcoh(  (\cX_{k_{n'}})_p).   
$$
After this, we can pass to $\Db{-}$ by restricting to compact objects. 
In the last display, and throughout the proof, we will denote in the same way a functor and its Ind-completion (e.g. we write  $(F_n)_p$ instead of $\mathrm{Ind}((F_n)_p)$).

The advantage of working with quasi-coherent sheaves is that we can take adjoints 
and prove instead the equivalent statement 
$$
(F_n)_p \circ (v_p)_*  \simeq  (w_p)_* \circ (F_{n'})_p . 
$$ 
Note further that, since $(j_n)_*$ is faithful, it reflects isomorphisms.  This is true in any triangulated category: consider a morphism $f: F \to G $ and the distinguished triangle 
$$ (j_n)_* F \xrightarrow{(j_n)_* (f)} (j_n)_* G \to (j_n)_* (\mathrm{cone} f) .$$ Then $ (j_n)_*(f)$ is an isomorphism if and only if $(j_n)_* (\mathrm{cone} f )$ vanishes, and $(j_n)_* (\mathrm{cone} f)$ vanishes if and only if $\mathrm{cone} f$ vanishes.

Thus it is sufficient to show that there is a natural equivalence 
$$
(j_n)_* \circ (F_n)_p \circ (v_p)_*  \simeq (j_n)_* \circ (w_p)_* \circ (F_{n'})_p.
$$
This is a consequence of the following  three facts: 
\begin{itemize} 
\item By the commutativity of the two diagrams above, we have that 
$$
w_* \circ (j_{n'})_* \simeq (j_n)_* \circ (w_p)_* \quad  \text{ and } \quad  v_* \circ (i_{n'})_* \simeq (i_n)_* \circ (v_p)_*.
$$
\item {As explained in Remark \ref {rem:intert} we have an equivalence} 
$
F_n \circ (i_n)_* \simeq  (j_n)_*\circ  (F_n)_p
$ (and there is an analogous equivalence for the index $n'$). Recall for every $n$ we are denoting by $F_n$ the (Ind-completion of the) Fourier-Mukai functor $F_{n}\colon \Db{\cY_{k_n}}\to \Db{\cX_{k_n}}$.
\item By Proposition \ref{prop:FM.compatible} there is an equivalence:   
$
F_{n'} \circ v^*  \simeq w ^*\circ F_n. 
$
Taking adjoints we obtain an equivalence 
$
F_n \circ v_*  \simeq w_*\circ F_{n'}. 
$
\end{itemize}
Leveraging these three facts we obtain a chain of equivalences
\begin{multline*}
  (j_n)_* \circ (w_p)_* \circ (F_{n'})_p  \simeq w_* \circ (j_{n'})_* \circ (F_{n'})_p  \simeq 
w_* \circ F_{n'} \circ (i_{n'})_* \simeq  \\
 F_n \circ v_* \circ (i_{n'})_* \simeq F_{n} \circ  (i_n)_* \circ (v_p)_* \simeq (j_n)_* \circ (F_n)_p  \circ (v_p)_* 
\end{multline*}
and this concludes the proof. 
\end{proof}

By taking the colimit over the sequence $\{(k_n,\cT_{k_n})\}_{n\in \bN}$, we obtain an equivalence of dg-categories
$$
\varinjlim_n (F_{k_n})_p\colon \Db{(X^\val_\infty)_\xi}\to \Db{(\irs{X})_\xi}.
$$
This concludes the proof of Theorem \ref{thm:main1}.

\bibliographystyle{plain}
\bibliography{logmckay}

\end{document}